\documentclass[12pt]{amsart}
\usepackage[english]{babel}
\usepackage{amsfonts,graphicx,latexsym}
\usepackage{amsmath}
\usepackage{amsthm}
\usepackage{dsfont}
\usepackage{amssymb}
\usepackage{mathrsfs}
\usepackage{times}
\usepackage[margin=45pt]{geometry}

\usepackage[foot]{amsaddr}

\usepackage{hyperref}\hypersetup{
    colorlinks=true,
    linkcolor=purple,
    filecolor=blue,      
    urlcolor=cyan,
    }
\usepackage{comment}
\usepackage{xcolor}
\usepackage[most]{tcolorbox}
\usepackage{tikz}
\usepackage{lipsum}
\usetikzlibrary{arrows.meta,svg.path}
\definecolor{lime}{HTML}{A6CE39}
\DeclareRobustCommand{\orcidicon}{%
	\begin{tikzpicture}
	\draw[lime, fill=lime] (0,0) 
	circle [radius=0.16] 
	node[white] {{\fontfamily{qag}\selectfont \tiny ID}};
	\draw[white, fill=white] (-0.0625,0.095) 
	circle [radius=0.007];
	\end{tikzpicture}
	\hspace{-2mm}
}
\foreach \x in {A, ..., Z}{%
	\expandafter\xdef\csname orcid\x\endcsname{\noexpand\href{https://orcid.org/\csname orcidauthor\x\endcsname}{\noexpand\orcidicon}}
}

\newtheorem{theorem}{Theorem}[section]

\newtheorem{lemma}[theorem]{Lemma}

\newtheorem{proposition}[theorem]{Proposition}

\newtheorem{remark}[theorem]{Remark}

\tcolorboxenvironment{theorem}{
  enhanced,
  breakable,
  colback=gray!3,
  colframe=black!50!black,
  boxrule=0.8pt,
  arc=2mm,
  left=6pt,
  right=6pt,
  top=6pt,
  bottom=6pt,
  before skip=10pt,
  after skip=10pt
}

\tcolorboxenvironment{lemma}{
  enhanced,
  breakable,
  colback=gray!3,
  colframe=black!40!black,
  boxrule=0.8pt,
  arc=2mm,
  left=6pt,
  right=6pt,
  top=6pt,
  bottom=6pt,
  before skip=10pt,
  after skip=10pt
}

\tcolorboxenvironment{proposition}{
  enhanced,
  breakable,
  colback=gray!3,
  colframe=black!50!black,
  boxrule=0.8pt,
  arc=2mm,
  left=6pt,
  right=6pt,
  top=6pt,
  bottom=6pt,
  before skip=10pt,
  after skip=10pt
}

\def\P{{\mathbb P}} 
 
\def\E{{\mathbb E}} 
\def\Z{{\mathbb Z}}

\def \R {{\mathbb{R}}}
\def \0{{\mathbf 0}} 
\def \1{{\mathbf 1}} 
\def \X {{\mathbf{X} }}

\makeatletter
\newsavebox{\@brx}
\newcommand{\llangle}[1][]{\savebox{\@brx}{\(\m@th{#1\langle}\)}%
  \mathopen{\copy\@brx\mkern2mu\kern-0.9\wd\@brx\usebox{\@brx}}}
\newcommand{\rrangle}[1][]{\savebox{\@brx}{\(\m@th{#1\rangle}\)}%
  \mathclose{\copy\@brx\mkern2mu\kern-0.9\wd\@brx\usebox{\@brx}}}
\makeatother
\newcommand{\rmd}{\,\mathrm{d}}
\newcommand{\eps}{\varepsilon}

\def\1{{\mathchoice {1\mskip-4mu\mathrm l}      
{1\mskip-4mu\mathrm l} 
{1\mskip-4.5mu\mathrm l} {1\mskip-5mu\mathrm l}}}

\title[{Random walks on the 3D randomly oriented Manhattan lattice}]{Superdiffusivity of random walks on the three-dimensional randomly oriented Manhattan lattice}
\author[Tuan-Minh Nguyen]{Tuan-Minh Nguyen\orcidA{}
}
\address{School of Mathematics, Monash University, 9 Rainforest Walk, Clayton 3800, Victoria, Australia.}
\email{tuanminh.nguyen@monash.edu}

\begin{document}

\begin{abstract}
We study the superdiffusive behavior of random walks on the randomly oriented Manhattan lattice, i.e., the $d$-dimensional integer lattice $\mathbb{Z}^d$ where each axis-aligned line is independently assigned a random direction (forward or backward) with equal probability. The walker takes nearest-neighbor steps, choosing an axis randomly and moving along the assigned direction of that axis’s line, with equal probabilities for each axis. We show that, in the critical dimension $d=3$, the diffusion coefficient of the random walk diverges in the Tauberian sense as $\sqrt{\log t}$ with a multiplicative correction $(\log\log t)^{\pm(2+\eps)}$ as time $t\to\infty$. This gives an answer to a conjecture by Ledger, T\'oth and Valk\'o (2018).    

\end{abstract}
\keywords{random walks in random environment, superdiffusivity, diffusion coefficient}
\subjclass{60K50, 82C41, 82C27}

\maketitle
\tableofcontents

\section{Introduction}
Random walks in stationary doubly stochastic, or equivalently divergence-free, random environments form a natural class of non-reversible stochastic systems in which one may study the transition from diffusive to anomalous transport. This class of random walks was introduced by Kozlov in \cite{K1985}. For further background and a general treatment based on martingale approximation, see \cite[Section~3.2]{KLO2012}. Under uniform ellipticity of the symmetric part of the jump rates, Kozma and T\'oth \cite{KT2017} proved that the $\mathcal H_{-1}$ condition on the antisymmetric drift is sufficient for a central limit theorem. T\'oth \cite{Toth2024} subsequently gave an alternative functional-analytic proof that does not use the diagonal heat-kernel bound appearing in \cite{KT2017}, and \cite{Toth2026} extended the result to degenerate environments under substantially weaker ellipticity assumptions. For an overview of diffusive and superdiffusive limits for random walks in divergence-free random drift fields, see \cite{Toth2018}.

The random walk on the randomly oriented Manhattan lattice, first introduced in \cite{LTV2018,Toth2018}, is a typical discrete model in this class exhibiting anomalous behavior in lower dimensions. In this model, independently for every axis-aligned line of $\mathbb Z^d$, one of its two orientations is chosen with equal probability, and the walker jumps at rate one along the prescribed orientation in each coordinate direction. The random environment therefore has infinite-range correlations along every coordinate line. For this model, the $\mathcal H_{-1}$ condition holds in dimensions $d\ge4$, and therefore the walk is diffusive. This condition, however, fails for $d\in\{2,3\}$. In the critical dimension $d=3$, the corresponding regularized $\mathcal H_{-1}$ norm diverges logarithmically, and Ledger, T\'oth and Valk\'o \cite{LTV2018} conjectured that the diffusion coefficient of the model diverges as $\sqrt{\log t}$ as $t\to\infty$.

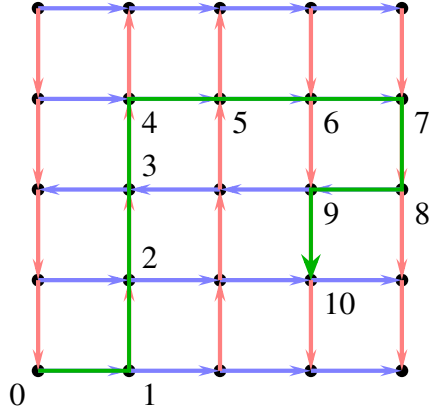
\begin{figure}[ht]
    \centering
\begin{tikzpicture}[scale=1.2, >={Stealth[length=8pt, width=4pt]}]
    \def\n{4} 
    \foreach \x in {0,...,\n} {
        \foreach \y in {0,...,\n} {
            \fill (\x, \y) circle (2pt); 
        }
    }
    \foreach \y in {0,...,\n} {
        \ifnum\y=0 
            \foreach \x in {0,...,\the\numexpr\n-1} {
                \draw[->, ultra thick, blue!50] (\x, \y) -- (\x+1, \y);
            }
        \else
            \ifnum\y=1 
                \foreach \x in {0,...,\the\numexpr\n-1} {
                    \draw[->, ultra thick, blue!50] (\x, \y) -- (\x+1, \y);
                }
            \else
                \ifnum\y=2 
                    \foreach \x in {1,...,\n} {
                        \draw[->, ultra thick, blue!50] (\x, \y) -- (\x-1, \y);
                    }
                \else
                    \ifnum\y=3 
                        \foreach \x in {0,...,\the\numexpr\n-1} {
                            \draw[->, ultra thick, blue!50] (\x, \y) -- (\x+1, \y);
                        }
                    \else
                        \ifnum\y=4 
                            \foreach \x in {0,...,\the\numexpr\n-1} {
                                \draw[->, ultra thick, blue!50] (\x, \y) -- (\x+1, \y);
                            }
                        \fi
                    \fi
                \fi
            \fi
        \fi
    }

    \foreach \x in {0,...,\n} {
        \ifnum\x=0 
            \foreach \y in {1,...,\n} {
                \draw[->, ultra thick, red!50] (\x, \y) -- (\x, \y-1);
            }
        \else
            \ifnum\x=1 
                \foreach \y in {0,...,\the\numexpr\n-1} {
                    \draw[->, ultra thick, red!50] (\x, \y) -- (\x, \y+1);
                }
            \else
                \ifnum\x=2 
                    \foreach \y in {0,...,\the\numexpr\n-1} {
                        \draw[->, ultra thick, red!50] (\x, \y) -- (\x, \y+1);
                    }
                \else
                    \ifnum\x=3 
                        \foreach \y in {1,...,\n} {
                            \draw[->, ultra thick, red!50] (\x, \y) -- (\x, \y-1);
                        }
                    \else
                        \ifnum\x=4 
                            \foreach \y in {1,...,\n} {
                                \draw[->, ultra thick, red!50] (\x, \y) -- (\x, \y-1);
                            }
                        \fi
                    \fi
                \fi
            \fi
        \fi
    }

    \draw[green!70!black, line width=1.5pt, -Stealth] 
        (0,0) node[below left, black] {0} 
        -- (1,0) node[below right, black] {1} 
        -- (1,1) node[above right, black] {2} 
        -- (1,2) node[above right, black] {3} 
        -- (1,3) node[below right, black] {4} 
        -- (2,3) node[below right, black] {5} 
        -- (3,3) node[below right, black] {6} 
        -- (4,3) node[below right, black] {7} 
        -- (4,2) node[below right, black] {8} 
-- (3,2) node[below right, black] {9} 
-- (3,1) node[below right, black] {10}; 
\end{tikzpicture}
\caption{A walk of 10 steps on the two-dimensional randomly oriented Manhattan lattice}
    \label{fig1}
\end{figure}

The model is formally described as follows.  Let $(e_i)_{i\in[d]}$ be the standard orthonormal basis of $\mathbb R^d$, where $[d]:=\{1,\ldots,d\}$. For each $i\in [d]$, let $U_i:=\{ y=(y_j)_{1\le j\le d}\in \Z^d$: $y_i=0\}$.
For each $i\in [d]$ and each point $x$ in the hyperplane $U_i$, let $L_{x,i}:=\{ y=x+te_i : t\in \Z \}$ be the grid line passing through $x$ and parallel to $e_i$. 
We attach a random direction $\omega(x,i) e_i$ to $L_{x,i}$, where $\omega(x,i)$ is a random variable taking values in $\{-1,1\}$. 
We call $\omega=(\omega(x,i))_{i\in [d], x\in U_i}$ a \textbf{randomly oriented Manhattan environment}. We assume throughout this paper that 
$$\P(\omega(x,i) = 1)=\P(\omega(x,i) = -1)=\frac{1}{2}$$ 
and $(\omega(x,i))_{i\in [d], x\in U_i}$ are independent. 
For any $x\in \Z^d$ and $i\in [d]$, set ${\omega(x,i)}=\omega(x-x_i e_i,i)$ which is associated to the chosen direction on $L_{x,i}$.

Let $\X=(X_t)_{t\ge0}$ with $X_t=(X_{i,t})_{i\in [d]}$ be the nearest-neighbor continuous-time random walk on $\Z^d$ in the randomly oriented Manhattan environment which is defined as follows. It starts from the origin $\0$, i.e., $X_0=\0$. Given the environment $\omega$, 
the process jumps from its current position $x$ to a neighbor along the direction $\omega(x,i)e_i$ with rate 1 for each $i\in [d]$.
Notice that $\X$ is a random walk in a doubly stochastic environment, i.e., its quenched law is Markovian and is given by 
$$P_{\omega}( X_{t+h}=x\pm e_i \ | X_t=x)= 
p_{\pm i}(\tau_x\omega)h+o(h)\quad  \text{for $i\in [d]$},$$
where $\tau_x$ is the translation of the environment by $x$ and  $p_{\pm i}(\omega):=(1\pm\omega(\0,i))/{2}.$

For $d\ge 4$, the process $\X$ satisfies the so-called $\mathcal{H}_{-1}$ condition and it is thus diffusive (see Section 1.4 in \cite{Toth2018}). This condition no longer holds for $d\in \{2, 3\}$. Using a non-rigorous
Alder-Wainwright type scaling argument (see \cite{AW1967}, \cite{FNS1977} and \cite{TV2012}), Ledger, T\'oth and Valk\'{o} \cite{LTV2018} conjectured that
$$\E[|X_t|^2] \asymp t^{4/3} \text{ when $d=2$ and }\E[|X_t|^2]\asymp t (\log t)^{1/2} \text{ when $d=3$}.$$
Let $D(\lambda)=\int_0^{\infty}\E[|X_t|^2]{\rm e}^{-\lambda t}\rmd t$. This is equivalent in the Tauberian sense to 
$$ D(\lambda)\asymp \lambda^{-7/3} \text{ for $d=2$}, \quad   D(\lambda)\asymp \lambda^{-2} \sqrt{|\log\lambda|} \text{ for $d=3$}$$
as $\lambda \to 0$.

Our main result confirms this conjecture in the critical dimension $d=3$, up to a  multiplicative correction $(\log|\log\lambda|)^{\pm (2+\eps)}$.

\begin{theorem}\label{thm.main} When $d=3$, for each $\eps>0$ there exists a constant $C_{\eps}\in (1,\infty)$ such that for each $\lambda\in (0,e^{-2})$, 
$$C_{\eps}^{-1}( \log|\log\lambda|)^{-2-\eps} \le \frac{\lambda^2D(\lambda)}{\sqrt{|\log\lambda|}}\le C_{\eps}( \log|\log\lambda|)^{2+\eps}.$$
\end{theorem}

\begin{remark}It is worth mentioning that Theorem \ref{thm.main} significantly strengthens a previous result obtained in \cite{LTV2018} for $d=3$:
$$C^{-1}\lambda^{-2}\log|\log\lambda|  \le D(\lambda)\le C \lambda^{-2}|\log\lambda|.$$
By a well-known argument
\cite{QV2008}, one can obtain from the upper bound in
Theorem~\ref{thm.main} the following upper bound for the diffusivity in
real time:
\[
\mathbb E[|X_t|^2]
\le
O\left(
t\sqrt{\log t}\,(\log\log t)^{2+\varepsilon}
\right).
\]
Moreover, applying for instance \cite[Theorem~1.7.1]{BGT1987}, the lower bound in Theorem~\ref{thm.main} implies that
\[
\limsup_{t\to\infty}
\frac{\mathbb E[|X_t|^2]}
{t\sqrt{\log t}\,(\log\log t)^{-2-\varepsilon}}
>0 .
\]
\end{remark}

The phenomenon studied in this paper belongs to a broader class of models exhibiting logarithmic superdiffusivity at criticality. Starting from the numerical observations and physical arguments of Alder and Wainwright~\cite{AW1967,AW1970}, logarithmically
superdiffusive behavior has appeared in several models of
non-reversible random motion and interacting particle systems. In particular, Alder and Wainwright identified a non-integrable long-time tail of the velocity autocorrelation in dimension two, which leads to a divergent diffusion coefficient. Forster, Nelson and Stephen \cite{FNS1977} subsequently developed a non-rigorous renormalization group and scaling analysis of this long-time tail phenomenon, providing a systematic description of the logarithmic corrections arising at criticality. A major breakthrough in this area was achieved by H.-T.~Yau~\cite{Y2004}, who proved that, in the Tauberian sense, the diffusion coefficient of the two-dimensional asymmetric exclusion process diverges as $(\log t)^{2/3}$, up to a multiplicative correction
$\exp\big(\pm\gamma(\log\log\log t)^2\big)$
for some $\gamma>0$. For Brownian particles in divergence-free random environments, T\'oth and Valk\'o~\cite{TV2012} predicted $\sqrt{\log t}$-superdiffusivity for diffusion in the curl of the two-dimensional Gaussian free field, and this prediction was proved up to a multiplicative factor $(\log\log t)^{\pm(1+\eps)}$ in the Tauberian sense by Cannizzaro, Haunschmid-Sibitz and Toninelli~\cite{CHT2022}. Using stochastic homogenization, Chatzigeorgiou, Morfe, Otto and Wang~\cite{CMOW2025} proved the optimal $t\sqrt{\log t}$ growth of the mean squared displacement for this model, while Armstrong, Bou-Rabee and Kuusi~\cite{ABK2024} subsequently established a quenched superdiffusive invariance principle, for a broader class of critically correlated incompressible random drifts that includes the curl of the two-dimensional Gaussian free field. Closely related logarithmic superdiffusive behavior has also been studied for the anisotropic KPZ equation~\cite{CET23,CET2023}, self-repelling Brownian polymers~\cite{CG2025}, asymmetric simple exclusion processes~\cite{LQSY2004,Y2004,QV2007}, and the stochastic Burgers equation at critical dimension~\cite{DH2024,CMT2024}.

Random walk on the randomly oriented Manhattan lattice is closely related to the Matheron--de Marsily model, introduced in \cite{MdM1980}. In the Matheron--de Marsily model, one coordinate direction is randomly oriented as in the randomly oriented Manhattan lattice, while the remaining coordinate lines are undirected. At each jump time, the walker chooses uniformly one of the $d$ coordinate lines through its current position. It follows the prescribed orientation when the chosen line is directed, and chooses either neighboring vertex with probability $1/2$ when the chosen line is undirected. In dimension two, this model is closely related to random walk in random scenery. Its transience and functional scaling limit were established in \cite{GPLN2007, GPLN2008}; see also the survey \cite{P2020}. Ledger, T\'oth and Valk\'o \cite{LTV2018} also considered the three-dimensional intermediate model with two directed coordinate directions and one undirected coordinate direction, and obtained superdiffusive upper and lower bounds for its mean squared displacement. More recent work has considered several related models. Collevecchio, Hamza and Tournier \cite{CHT2019} studied a deterministic walk in the same randomly oriented Manhattan environment. They proved that almost surely, the trajectory is eventually periodic and hence bounded, whereas in dimension two, it almost surely localizes on two vertices. Guillotin-Plantard, P\`ene and Watbled \cite{GPW2026} established a limit theorem for a three-dimensional randomly oriented model extending the Matheron--de Marsily model, whose coordinates are described by iterated random walks in random scenery. They also observed that the mutual dependence between the coordinate motions makes the randomly oriented Manhattan lattice considerably more difficult, and that the bounds in \cite{LTV2018} were the only quantitative estimates previously available for this model.

The remainder of the paper is organized as follows. Following Ledger, T\'oth and Valk\'o~\cite{LTV2018}, in Section~\ref{sec:res.method} we introduce the environmental process seen from the walker and express the Laplace transform of the mean squared displacement in terms of the resolvent of its generator $G$. The proof of Theorem~\ref{thm.main} therefore reduces to estimating a quadratic form involving $(\lambda-G)^{-1}$. In Section~\ref{sec:chao.expansion}, we establish a Rademacher chaos expansion to identify the space of square-integrable functions of the random environment with the graded bosonic Fock space
$\bigoplus_{n=0}^{\infty}\ell^2_0(U,d,n)$, where $U:=\{ (x^1,x^2,\cdots, x^d)\in \mathbb{Z}^d: \sum_{j=1}^d x^j=0\}$ and $\ell^2_0(U,d,n)$ denotes the space of square-summable symmetric functions on $(U\times [d])^n$ that vanish on diagonals.
Under this identification, the symmetric part $S$ of the generator $G$ preserves the chaos degree, while the asymmetric part is decomposed as $A=A_+-A_-$, with $A_+$ and $A_-$ increasing and decreasing the degree by one, respectively. As in \cite{LQSY2004}, the resolvent of the truncated operator $G_m$ defined on  the first $m$ chaos provides alternating lower and upper bounds for the resolvent of $G$. Moreover, the quadratic form corresponding to the truncated resolvent $(\lambda-G_m)^{-1}$ is reduced to the quadratic form corresponding to the resolvent of $S-\Lambda_m$ defined on the first chaos, where the sequence of operators $(\Lambda_m)_{m\ge 0}$ satisfies the recursion $\Lambda_{m+1}=A_{-} ( \lambda-S+{\Lambda}_{m})^{-1} {A}_{+}$ for $m\ge1$, and $\Lambda_1=0$. The diagonal-vanishing condition defining the above-mentioned bosonic Fock space, also known as the hard-core constraint, obstructs the direct application of Fourier methods. In Section~\ref{sec:hard-core.removal}, for a positive self-adjoint diagonal operator $R$ and $f\in \ell^2_0(U,3,n)$, we estimate the quadratic form $\langle f,A_-RA_+f\rangle$ by comparing $A_+$ with its unconstrained analogue $\widetilde A_+$, controlling the error produced by removing the hard-core constraint, and estimating the diagonal and off-diagonal parts of the Fourier representation of $\langle f,(\widetilde A_+)^*R\widetilde A_+f\rangle$. In Section~\ref{sec:bounds.msd}, motivated in part by the recursive comparison scheme of \cite{Y2004} and \cite{CHT2022}, we introduce diagonal comparison operators $(T_m)_{m\ge 1}$ together with coefficients $(c_m)_{m\ge 1}$ and combine the estimates from Section~\ref{sec:hard-core.removal} to
establish the recursive bounds $\Lambda_{2m}\le c_{2m}T_{2m}$ and $\Lambda_{2m-1}\ge c_{2m-1}T_{2m-1}$. Finally, we substitute these bounds into the truncated resolvent representation and estimate the resulting first-chaos Fourier integrals to obtain the upper and lower bounds in Theorem~\ref{thm.main}.

Although the present work is motivated by the same general resolvent method framework, the proof is substantially different from the Gaussian-field arguments used in~\cite{CHT2022}. The randomly oriented Manhattan lattice leads to a discrete, non-Gaussian environment. We use a Rademacher chaos expansion rather than a Wiener chaos expansion, and the exclusion of repeated environment variables produces a hard-core constraint on the corresponding Fock space. A substantial part of the proof is therefore devoted to comparing the operator $A_+$ with its unconstrained analogue $\widetilde A_+$ and
controlling the resulting error by the Dirichlet form associated with the symmetric part of the generator. After this hard-core removal, the main estimates are carried out through discrete Fourier analysis on products of lower-dimensional tori. This creates additional difficulties that do not appear in the continuum Gaussian setting, as the relevant singular integrals must be controlled uniformly on the torus and with the correct dependence on the chaos degree. In this respect, the proof is closer in spirit to the recursive resolvent and Fourier-analytic methods used for asymmetric exclusion processes in the works of Landim--Quastel--Salmhofer--Yau~\cite{LQSY2004}, Yau~\cite{Y2004}, and Quastel--Valk\'o~\cite{QV2007}, while the Rademacher chaos and hard-core constraint are specific to the randomly oriented Manhattan lattice model. It is worth noting that the recent sharp results on diffusion in the curl of the two-dimensional Gaussian free field~\cite{CMOW2025,ABK2024} and on the critical stochastic Burgers equation~\cite{CMT2024} are obtained through different approaches. The former are based on quantitative stochastic homogenization and and renormalization arguments, while the latter relies on refined multiscale estimates for the resolvent of the generator. These approaches exploit structural properties that are not generally available for non-reversible lattice models such as asymmetric exclusion processes or random walks in discrete random environments. 

\subsection*{AI disclosure statement}
The author used Grok to generate the TikZ code for illustrating Figure \ref{fig1} and Microsoft 365 Copilot for language editing/rephrasing in the paper, as well as for numerical and symbolic checking of the integral estimates in Section \ref{subsec:integral.estimates} (this aided the author in refining technical bounds in Lemma \ref{lem:ULB} and Lemma \ref{lem:diagonal.integrals}). The main ideas and the mathematical arguments underlying the paper were developed by the author and have been discussed with colleagues since late 2024.

\section{Resolvent method}\label{sec:res.method}

We study the asymptotic behavior of the mean squared displacement $\E[|X_t|^2]$ using the resolvent method. This approach was first developed in \cite{KO2002, LQSY2004, Y2004} to estimate diffusion coefficients for tracer particles in a Gaussian drift field and for asymmetric exclusion processes on $\Z^d$ in dimensions $d\in\{1,2\}$. Subsequently, it was applied in \cite{TTV2012, TV2012} to investigate the asymptotic behavior of self-repelling diffusions driven by the negative gradient of their local time and diffusions in the curl of the two-dimensional Gaussian free field. In this section, for the sake of completeness, we recall the resolvent framework for the environmental process viewed from the walker, as used in \cite{LTV2018}. This framework reduces the analysis of the mean squared displacement to estimates for a quadratic form involving the resolvent of the generator of the environmental process.

Let $\mathcal{E}=\bigotimes_{i\in [d]}\bigotimes_{x\in U_i}\{-1,1\}$ be the set of all possible Manhattan environments. Let $\eta_t\in \mathcal{E}$ be the environment seen from the current position of the walk at time $t$, i.e.,
$$\eta_t(x,i)=\eta_0(x+X_t,i)\quad\text{for every } i\in [d],\, x\in \Z^d,$$
where we set ${\eta_t(x,i)}:=\eta_t(x-x_i e_i,i)$ for $i\in [d], x\in \Z^d$. 
Note that the environmental process $(\eta_t)_{t\in \R_+}$ is a continuous-time Markov chain on $\mathcal{E}$.

Let $\tau_{i} : \mathcal{E}\to \mathcal{E}$ be the translation of the environment by $e_{i}$ and let $\tau_{i}^{-1}$ be its inverse. In coordinates, these operators are defined by
$$
(\tau_{i} \omega)( x, i )=(\tau_{i}^{-1} \omega)( x,i )=\omega( x, i ) \quad\mathrm{and} 
$$
$$
(\tau_{i} \omega)(x,j)=\omega( x+e_{i}, j ),  \quad(\tau_{i}^{-1} \omega)( x,j  )=\omega(  x-e_{i}, j ) \quad\text{for } \ j \neq i. 
$$
Recall that the initial state $\eta_0$ has distribution
$$\pi=\bigotimes_{i\in [d]}\bigotimes_{x\in U_i}\mu_{x,i}$$
where $\mu_{x,i}=\frac{1}{2}(\delta_{-1}+\delta_{1})$ is the Rademacher probability measure on $\{-1,1\}$. Notice that $\pi$ is invariant under the translations $\tau_i$ and $\tau_i^{-1}.$
For two functions $f,g :\mathcal{E}\to \R $, define  the scalar product
$$\langle f,g \rangle_{L^2(\mathcal{E},\pi)} = \int_{\mathcal{E}} f(\omega)g(\omega)\rmd \pi(\omega).$$

Let $G$ be the infinitesimal generator of $(\eta_t)$. Notice that $G$ is defined for each function $f: \mathcal{E}\to \R$ by
\begin{align*}
(G f)(\omega) 
&=\sum_{i=1}^{d} \left( \frac{1+\omega({\mathbf{0},i})} {2} (f ( \tau_{i} \omega)-f ( \omega))+\frac{1-\omega({\mathbf{0},i}) } {2} (f ( \tau_{i}^{-1}\omega)-f ( \omega)) \right).
\end{align*}
Notice that $G=S+A$, where
$$
S  f ( \omega)=\frac{1} {2} \sum_{i=1}^{d} ( f ( \tau_{i} \omega)+f ( \tau_{i}^{-1} \omega)-2 f ( \omega) ), \quad A f ( \omega)=\frac{1}{2}\sum_{i=1}^{d}\omega(\0, i )  ( f ( \tau_{i} \omega)-f ( \tau_{i}^{-1} \omega) ). 
$$
Note that $S$ is the infinitesimal generator of the environmental process  seen from a symmetric
simple random walk on $\Z^d$.

Notice that $(X_{1,t}, \eta_t)_{t\ge0}$ is also a Markov chain on $\Z \times \mathcal{E}$ with the generator
$$
(\widetilde{G}_{1} f) ( z, \omega)=\sum_{i=1}^d\frac{1+\omega(\0, i )} {2} f ( z+\mathbf{1}_{\{i=1\}}, \tau_{i} \omega)+\sum_{i=1}^d\frac{1-\omega( \0,i )} {2} f ( z-\mathbf{1}_{\{i=1\}}, \tau_{i}^{-1} \omega)-d f ( z, \omega). 
$$ 
For $f(z,\omega)=z$, we have $(\widetilde G_1 f)(z,\omega)=\omega(\0,1)$.

Define the function $\phi:\mathcal{E}\to \R$ by $\phi(\omega)=\omega(\0,1)$.
We notice that
    $$M_t:=X_{1,t}-\int_{0}^{t} \phi(\eta_s)\rmd s$$
is a martingale. In particular, it follows that 
$$\E[|X_t|^2]=d\E[|X_{1,t}|^2]=d\left( \E[M_t^2]+2\int_0^t\E[M_s\phi(\eta_s)]\rmd s+ \E\left[ \Big(\int_{0}^{t} \phi(\eta_s)\rmd s \Big)^2\right]\right).$$
Note that $\E[M_t^2]=t$ and $\E[M_s\phi(\eta_s)]=0$. Since $\eta_0\sim\pi$ and $\pi$ is invariant for $(\eta_t)_{t\ge 0}$, we have
\begin{align*}
    E_G(t):&=\E\left[ \Big(\int_{0}^{t} \phi(\eta_s)\rmd s \Big)^2\right]= 2\int_0^t (t-s)\E[\phi(\eta_s)\phi(\eta_0)]\rmd s=2\int_0^t(t-s)\langle e^{sG}\phi,\phi\rangle_{L^2(\mathcal{E},\pi)}\rmd s.\end{align*}
Denote by $D_G(\lambda):=\int_0^{\infty} {\rm e}^{-\lambda t} E_G(t)\rmd t$ the Laplace transform of $E_G$. We obtain
$$D_G(\lambda)=2\lambda^{-2}\int_0^\infty e^{-\lambda s} \langle e^{sG}\phi,\phi\rangle_{L^2(\mathcal{E},\pi)}\rmd s =2\lambda^{-2} \langle\phi,(\lambda-G)^{-1}\phi\rangle_{L^2(\mathcal{E},\pi)}.$$
Consequently, to prove Theorem~\ref{thm.main}, it suffices to establish the corresponding upper and lower bounds for
$\langle\phi,(\lambda-G)^{-1}\phi\rangle_{L^2(\mathcal E,\pi)}$.

\section{Rademacher chaos and truncated Helmholtz equations}  \label{sec:chao.expansion}
In this section, we represent the space $L^2(\mathcal E,\pi)$ in terms of Rademacher chaos and identify it with a graded bosonic Fock space. Under this identification, the symmetric part $S$ of the generator $G$ preserves the chaos degree, whereas its antisymmetric part $A$ decomposes into the operators $A_+$ and $A_-$, which increase and decrease the chaos degree by one, respectively. We then restrict the generator $G$ to the first $m$ chaos sectors and study the corresponding truncated Helmholtz equations. By solving these equations recursively from the highest chaos sector to the first chaos sector, we express the first chaos component of each truncated resolvent in terms of a recursively defined family of self-adjoint operators. Finally, in Lemma~\ref{lemma.monotonicity}, we prove that even truncation levels yield lower bounds for the quadratic form $\langle\phi,(\lambda-G)^{-1}\phi\rangle_{L^2(\mathcal E,\pi)}$, whereas odd truncation levels yield upper bounds.

Throughout the paper, for any positive integer $n$ and for any sequence $(a_k)_{k\geq 1}$, we write $a_{1:n}:=(a_1,\ldots,a_n)$. For $1\le r\le n$, we also write
$
a_{1:n\setminus\{r\}}
:=
(a_1,\ldots,a_{r-1},a_{r+1},\ldots,a_n).
$

\subsection{Rademacher chaos expansion}
Consider the hyperplane 
$${U}:=\Big\{x=(x^1,\cdots, x^d)\in \Z^d: \sum_{j=1}^{d}x^j=0\Big\}.$$
Notice that the orientations in $U$ give all information for the orientations on $\Z^d$ as $U$ intersects every line $L_{x,i}=\{y=x+te_i: t\in \Z\}$ for $x\in U_i$ and $i\in [d]$.
Hence each environment $\omega\in \mathcal{E}$ is uniquely defined by $(\omega(x,i))_{x\in U, i\in [d]}\in \{-1,1\}^{U\times [d]}$. For simplicity, from now on, we set $$\mathcal{E}= \{-1,1\}^{U\times [d]} \quad \text{and} \quad \pi=\bigotimes_{ x\in U, i\in [d]} \mu_{x,i}$$
where $\mu_{x,i}$ is the uniform probability measure on $\{-1,1\}$, i.e., $\pi$ is generated by a $d$-dimensional i.i.d. Rademacher field on the hyperplane $U$. 
Without loss of generality, we consider the probability space $(\Omega, \mathcal{F}, \P)$ where $\Omega=\mathcal{E}$, $\P=\pi$ and $\mathcal{F}$ is the $\sigma$-field generated by finite cylinders.

Denote by $\ell^{2}(U, d,n)$ the set of all functions $u: (U\times [d])^n\to\R$ which are \textit{square-summable}, i.e., $$\sum_{x_{1:n}\in U^n, j_{1:n}\in [d]^n}u^2_{j_{1:n}}(x_{1:n})<\infty.$$
For $f, g\in \ell^2(U,d,n)$,  we define the scalar product
$$\langle f, g \rangle_{\ell^2(U,d,n)}:=\sum_{j\in [d]^n, x\in U^n} f_j(x)g_j(x).$$

Let $\ell^{2}_{\rm sym}(U, d,n)\subset \ell^{2}(U, d,n)$ be the set of square-summable functions which are  \textit{symmetric} in the sense that $$u_{j_1,\cdots, j_n}(x_1,\cdots, x_n)=u_{j_{\sigma(1)},\cdots, j_{\sigma(n)}}(x_{\sigma(1)},\cdots, x_{\sigma(n)})$$ for any permutation $\sigma$ and $(x_{1},\cdots, x_n)\in U^n$, $(j_{1},\cdots, j_n)\in  [d]^n$. 

Let $\ell^{2}_{0}(U, d,n)\subset \ell^{2}_{\rm sym}(U, d,n)$ be the set of symmetric square-summable functions that \textit{vanish on  diagonals}, i.e., $$u_{j_1,\cdots, j_n}(x_1,\cdots, x_n)=0 \quad \text{if $x_l=x_k$ and $j_l=j_k$ for some $k\neq l$.}$$
 
For each environment $\omega\in \mathcal{E}$, define the discrete multiple \textit{Walsh integral} of $f\in \ell_0^2(U,d,n)$, which is a random variable 
$I_n^{(f)}$ given by
$$I_n^{(f)}(\omega)=\sum_{x_{1:n}\in U^n; j_{1:n}\in[d]^n} f_{j_{1:n}}(x_{1:n})\prod_{\ell=1}^n\omega(x_{\ell},j_{\ell} ),$$
in which the above formal sum is defined as the limit of 
$I_n^{(f^{(m)})}$ as $m\to\infty$ where $$f^{(m)}_{j_{1:n}}(x_{1:n})=f_{j_{1:n}}(x_{1:n})\1_{\{|x_i|\le m,\, i\in [n]\}}.$$ Note that the sequence $(I_n^{(f^{(m)})})_{m\ge 1}$ is well defined and it converges  $\P$-almost surely and in $L^4(\P)$ to $I_n^{(f)}$ (see Lemma 2.1. in \cite{DK19}).

\begin{proposition}\label{isometry} For $f\in {\ell_0^2(U,d,n)}$ and $g\in {\ell_0^2(U,d,m)}$, we have
    $\E[I_n^{(f)}]=0$ and $$\E[I_n^{(f)}I_m^{(g)}]=\langle I_n^{(f)}, I_m^{(g)} \rangle_{L^2(\mathcal{E},\pi)}=\delta_{m,n}n!\langle f, g\rangle_{\ell^2(U,d,n)},$$ where $\delta_{m,n}$ denotes the Kronecker delta symbol.
\end{proposition}
\begin{proof} It is sufficient to prove the result when $f$ and $g$ are finitely supported. It is easy to see that $\E[I_n^{(f)}]=0$ and $\E[I_n^{(f)}I_m^{(g)}]=0$ for $m\neq n$.
    When $m=n$, we have
\begin{align*} 
    \mathbb{E} [ I_{n}^{( f )} I_{n}^{( g )} ]& =\mathbb{E}\Bigg[ \Bigg( \sum_{x_{1:n}\in U^n , j_{1:n}\in [d]^n} f_{j_{1:n}} ( x_{1:n} ) \prod_{\ell=1}^{n} \omega( x_{\ell}, j_{\ell} ) \Bigg) \Bigg( \sum_{y_{1:n}\in U^n, k_{1:n}\in [d]^n} g_{k_{1:n}} ( y_{1:n} ) \prod_{r=1}^{n} \omega( y_{r}, k_{r} ) \Bigg) \Bigg]
\\& =\sum_{x_{1:n}\in U^n , j_{1:n}\in [d]^n} \sum_{y_{1:n}\in U^n, k_{1:n}\in [d]^n} f_{j_{1:n}} ( x_{1:n} ) g_{k_{1:n}} ( y_{1:n} ) \mathbb{E} \Big[ \prod_{\ell=1}^{n} \omega( x_{\ell}, j_{\ell} ) \prod_{r=1}^{n} \omega( y_{r}, k_{r} ) \Big]. 
\end{align*}
 Note that \begin{align*}
     \mathbb{E} \Big[ \prod_{\ell=1}^{n} \omega( x_{\ell}, j_{\ell} ) \prod_{r=1}^{n} \omega( y_{r}, k_{r} ) \Big]=\left\{\begin{matrix} 1 & \text{ if $(x_{\ell},j_{\ell})_{1\le \ell \le n}$ is a permutation of $(y_{r},k_{r})_{ 1\le r \le n}$}\\ 0 & \text{otherwise}.\end{matrix}\right.
     \end{align*}
As $f$ and $g$ are symmetric on $U\times [d]$, we obtain $$\E[I_n^{(f)}I_n^{(g)}]=n!\sum_{x\in U^n,j\in [d]^n} f_j(x)g_j(x) =n!\langle f, g\rangle_{\ell^2(U,d,n)}.$$
\end{proof}

Let $\mathcal{H}_0$ be the set containing constant random variables and $\mathcal{H}_n$ be the closure of the span
of $I_n^{(f)}$ for each $f\in \ell^{2}_{0}(U,d,n)$. The linear subspace $\mathcal{H}_n\subset L^2(\mathcal{E},\pi)$ is called the \textit{Walsh chaos} or
\textit{Rademacher chaos} of degree $n$ (see e.g. Section 4.4 in \cite{LT1991} or \cite{NPR2010}). Notice that $L^2(\mathcal{E},\pi)=L^2(\P)$ can be orthogonally decomposed as
$$L^2(\mathcal{E},\pi)= \bigoplus_{n\ge 0} \mathcal{H}_n.$$
More specifically, for each $\psi\in L^2(\mathcal{E},\pi)$, there exists a sequence of functions $(f^{(n)})_{n\ge 1}\subset \ell_0^2(U,d,n)$ such that $\psi$ has the following \textit{Rademacher chaos expansion}:
\begin{align}\label{chaos.decomp}
\psi=\E[\psi]+\sum_{n=1}^{\infty}I_n^{(f^{(n)})}.
\end{align}
In particular, if $(x_1,j_1),\ldots,(x_n,j_n)$ are pairwise distinct, then
$$
f^{(n)}_{j_1,\cdots,j_n}(x_1,\cdots,x_n) =\frac1{n!}\mathbb E\left[ \psi(\omega) \prod_{\ell=1}^n\omega(x_\ell,j_\ell)\right].
$$

\subsection{Decomposition of the generator}

\subsubsection{Diffusion part}
Recall that for $i, j\in [d]$ and $x\in U$, we have
$$\tau_i \omega(x, j)=\omega(x+e_i-e_j,j)\text{ and } \tau_i^{-1} \omega(x, j)=\omega(x-e_i+e_j,j).$$ 
By a slight abuse of notation, we also define the operator $\tau_i, \tau_{-i} : \mathcal{H}_n\to \mathcal{H}_n$
by $$\tau_i\psi(\omega)= \psi(\tau_i\omega)\quad {\rm and}\quad \tau_{-i}\psi(\omega)=\psi(\tau_{i}^{-1}\omega).$$ 

Let $\psi(\omega)=\prod_{\ell=1}^n\omega( x_{\ell},j_{\ell})\in \mathcal{H}_n$, where $(x_1,x_2,\cdots, x_n)\in U^n$ and $(j_1,j_2,\cdots, j_n)\in [d]^n$ such that $(x_k,j_k)\neq (x_{h}, j_h)$ for $k\neq h$. We have 
\begin{align}\label{def.tau.op}
    \tau_{\pm i}\psi(\omega)=\prod_{\ell=1}^n \omega(x_{\ell}\pm e_i,j_{\ell}) = \prod_{\ell=1}^n \omega(x_\ell \pm( e_i - e_{j_{\ell}}), j_{\ell}),
    \end{align}
We note that $x_{\ell}\pm e_i\notin U$ while $x_\ell \pm( e_i - e_{j_{\ell}})\in U$ and $\omega(x_{\ell}\pm e_i,j_{\ell})=\omega(x_\ell \pm( e_i - e_{j_{\ell}}), j_{\ell})$ for each $\ell\in [n]$. Furthermore, $(x_\ell \pm( e_i - e_{j_{\ell}}),j_{\ell})\neq (x_h \pm( e_i - e_{j_{h}}),j_{h})$ for $\ell\neq h$. 

Notice that $\tau_{-i}$ is the adjoint of $\tau_{i}$. 
In fact, it is sufficient to verify the claim for monomials $\psi(\omega)=\prod_{\ell=1}^n\omega( x_{\ell},j_{\ell})\in \mathcal{H}_n$ and $\varphi(\omega)=\prod_{\ell=1}^n\omega( y_{\ell},k_{\ell})\in \mathcal{H}_n$. We have
\begin{align*}
    \langle \tau_i\psi, \varphi  \rangle_{L^2(\mathcal{E},\pi)} & = \E \Big[  \prod_{\ell=1}^n\omega( x_{\ell}+e_i-e_{j_{\ell}},j_{\ell})   \prod_{\ell=1}^n\omega( y_{\ell},k_{\ell}) \Big] \\
    & = \delta_{\{(x_{\ell}+e_{i}-e_{j_{\ell}},j_{\ell}),\ \ell\in [n] \},\{ (y_{\ell}, k_{\ell}), \ \ell\in [n]\} } = \delta_{\{(x_{\ell},j_{\ell}),\ \ell\in [n] \},\{ (y_{\ell}-e_{i}+e_{k_{\ell}}, k_{\ell}), \ \ell\in [n]\} }
 \\
 & =  \E \Big[\prod_{\ell=1}^n\omega( x_{\ell},j_{\ell})   \prod_{\ell=1}^n\omega( y_{\ell}-e_i+e_{k_{\ell}},k_{\ell}) \Big] =\langle \psi, \tau_{-i}\varphi  \rangle_{L^2(\mathcal{E},\pi)}.
\end{align*}
Define the operator $\nabla_i$  such that for each $\psi: \mathcal{E}\to\R$,  
$$\nabla_i \psi(\omega)=\tau_i\psi(\omega)-\psi(\omega).$$
Recall that the symmetric part of the generator is given by $$S=\frac{1}{2}(G+G^*)=-\frac{1}{2}\sum_{i=1}^d \nabla_i\nabla_{-i}=\frac{1}{2}\sum_{i=1}^d(\nabla_i+\nabla_{-i}).$$ 
Hence $S$ is a self-adjoint operator mapping $\mathcal{H}_n$ to itself. In particular, for $\psi(\omega)=\prod_{\ell=1}^n\omega( x_{\ell},j_{\ell})\in \mathcal{H}_n$, we have
\begin{align}
    S\psi(\omega)=\frac{1}{2}\sum_{i=1}^d\left( \prod_{\ell=1}^n \omega({x}_{\ell}+e_{i}-e_{j_{\ell}}, j_{\ell}) + \prod_{\ell=1}^n \omega({x}_{\ell}-e_{i}+e_{j_{\ell}}, j_{\ell})-2\prod_{\ell=1}^n\omega( x_{\ell},j_{\ell})\right).
\end{align}


\subsubsection{Decomposition of the asymmetric part}

In this section, we explain how the following decomposition of $A$ naturally arises. Recall that the asymmetric part of $G$ is given by $A=\frac{1}{2}(G-G^*)$. For $\psi:\mathcal{E}\to \R$, we have
$$A \psi ( \omega)=\frac{1}{2}\sum_{i=1}^{d}\omega(\0, i )  ( \psi ( \tau_{i} \omega)-\psi ( \tau_{i}^{-1} \omega) ). 
$$
We notice that, when multiplying a Rademacher random variable (which takes values in $\{-1,1\}$) by a product of i.i.d. Rademacher random variables, the product either gains a new term or loses one (the latter case occurs when there is a duplication). This hints a decomposition  of $A$ as below.

For each $i\in[d]$, let $a_i^{\dagger}: \mathcal{H}_n\to \mathcal{H}_{n+1}$ and $a_{i}: \mathcal{H}_n\to \mathcal{H}_{n-1}$ be respectively the \textit{creation operator} and the \textit{annihilation operator}, which are defined as follows. For each $(x_1,x_2,\cdots, x_n)\in U^n$ and $(j_1,j_2,\cdots, j_n)\in [d]^n$ such that $(x_k,j_k)\neq (x_{h}, j_h)$ for  $k\neq h$, let
\begin{align}
    a_{i}^{\dagger} & : \prod_{\ell=1}^n \omega( x_{\ell},j_{\ell} )  \!  \mapsto  \prod_{\ell=1}^n\big(1-\delta_{(x_{\ell}, j_{\ell}), (\0,i) }\big) \cdot  \omega( \0 , i )  \prod_{\ell=1}^n \omega( x_{\ell},j_{\ell} ),\\
    a_i & : \prod_{\ell=1}^n \omega( x_{\ell},j_{\ell} ) \!  \mapsto \sum_{r=1}^{n} \delta_{ (x_r,j_r), (\0,i)} \prod_{\ell\in [n]\setminus\{r\}} \omega( x_{\ell},j_{\ell} ).
\end{align}
We define the operators $A_+:\mathcal{H}_n\to \mathcal{H}_{n+1}$ and $A_{-}: \mathcal{H}_{n}\to \mathcal{H}_{n-1}$ such that
$$A_+ :=\frac{1}{2}\sum_{i\in [d]} a_i^{\dagger} (\nabla_{i}-\nabla_{-i})=\frac{1}{2}\sum_{i\in [d]} a_i^{\dagger} (\tau_{i}-\tau_{-i}), \quad A_{-} :=\frac{1}{2}\sum_{i\in [d]}  (\nabla_{-i}-\nabla_{i})a_{i}=\frac{1}{2}\sum_{i\in [d]}  (\tau_{-i}-\tau_{i})a_{i}.
$$ 
Let $\psi(\omega)=\prod_{\ell=1}^n\omega( x_{\ell},j_{\ell})\in \mathcal{H}_n$, where $(x_1,x_2,\cdots, x_n)\in U^n$ and $(j_1,j_2,\cdots, j_n)\in [d]^n$ such that $(x_k,j_k)\neq (x_{h}, j_h)$ for  $k\neq h$. 
    Then 
    \begin{align}
\nonumber A_+\psi(\omega)&=
\frac{1}{2}\sum_{i=1}^{d} \prod_{\ell\in[n]} (1-\delta_{(x_{\ell},\; j_{\ell}),(\0,\; i)}) \omega( \mathbf{0}, i ) \Big[ \prod_{\ell\in[n]} \omega( {x}_{\ell}+e_i-e_{j_{\ell}}, j_{\ell} )-\prod_{\ell\in[n]} \omega( {x}_{\ell}-e_i+e_{j_{\ell}}, j_{\ell} ) \Big],\\ 
  \nonumber      A_{-}\psi(\omega)&=  \frac{1}{2} \sum_{i\in [d]} \sum_{r\in[n]}\delta_{\left({x}_{r},\; j_{r}\right), (\0,\; i) } \Big[ \prod_{\ell\in [n]\setminus \{r\} } \omega({x}_{\ell}-e_i+e_{j_{\ell}},j_{\ell})- \prod_{\ell\in [n]\setminus \{r\} } \omega({x}_{\ell}+e_i-e_{j_{\ell}},j_{\ell})\Big].
    \end{align}
We notice that
    $$A=A_+-A_{-}.$$
    
Also note that $a_i^{\dagger}$ is the adjoint of $a_i$, yielding that $A_{-}$ is the adjoint of $A_+$. In fact, it is sufficient to prove the claim for monomials $\psi(\omega)=\prod_{\ell=1}^n\omega( x_{\ell},j_{\ell})\in \mathcal{H}_n$ and $\varphi(\omega)=\prod_{\ell=1}^{n+1}\omega( y_{\ell},k_{\ell})\in \mathcal{H}_{n+1}$.
We have
\begin{align*}
    \langle a_i^{\dagger} \psi, \varphi\rangle_{L^2(\mathcal{E},\pi)} & =\E\Big[\omega(\0,i)\prod_{\ell=1}^n\omega(x_{\ell},j_{\ell})\prod_{\ell=1}^{n+1}\omega( y_{\ell},k_{\ell})\Big]\\
    & = \sum_{r=1}^{n+1}\E\Big[\prod_{\ell=1}^n\omega(x_{\ell},j_{\ell}) \prod_{\ell\in [n+1]\setminus\{r\}}\omega(y_{\ell},k_{\ell})\Big]\E[\omega(\0,i) \omega(y_{r},k_r)] = \langle \psi, a_i\varphi\rangle_{L^2(\mathcal{E},\pi)}. 
    \end{align*}
    
\subsection{Induced operators on the bosonic Fock space}\label{sec:Fock}

Let $f=(f^{(0)}, f^{(1)},\cdots)$ and $g=(g^{(0)}, g^{(1)},\cdots)$ be two elements in $\bigoplus_{n\ge0} \ell^2(U,d,n)$, where we use the convention that $\ell^2(U,d,0)=\ell^2_{\rm sym}(U,d,0)=\ell^2_0(U,d,0):=\R$. We define the scalar product
\begin{align}
\label{scalar}
    \llangle f,g \rrangle = \sum_{n\ge 0} n!\langle f^{(n)}, g^{(n)}\rangle_{\ell^2(U,d,n)},
\end{align}
and let $\|f\|=\sqrt{\llangle f,f\rrangle}$ be the associated norm. 

Let $\Gamma_0(U,d)$ be the bosonic Fock space consisting of all $f\in \bigoplus_{n\ge0} \ell^2_{0}(U,d,n)$ such that $\|f\|<\infty$.
By Proposition \ref{isometry} and the Rademacher chaos expansion \eqref{chaos.decomp}, each element $f\in \Gamma_0(U,d)$ can be identified with a random variable in $L^2(\mathcal{E},\pi)$.  With a slight abuse of notation, we also define $G, A$ and $S$ as operators on $\Gamma_0(U,d)$ as below.

For each $n\ge 1$, let $a_i^{\dagger}: \ell^2_{0}(U,d,n) \to \ell^2_{0}(U,d,n+1) $ and $a_i: \ell^2_{0}(U,d,n) \to \ell^2_{0}(U,d,n-1)$ be defined such that
for $f\in \ell^2_{0}(U,d,n)$, 
\begin{align*}
    (a_i^{\dagger} f)_{j_{1:n+1}}(x_{1:n+1})&:=
     \frac{1}{n+1}\sum_{r=1}^{n+1} \delta_{ (x_{r},j_{r}), (\0,i)}\prod_{\ell \in[n+1]\setminus\{r\}}(1-\delta_{(x_{\ell}, j_{\ell}), (\0, i)})  f_{j_{1:n+1\setminus\{r\}}}(x_{1:n+1\setminus\{r\}}),\\
    (a_i f)_{j_{1:n-1}}(x_{1:n-1})& :=    n f_{j_1,j_2,\cdots,j_{n-1},i}(x_1,\cdots, x_{n-1}, \0).
\end{align*}

We notice that for every $f\in \ell^2_{0}(U,d,n)$,
\begin{align}\label{eq.Ia}
   a_i^{\dagger} I^{(f)}_n = I^{(a_i^{\dagger} f)}_{n+1}\quad\text{and}\quad a_i I^{(f)}_n = I^{(a_i f)}_{n-1}. 
\end{align}
Indeed, by the definitions of $a_i^\dagger f$ and $I_{n+1}^{(a_i^\dagger f)}$, and by the symmetry of $f$, we have 
\begin{align*} I_{n+1}^{(a_i^\dagger f)} &= \frac{1}{n+1} \sum_{\substack{ x\in U^{n+1}, j\in [d]^{n+1}}} \sum_{r=1}^{n+1} \delta_{(x_r,j_r),(\mathbf{0},i)} \prod_{\ell\in[n+1]\setminus\{r\}} \left(1-\delta_{(x_\ell, j_\ell),(\mathbf{0},i)}\right) f_{j_{1:n+1\setminus\{r\}}}(x_{1:n+1\setminus\{r\}}) \prod_{\ell=1}^{n+1}\omega(x_\ell,j_\ell)\\
&= \omega(\mathbf{0},i) \sum_{\substack{x_{1:n}\in U^n, j_{1:n}\in[d]^n}} \prod_{\ell=1}^n \left(1-\delta_{(x_\ell,j_\ell),(\mathbf{0},i)}\right) f_{j_{1:n}}(x_{1:n}) \prod_{\ell=1}^n\omega(x_\ell,j_\ell)= a_i^\dagger I_n^{(f)}. \end{align*} Similarly, \begin{align*} a_iI_n^{(f)} &= \sum_{\substack{x_{1:n}\in U^n, j_{1:n}\in[d]^n}} f_{j_{1:n}}(x_{1:n}) \sum_{r=1}^n \delta_{(x_r,j_r),(\mathbf{0},i)} \prod_{\ell\in[n]\setminus\{r\}}\omega(x_\ell,j_\ell)\\ &= n\sum_{\substack{x_{1:n-1}\in U^{n-1},j_{1:n-1}\in[d]^{n-1}}} f_{j_1,\ldots,j_{n-1},i}(x_1,\ldots,x_{n-1},\mathbf{0}) \prod_{\ell=1}^{n-1}\omega(x_\ell, j_\ell)
= I_{n-1}^{(a_if)}. \end{align*} 
This verifies \eqref{eq.Ia}.


For each $h\in [d]$, we define $\tau_{\pm h}:\ell_0^2(U,d,n) \to \ell_0^2(U,d,n)$ such that
$$(\tau_{\pm h} f)_{j_1,\cdots,j_n}(x_1,\cdots,x_n)=f_{j_1,\cdots,j_n}(x_1\mp (e_{h}-e_{j_1}),\cdots, x_n\mp (e_{h}-e_{j_n})).
$$  
It is clear that
\begin{align}\label{eq.tau}
   \tau_{\pm h} I^{(f)}_n = I^{(\tau_{\pm h} f)}_{n}.
\end{align}
Define $$A_+:=\frac{1}{2} \sum_{i=1}^d a_i^{\dagger}(\tau_i  -\tau_{-i}),\quad A_{-} := \frac{1}{2} \sum_{i=1}^d (\tau_{-i}  -\tau_{i})a_i ,\quad S:= \frac{1}{2}\sum_{i=1}^d (\tau_i  +\tau_{-i}-2) \quad\text{and}\quad G:=S+A_+-A_{-}.$$ 
From \eqref{eq.Ia} and \eqref{eq.tau}, we thus have
$$SI_{n}^{(f)}=I_n^{(Sf)}, A_+ I_n^{(f)}=I_{n+1}^{(A_+f)}\text{ and }   A_{-} I_n^{(f)}=I_{n-1}^{(A_{-}f)}.$$
For $(x_{1:n+1},j_{1:n+1})\in (U\times[d])^{n+1}$, 
we have
\begin{align} &(A_+f)_{j_{1:n+1}}(x_{1:n+1}) 
= \frac{1}{2(n+1)}\sum_{r=1}^{n+1}\delta_{ x_{r},\0} \prod_{\ell \in[n+1]\setminus\{r\}}(1-\delta_{(x_{\ell}, j_{\ell}), (\0, j_r)}) \\
\nonumber  & \cdot \Big(f_{j_{1:n+1\setminus\{r\}}}\big((x_{\ell}-  e_{j_r}+e_{j_{\ell}})_{\ell\in 1:n+1\setminus\{r\}}\big)-f_{j_{1:n+1\setminus\{r\}}}\big((x_{\ell}+  e_{j_r}-e_{j_{\ell}})_{\ell\in 1:n+1\setminus\{r\}}\big)\Big).\end{align}
Notice that $A_-=(A_+)^*$ on $\Gamma_0(U,d)$. In particular, for $f\in \ell_0^2(U,d,n)$ and $g\in \ell_0^2(U,d,n+1)$, we have
$$(n+1)\langle  A_+ f, g\rangle_{\ell^2(U,d,n+1)} =\langle  f, A_{-}g\rangle_{\ell^2(U,d,n)}.$$ 
For $(x_{1:n},j_{1:n})\in (U\times [d])^n$, we have
 \begin{align}
     &(Sf)_{j_{1:n}}(x_{1:n})=   \frac{1}{2}\sum_{h\in[d]}\Big( f_{j_{1:n}}\big((x_{\ell}+  e_{h}-e_{j_{\ell}})_{\ell\in 1:n}\big)+f_{j_{1:n}}\big((x_{\ell}-  e_{h}+e_{j_{\ell}})_{\ell\in 1:n}\big) -2f_{j_{1:n}}(x_{1:n})\Big).
 \end{align}

\subsection{Truncated Helmholtz equations}
\label{sec:truncated.helmholtz}

Recall that $\Gamma_0(U,d)$ is the bosonic Fock space defined in Section~\ref{sec:Fock}. For $n\ge1$, let
$$\Gamma_0^{(n)}(U,d):=\bigoplus_{k=0}^n\ell^2_0(U,d,k),$$
and let $Q_n$ be the orthogonal projection from $\Gamma_0(U,d)$ onto
$\Gamma_0^{(n)}(U,d)$. Define the truncated generator
$$G_n:=Q_nGQ_n.$$

Recall that $\phi(\omega)=\omega(\mathbf0,1)$. Let
$u=(0,u^{(1)},0,0,\ldots),$ where $u^{(1)}\in\ell^2_0(U,d,1)$ is given by $u_j^{(1)}(x)=\delta_{(x,j),(\mathbf0,1)}$, for $(x,j)\in U\times[d]$. Then $I_1^{(u^{(1)})}=\phi,$ and therefore
$$\langle \phi,(\lambda-G)^{-1}\phi\rangle_{L^2(\mathcal{E},\pi)}=\llangle u,(\lambda-G)^{-1}u\rrangle.$$

Similarly to Lemma 2.1 in \cite{LQSY2004}, we obtain the following result:

\begin{lemma}
\label{lemma.monotonicity}
For every $n\ge1$ and every $\lambda>0$,
\begin{equation}
\label{eq:truncation.monotonicity}
\llangle u,(\lambda-G_{2n})^{-1}u\rrangle\le\llangle u,(\lambda-G)^{-1}u\rrangle\le\llangle u,(\lambda-G_{2n+1})^{-1}u\rrangle.
\end{equation}
\end{lemma}

\begin{proof}
Let $v_n=(v^{(0)}_n, v^{(1)}_n,\cdots, v^{(n)}_n,0,0,\cdots )\in \Gamma_{0}^{(n)}(U,d)$ be the solution to the truncated Helmholtz equation
\begin{equation}\label{tr.G.Poisson}
    (\lambda- G_n)v_n=u.
\end{equation}
Since $u$ is supported on the first chaos, the quantity $\llangle u,v_n\rrangle$ is determined by the first-chaos component:
\begin{align*}
\llangle u,v_n\rrangle
=\langle u^{(1)},v_n^{(1)}\rangle_{\ell^2(U,d,1)}.\end{align*}
The equation \eqref{tr.G.Poisson} is equivalent to the following system of equations:
$$
\left\{\begin{array} {l} {{{{\left( \lambda-S\right)} v_{n}^{( n )}-{ A}_{+} v_{n}^{( n -1)}=0,}}} \\ {{{{\left( \lambda-S\right)} v_{n}^{( k )}-{ A}_{+} v_n^{( k-1 )}+A_{-} v_{n}^{( k+1 )}=0,}}} \text{ for $2\le k\le n-1$},\\ {{{{\left( \lambda-S\right)} v_{n}^{( 1 )}+A_{-} v_{n}^{( 2 )}=u^{(1)}.}}} \\ \end{array} \right. 
$$
Solving the above system iteratively starting from $k=n$, we get
\begin{align} \label{eq:schur.rep}
\llangle u, v_n \rrangle=\langle u^{(1)}, v_n^{( 1 )} \rangle_{\ell^2(U,d,1)}=\langle u^{(1)}, ( \lambda-S+\Lambda_{n} )^{-1} u^{(1)}\rangle_{\ell^2(U,d,1)} 
\end{align}
where the self-adjoint operators $\Lambda_{j}$ are recursively defined as
\begin{align}\label{eq:def.Lambda}
\Lambda_{1} := 0,\quad\Lambda_{j+1}=A_{-} ( \lambda-S+{\Lambda}_{j} )^{-1} {A}_{+} \quad \text{for } \, j \geq1.\end{align}
Note that these operators are positive and leave each space $\ell_{0}^2(U,d,n)$ invariant, i.e., ${\Lambda}_{j} \ell_{0}^2(U,d,n) \subset  \ell_{0}^2(U,d,n),$ for all $j, \, n \in\mathbb{N}$.

Notice that $\lambda-S$ is non-negative. It is clear that $\Lambda_1=0\le \Lambda_3=A_{-}(\lambda-S+\Lambda_2)^{-1}{A}_+\le A_{-}(\lambda-S)^{-1}{A}_+=\Lambda_2$. By induction, $\Lambda_{m}\ge\Lambda_n$ if $\Lambda_{m-1}\le\Lambda_{n-1}$. In particular, for each $m\ge 1$, 
\begin{equation}
\label{eq:Lambda.parity.order}\Lambda_{2m-1}\le\Lambda_{2m+1}\le\Lambda_{2m+2}\le\Lambda_{2m}.
\end{equation}
The comparison of the truncated resolvents with the full resolvent follows from the same variational truncation argument as in \cite{LQSY2004}. Combining this argument with \eqref{eq:schur.rep} and \eqref{eq:Lambda.parity.order}, we obtain the result of the lemma.
\end{proof}

\section{Removal of hard-core}
\label{sec:hard-core.removal}

The diagonal-vanishing condition defining the bosonic Fock space $\Gamma_0(U,d)$, also called the hard-core constraint (see, e.g., Section 4 in \cite{LQSY2004}), substantially complicates the analysis, in particular the direct use of Fourier methods. In this section, we introduce an unconstrained analogue $\widetilde A_+$ of the operator $A_+$ and derive its explicit Fourier representation. We then compare the quadratic forms associated with $A_+$ and $\widetilde A_+$ and prove that the error produced by removing the hard-core constraint is controlled by the Dirichlet form associated with the operator $S$ (see Proposition \ref{prop:hc.form}).

\subsection{Unconstrained operator \texorpdfstring{$\widetilde A_+$}{\~A+}}

Recall that ${\ell}^2_{\rm sym}(U,d,n)$ is the space of square-summable symmetric functions on $(U\times [d])^n$  (they may not necessarily vanish on diagonals). 
For each $n\ge 0$, define $\widetilde a_i^\dagger:\ell^2_{\rm sym}(U,d,n)\rightarrow\ell^2_{\rm sym}(U,d,n+1)$ and $\widetilde a_i:\ell^2_{\rm sym}(U,d,n+1)\rightarrow\ell^2_{\rm sym}(U,d,n)$ by
\begin{align*}
(\widetilde a_i^\dagger f)_{j_{1:n+1}}(x_{1:n+1})&:=\frac{1}{n+1}\sum_{r=1}^{n+1}\mathbf 1_{\{(x_r,j_r)=(\mathbf0,i)\}}f_{j_{1:n+1\setminus\{r\}}}(x_{1:n+1\setminus\{r\}}),\\
(\widetilde a_i f)_{j_{1:n}}(x_{1:n})&:=(n+1)f_{i,j_{1:n}}(\mathbf 0,x_{1:n}).
\end{align*}
Define $\widetilde{A}_+: {\ell}^2_{\rm sym}(U,d,n) \to  {\ell}^2_{\rm sym}(U,d,n+1)$ by
\begin{equation}
\label{eq:def.unconstrained.total.creation}\widetilde A_+:=\sum_{i=1}^d\widetilde A_{+,i}\quad\text{with}\quad \widetilde A_{+,i}:=\frac12\widetilde a_i^\dagger(\tau_i-\tau_{-i}).
\end{equation}
Then
\begin{align}\label{eq:unconstrained.Aplus}    (\widetilde{A}_+f)_{j_{1:n+1}}(x_{1:n+1}) = 
  \frac{1}{2(n+1)}\sum_{r=1}^{n+1}  & \delta_{ x_{r},\0} \Big(f_{j_{1:n+1\setminus\{r\}}}\big((x_{\ell}-  e_{j_r}+e_{j_{\ell}})_{\ell\in 1:n+1\setminus\{r\}}\big)\\
 \nonumber \quad & -f_{j_{1:n+1\setminus\{r\}}}\big((x_{\ell}+  e_{j_r}-e_{j_{\ell}})_{\ell\in 1:n+1\setminus\{r\}}\big)\Big), \end{align}
which can be viewed as the unconstrained analogue of $A_+$ on ${\ell}^2_{\rm sym}(U,d,n)$. Note that the adjoint of $\widetilde{A}_+$ on $\Gamma_{\rm sym}(U,d):=\bigoplus_{n\ge 0}\ell^2_{\rm sym}(U,d,n)$ (equipped with the scalar product $\llangle \cdot ,\cdot \rrangle$ defined in \eqref{scalar}) is given by
$$(\widetilde{A}_+)^*:= \frac12\sum_{i=1}^d(\tau_{-i}-\tau_i)\widetilde a_i.$$
In particular, for $f\in \ell^2_{\rm sym}(U,d,n)$ and $g\in \ell^2_{\rm sym}(U,d,n+1)$, we have
\begin{align}\label{Atilde.adjoint}
\langle f,  (\widetilde{A}_+)^*g \rangle_{\ell^2_{\rm sym}(U,d,n)}=(n+1) \langle\widetilde A_+ f, g\rangle_{\ell^2_{\rm sym}(U,d,n+1)}. \end{align}
For $n\ge1$, define
$$E_n:=\left\{
(z_1,\ldots,z_n)\in (U\times[d])^n : z_r\neq z_s\text{ for all }r\neq s \right\}.$$
Let $P_n$ be multiplication by $\mathbf 1_{E_n}$ on $\ell^2_{\rm sym}(U,d,n)$.
Then
$\ell^2_0(U,d,n)=P_n\ell^2_{\rm sym}(U,d,n).$
For $z_{1:n}=(z_1,\ldots,z_n)\in (U\times[d])^n$, define
\begin{equation}
\label{eq:def.open.indicator}
\mathcal O_i^{(n)}(z_{1:n}):=\mathbf 1_{\{(\mathbf0,i)\notin\{z_1,\ldots,z_n\}\}},\quad \mathcal B_i^{(n)}(z_{1:n}):=\mathbf 1_{\{(\mathbf0,i)\in\{z_1,\ldots,z_n\}\}}.
\end{equation}
As multiplication operators on $\ell^2_{\rm sym}(U,d,n)$, we note that $\mathcal O_i^{(n)}+\mathcal B_i^{(n)}=I$ and $\mathcal O_i^{(n)}\mathcal B_i^{(n)}=0$. 

For $i\in [d]$, since $\tau_i$ and $\tau_{-i}$ preserve the set $E_n$, if $f\in\ell^2_0(U,d,n)$ then $(\tau_i-\tau_{-i})f$ is also supported on $E_n$.  Notice also that for $g\in\ell^2_{\rm sym}(U,d,n)$ supported on $E_n$, we have $a_i^\dagger g=\widetilde a_i^\dagger\mathcal O_i^{(n)}g$.
Thus
$$A_+ = \frac12 \sum_{i=1}^d\widetilde a_i^\dagger\mathcal O_i^{(n)}(\tau_i-\tau_{-i})\quad\text{on }\ell^2_0(U,d,n).$$
Define $B:\ell^2_0(U,d,n)\rightarrow (I-P_{n+1})\ell^2_{\rm sym}(U,d,n+1)$ by
\begin{equation}\label{eq:def.true}
B:=\sum_{i=1}^d B_i \quad\text{with}\quad  B_i:=\frac12\widetilde a_i^\dagger\mathcal B_i^{(n)}(\tau_i-\tau_{-i}).\end{equation}
Hence, on $\ell^2_0(U,d,n)$, we have
\begin{equation}
\label{eq:Aplus.blocked.decomposition}\widetilde A_+=A_++B.
\end{equation}
Note that for $f\in {\ell}^2_{0}(U,d,n),$
\begin{align*}
   & \big(Bf\big)_{j_{1:n+1}}(x_{1:n+1})=\frac{1}{2(n+1)}\sum_{r=1}^{n+1}\delta_{ x_{r},\0} \sum_{s\in [n+1]\setminus\{r\}} \delta_{(x_s,j_s),(\0,j_r)} \\
   & \cdot \Big(f_{j_r,j_{1:n+1\setminus\{r,s\}}}\big(\0,(x_{\ell}-  e_{j_r}+e_{j_{\ell}})_{\ell\in 1:n+1\setminus\{r,s\}}\big)-f_{j_r,j_{1:n+1\setminus\{r,s\}}}\big(\0,(x_{\ell}+  e_{j_r}-e_{j_{\ell}})_{\ell\in 1:n+1\setminus\{r,s\}}\big)\Big).
\end{align*}

\subsection{Fourier representation}
\label{subsec:fourier.trans}

For $\varphi : (\mathbb{Z}^{d})^n \to\mathbb{R}$, denote by $\mathcal{F}{\varphi}$ its Fourier transform, which is given by
$${\mathcal{F}{\varphi}} ( p_1,\cdots, p_n ) :=\sum_{x_1,\cdots,x_n\in\mathbb{Z}^{d}} {\rm e}^{\mathbf{i} (p_1 \cdot x_1+\cdots+p_n\cdot x_n)} \varphi ( x_1,\cdots, x_n)\quad \text{ for } p_{\ell}=(p_{\ell}^1,\cdots, p_{\ell}^d) \in\mathbb{T}^{d}, \ell\in [n],$$
where $\mathbb{T}^d=\R^d/(2\pi\Z^d)$ and $p\cdot x$ is the usual Euclidean dot product on $\mathbb{R}^d$. By the Parseval–Plancherel formula,
$$
\langle \varphi, \psi \rangle_{\ell^2((\mathbb{Z}^{d})^n)}=\frac{1} {( 2 \pi)^{nd}} \int_{\mathbb{T}^{nd}}  \mathcal{F}{\varphi} ( p ) \overline{\mathcal{F}{\psi} ( p )}  \rmd p. 
$$

For a function $f : (U\times [d])^n \to\mathbb{R}$, we define the Fourier transform ${\mathcal{F}{f}}_{j} (p)$ by first extending $f_j$ to $(\mathbb{Z}^{d})^n$ by setting it equal to zero outside $U^n$ and then taking the Fourier transform.

For $r\in[d]$, define the linear map $V_r:\mathbb Z^d\longrightarrow \mathbb Z^d$ such that for each $y=(y^h)_{h\in [d]}\in \Z^d$,
$$V_r y :=\left(y^1,\ldots,y^{r-1},-\sum_{h\in[d]\setminus\{r\}}y^h, y^{r+1},\ldots,y^d\right).$$
Then $V_r$ maps
$U_r:=\{y\in\mathbb Z^d:y^r=0\}$ bijectively onto $U=\big\{x\in\mathbb Z^d:\sum_{h=1}^d x^h=0\big\}.$
The transpose map satisfies
$$V_r^\top p=\left(p^1-p^r,\ldots,p^{r-1}-p^r, 0, p^{r+1}-p^r,\ldots,p^d-p^r\right).$$
We have  
\begin{align*}
{\mathcal{F}{f}}_{j_1,\cdots,j_n} ( p_1,\cdots, p_n ) & =\sum_{(x_1,\cdots,x_n)\in U^{n}} {\rm e}^{\mathbf{i} (p_1 \cdot x_1+\cdots+p_n\cdot x_n)} f_{j_1,\cdots,j_n} ( x_1,\cdots, x_n)\\
& =\sum_{(y_1,\cdots,y_n)\in U_{j_1}\times \cdots \times U_{j_n}} {\rm e}^{\mathbf{i} (p_1 \cdot V_{j_1} y_1+\cdots+p_n\cdot V_{j_n} y_n)} f_{j_1,\cdots,j_n} ( V_{j_1} y_1,\cdots, V_{j_n} y_n)\\
& = \sum_{(y_1,\cdots,y_n)\in U_{j_1}\times \cdots \times U_{j_n}} {\rm e}^{\mathbf{i} ( V_{j_1}^\top p_1 \cdot  y_1+\cdots+ V_{j_n}^\top p_n\cdot  y_n)} f_{j_1,\cdots,j_n} ( V_{j_1} y_1,\cdots, V_{j_n} y_n).
\end{align*}
For $j=(j_1,\ldots,j_n)\in[d]^n$, set
$$\mathbb T_j^d:=V_{j_1}^\top  \mathbb{T}^d\times\cdots\times V_{j_n}^\top  \mathbb{T}^d=\left\{(p_1,\ldots,p_n)\in(\mathbb T^d)^n : p_\ell^{j_\ell}=0\text{ for every }\ell\in[n]\right\}.$$

For $f\in\ell^2_{\rm sym}(U,d,n)$, define the parameterized Fourier
transform by
\begin{equation}
\label{eq:def.parameterized.fourier}
\widehat f_j(p):=\sum_{y_1\in U_{j_1}}\cdots\sum_{y_n\in U_{j_n}}\exp\left(\mathbf i\sum_{\ell=1}^n p_\ell\cdot y_\ell\right)f_j(V_{j_1}y_1,\ldots,V_{j_n}y_n)\quad\text{for } p\in\mathbb T_j^d.
\end{equation}
Note that $\widehat{f}_{j}(V_{j_1}^\top p_{1},\cdots, V_{j_n}^\top p_{n})={\mathcal{F}{f}}_{j_1,\cdots,j_n} ( p_1,\cdots, p_n).$
We also have the Parseval–Plancherel formula 
\begin{equation}
\label{eq:parseval.parameterized}
\langle f,g\rangle_{\ell^2(U,d,n)}=\frac{1}{(2\pi)^{n(d-1)}}\sum_{j\in[d]^n}\int_{\mathbb T_j^d}\widehat f_j(p)\overline{\widehat g_j(p)}\rmd p.
\end{equation}

\begin{lemma}\label{lem:Fourier}  For $f\in {\ell}_{\rm sym}^2(U,d,n)$, we have
\begin{align}   \label{Fou.S.psi}
    \widehat{Sf}_j(\xi)&=-2\sum_{h\in [d]}\sin^2\Big(\frac{1}{2}\sum_{\ell\in [n]}\xi_{\ell}^h \Big)\widehat{f_j}(\xi) \quad \text{for \ } \ j\in [d]^n, \xi\in \mathbb{T}_j^d; \\
  \label{Fou.A.psi} 
 \widehat{\widetilde{A}_+ f}_{j}(\xi) & = \frac{\mathbf{i}}{n+1} \sum_{r=1}^{n+1} \widehat{f}_{j_{1:n+1\setminus\{r\}}}(\xi_{1:n+1\setminus\{r\}}) \sin\Big(\sum_{\ell\in[n+1]\setminus\{r\}}\xi_{\ell}^{{j_r}} \Big)\quad\text{for $j\in [d]^{n+1}, \xi\in \mathbb{T}_j^d$}. 
  \end{align}
\end{lemma}

\begin{proof}
For $f\in {\ell}_{\rm sym}^2(U,d,n)$, $j\in [d]^n$ and $\xi\in \mathbb{T}_j^d$, we have
\begin{align*} 
\widehat{\tau_{\pm h} f}_j(\xi)& = \sum_{y\in U_j}(\tau_{\pm h}f)_j(V_{j_1}y_1,\ldots,V_{j_n}y_n){\rm e}^{\mathbf i\sum_{\ell=1}^n y_\ell\cdot\xi_\ell}.\\
& =\sum_{y\in U_j}f_j\big( V_{j_1}y_1\mp(e_h-e_{j_1}),\ldots, V_{j_n}y_n\mp(e_h-e_{j_n})\big){\rm e}^{\mathbf i\sum_{\ell=1}^n y_\ell\cdot\xi_\ell}.
\end{align*}
For each $\ell \in [n]$ with $h\neq j_\ell$ we have $V_{j_\ell}y_\ell\mp(e_h-e_{j_\ell})=V_{j_\ell}(y_\ell\mp e_h)$, and we change variable $y_\ell\mapsto y_\ell\mp e_h$ in $U_{j_\ell}$. Therefore
$$\widehat{\tau_{\pm h} f}_j(\xi)={\rm e}^{\pm\mathbf i\sum_{\ell=1}^n\xi_\ell^h}\widehat f_j(\xi).$$

Recall that
$(Sf)_j(x)=\frac12\sum_{h\in[d]}\left[(\tau_hf)_j(x)+(\tau_{-h}f)_j(x)-2f_j(x)\right].$
Hence 
\begin{align*}
\widehat{Sf}_j(\xi)=-2\sum_{h\in [d]}\sin^2\Big(\frac{1}{2}\sum_{\ell\in [n]}\xi_{\ell}^h \Big)\widehat{f_j}(\xi).
\end{align*}
Recall that for $x=x_{1:n+1}\in U^{n+1}, j=j_{1:n+1}\in [d]^{n+1}$,
$$(\widetilde A_+f)_j(x)=\frac{1}{2(n+1)}\sum_{r=1}^{n+1}\mathbf 1_{\{x_r=\mathbf0\}}\left[(\tau_{j_r}f)_{j_{1:n+1\setminus\{r\}}}(x_{1:n+1\setminus\{r\}})-(\tau_{-j_r}f)_{j_{1:n+1\setminus\{r\}}}(x_{1:n+1\setminus\{r\}})\right].$$
We write $j^{(r)}=j_{1:n+1\setminus\{r\}}$, $\xi^{(r)}=\xi_{1:n+1\setminus\{r\}}$. Taking the Fourier transform, and using the fact that $V_{j_r}y_r=\mathbf0$ if and only if $y_r=\mathbf0$, we obtain 
\begin{align*}\nonumber
 \widehat{\widetilde A_+f}_j(\xi)&=\frac{1}{2(n+1)}\sum_{r=1}^{n+1}\left[\widehat{\tau_{j_r}f}_{j^{(r)}}(\xi^{(r)})- \widehat{\tau_{-j_r}f}_{j^{(r)}}(\xi^{(r)})\right]   = \frac{\mathbf{i}}{n+1} \sum_{r=1}^{n+1}\widehat{f}_{j^{(r)}}(\xi^{(r)})   \sin\Big(\sum_{\ell\neq r}\xi_{\ell}^{j_r} \Big).
    \end{align*}
\end{proof}

\subsection{Quadratic forms related to \texorpdfstring{$\widetilde A_+$}{\~A+}}
For each $n\ge 1$, let $R:\ell^2_{\rm sym}(U,d,n)\to \ell^2_{\rm sym}(U,d,n)$ be a \textit{diagonal operator} with \textit{Fourier multipliers} $(\mathfrak r^{(n)})_{n\ge 1}$. That is, for every $n\ge 1$ and for every $f\in \ell^2_{\rm sym}(U,d,n)$, we have
\begin{equation*}
\widehat{Rf}_j(p)=\mathfrak r_{j}^{(n)}(p)\widehat f_j(p)\quad\text{for } j\in[d]^{n},\, p\in\mathbb T_j^d.
\end{equation*}
We suppose that:
\begin{itemize}
    \item $R$ is positive, i.e., for every $n\ge 1$ and for every non-zero function $f\in \ell^2_{\rm sym}(U,d,n)$,
   \begin{align*}
       \langle f, Rf \rangle_{\ell^2(U,d,n)} = \sum_{x\in U^n, j\in [d]^n} f_{j}(x)Rf_j(x) > 0.
   \end{align*}

    \item for each $n\ge 1$, $\mathfrak r^{(n)}$ is symmetric, i.e., for every $j=(j_1,\cdots, j_n)\in [d]^n,\, p=(p_1,\cdots, p_n)\in \mathbb{T}^d_{j}$ and for every permutation $\sigma$ of $[n]$,
\begin{align}\label{eq:mul.symm}
    \mathfrak r_{j_{\sigma(1)}, \cdots, j_{\sigma(n)} }^{(n)}(p_{\sigma(1)}, \cdots, p_{\sigma(n)})=\mathfrak r_{j_1,\cdots, j_n}^{(n)}(p_1, \cdots ,p_n).
\end{align}
\end{itemize}

Fix some $f\in {\ell}^2_{0}(U,d,n)$.
Using \eqref{Atilde.adjoint}, the Parseval–Plancherel formula \eqref{eq:parseval.parameterized} and Lemma~\ref{lem:Fourier}, we have 
\begin{align*}
  &  \langle f, (\widetilde{A}_{+})^* R \widetilde{A}_+ f \rangle_{\ell^2(U,d,n)} = (n+1)\langle \widetilde{A}_+ f,  R \widetilde{A}_+ f \rangle_{\ell^2(U,d,n+1)} = (n+1)\sum_{j\in [d]^{n+1}}\sum_{x\in U^{n+1}} \widetilde{A}_+f_j(x)\cdot  R \widetilde{A}_+ f_j(x)\\
  &= \frac{(n+1)}{(2\pi)^{(n+1)(d-1)}} \sum_{j\in [d]^{n+1}}\int_{\mathbb{T}_j^d}\overline{\widehat{\widetilde{A}_+ f}_j(p)} \widehat{R \widetilde{A}_+f}_j(p) \rmd p  =\frac{(n+1)}{(2\pi)^{(n+1)(d-1)}} \sum_{j\in [d]^{n+1}}\int_{\mathbb{T}_j^d}\mathfrak{r}_{j}^{(n+1)}(p) |\widehat{\widetilde{A}_+ f}_j(p)|^2  \rmd p   \\
    & = \frac{1}{(n+1)(2\pi)^{(n+1)(d-1)}}  \sum_{j\in [d]^{n+1}}\int_{\mathbb{T}_j^d}  \mathfrak{r}_{j}^{(n+1)}(p)  \left|\sum_{r=1}^{n+1} \widehat{f}_{j_{1:n+1\setminus\{r\}}}(p_{1:n+1\setminus\{r\}}) \sin\Big(\sum_{\ell\neq r}p_{\ell}^{j_r} \Big) \right|^2\rmd p.
\end{align*}

We now decompose the quadratic form $\langle f, (\widetilde{A}_{+})^* R \widetilde{A}_+ f \rangle_{\ell^2(U,d,n)}$ into a diagonal part and an off-diagonal part:
\begin{equation}
\label{eq:def.VDiag.VOff.forms}
\left\langle f, (\widetilde A_+)^*R\widetilde A_+f\right\rangle_{\ell^2(U,d,n)}=V_{\rm Diag}(R,f)+V_{\rm Off}(R,f),
\end{equation}
where \begin{align*}
V_{\rm Diag}(R,f)&:=\frac{1}{(n+1)(2\pi)^{(n+1)(d-1)}}\sum_{j\in[d]^{n+1}}\int_{\mathbb T_j^d}\mathfrak{r}_{j}^{(n+1)}(p)\nonumber\\
&\quad\times\sum_{r=1}^{n+1}\left|\widehat f_{j_{1:n+1\setminus\{r\}}}(p_{1:n+1\setminus\{r\}})\right|^2\sin^2\left(\sum_{\ell\in[n+1]\setminus\{r\}}p_\ell^{j_r}\right)\rmd p,\\
V_{\rm Off}(R,f)&:=\frac{1}{(n+1)(2\pi)^{(n+1)(d-1)}}\sum_{j\in[d]^{n+1}}\int_{\mathbb T_j^d}\mathfrak{r}_{j}^{(n+1)}(p)\nonumber\\&\times\sum_{\substack{r,s=1\\ r\neq s}}^{n+1}\overline{\widehat f_{j_{1:n+1\setminus\{r\}}}(p_{1:n+1\setminus\{r\}})}\widehat f_{j_{1:n+1\setminus\{s\}}}(p_{1:n+1\setminus\{s\}})\sin\left(\sum_{\ell\in[n+1]\setminus\{r\}}p_\ell^{j_r}\right)\sin\left(\sum_{\ell\in[n+1]\setminus\{s\}}p_\ell^{j_s}\right)\rmd p.
\end{align*}

By symmetry, we obtain 
\begin{align}\label{eq:VDiag}
    V_{\rm Diag }(R, f)=\frac{1}{(2\pi)^{(n+1)(d-1)}} \sum_{j\in [d]^{n+1}} \int_{\mathbb{T}_j^d}  \mathfrak{r}_{j}^{(n+1)}(p_{1:n+1})|\widehat{f}_{j_{1:n}}(p_{1:n})|^2 \sin^2\Big(\sum_{\ell\in [n]}p_{\ell}^{j_{n+1}} \Big)\rmd p_{1:n+1}.
\end{align}
and
\begin{align}\label{eq:VOff} \nonumber
    V_{\rm Off }(R, f)= \frac{ n}{(2\pi)^{(n+1)(d-1)}}  \operatorname{Re} \sum_{j\in [d]^{n+1}}\int_{\mathbb{T}_j^d} & \mathfrak{r}_{j}^{(n+1)}(p) \overline{\widehat{f}_{j_{1:n}}(p_{1:n})}\widehat{f}_{j_{1:n+1\setminus\{n\}}}(p_{1:n+1\setminus\{n\}}) \\
    &\cdot \sin\Big(\sum_{\ell\in [n+1]\setminus\{n\}}p_{\ell}^{j_{n}} \Big)\sin\Big(\sum_{\ell\in [n]}p_{\ell}^{j_{n+1}} \Big)\rmd p_{1:n+1}.
    \end{align}

\begin{lemma}\label{lem:V.diag}
Let $R$ be a positive diagonal operator with Fourier multipliers $(\mathfrak r^{(n)})_{n\ge 1}$ satisfying \eqref{eq:mul.symm}.
Let $\widetilde R$ be a positive diagonal operator with Fourier multipliers $(\widetilde{\mathfrak r}^{(n)})_{n\ge 1}$ such that, for every $j\in[d]^n$ and every $p\in\mathbb T_j^d$,
\begin{align}
\label{eq:diag.criterion}
&\frac{1}{(2\pi)^{d-1}} \sum_{i=1}^d \sin^2\Big(\sum_{\ell=1}^n p_\ell^i\Big) \int_{\mathbb T_i^d} \mathfrak{r}_{j,i}^{(n+1)}(p,q)\rmd q \le \widetilde{\mathfrak{r}}_{j}^{(n)}(p).
\end{align}
Then
\begin{equation}
\label{eq:diag.criterion.forms}
V_{\rm Diag}(R,f) \le \langle f,\widetilde R f\rangle_{\ell^2(U,d,n)}.
\end{equation}
If \eqref{eq:diag.criterion} holds with the opposite inequality, then \eqref{eq:diag.criterion.forms} also holds with the opposite inequality.
\end{lemma}

\begin{proof}
    To prove the criterion, integrate \eqref{eq:VDiag}
first in the last variable $q=p_{n+1}\in\mathbb T_{j_{n+1}}^d$. This gives
\begin{align*}
V_{\rm Diag}(R,f) &= \frac{1}{(2\pi)^{n(d-1)}} \sum_{j\in[d]^n} \int_{\mathbb T_j^d} |\widehat f_j(p)|^2 \frac{1}{(2\pi)^{d-1}} \sum_{i=1}^d \sin^2\Big(\sum_{\ell=1}^n p_\ell^i\Big) \int_{\mathbb T_i^d} \mathfrak r_{j,i}^{(n+1)}(p,q)\rmd q\rmd p.
\end{align*}
Using \eqref{eq:diag.criterion}, we have
$$V_{\rm Diag}(R,f)\le \frac{1}{(2\pi)^{n(d-1)}} \sum_{j\in[d]^n} \int_{\mathbb T_j^d} |\widehat f_j(p)|^2 \widetilde{\mathfrak r}_{j}^{(n)}(p)\rmd p.$$
Applying the Parseval–Plancherel formula \eqref{eq:parseval.parameterized}, we obtain \eqref{eq:diag.criterion.forms}. The proof with the opposite inequality is identical.
\end{proof}

\subsection{Bounds for quadratic forms related to \texorpdfstring{$A_+$}{A+}}
\label{subsec:hard-core.form}
Throughout this subsection, we fix $d=3$ and let $R$ be a positive diagonal operator with multipliers $(\mathfrak{r}^{(n)})_{n\ge 1}$, which depend only on \textit{total momentum}, i.e., for each $n\ge 1$, there exists a non-negative measurable function $r_{n}$ on $\mathbb T^3$ such that
\begin{equation}
\label{eq:R.momentum}
\mathfrak{r}_{j}^{(n)}(p)= r_{n}\left(\sum_{\ell=1}^{n}p_\ell\right)\quad\text{for } j\in [3]^n,\, p\in \mathbb{T}^3_j. \end{equation}
For $p=(p^1,p^2,p^3)\in\mathbb T^3$, set
$$\theta(p):=2\sum_{h=1}^3\sin^2\left(\frac{p^h}{2}\right).$$
We assume that, for some constants
$C_R<\infty$ and $\kappa\ge1$,
\begin{equation}
\label{eq:R.res}
0\le r_{n+1}(p)
\le
\frac{C_R}{\lambda+\kappa\theta(p)}\quad\text{for } p\in\mathbb T^3 .
\end{equation}
In this subsection, we compare the quadratic form $\langle f, A_- R A_+ f \rangle_{\ell^2(U,3,n)} $ with $\langle f, (\widetilde{A}_+)^* R \widetilde{A}_+ f \rangle_{\ell^2(U,3,n)}$.

\begin{lemma}
\label{lem:trace.estimate}
Fix $i\in[3]$. There exists $C<\infty$ such that, for every $p=(p^1,p^2,p^3)\in\mathbb T^3$ satisfying $\sin(p^i/2)\neq0$, every $\lambda>0$, and every $\kappa\ge1$, we have
\begin{equation*}
\sin^2(p^i)\left(\int_{\mathbb T_i^3}\frac{\rmd q}{\lambda+\kappa\theta(p+q)}\right)\left(\int_{\mathbb T_i^3}\frac{\rmd q}{\theta(p+q)}\right)\le\frac{C}{\kappa}.
\end{equation*}
\end{lemma}

\begin{proof}
Set
$$b:=\sin^2\left(\frac{p^i}{2}\right).$$
Then $0<b\le1$. Let $\{j,k\}=[3]\setminus\{i\}$. Since $q\in\mathbb T_i^3$, we have $q^i=0$. Using the translation $x=p^j+q^j$ and $y=p^k+q^k$,
which preserves Haar measure on $\mathbb T_i^3$, we get
\begin{equation}
\label{eq:theta.lower}
\theta(p+q) =2b+2\sin^2\left(\frac{x}{2}\right)+2\sin^2\left(\frac{y}{2}\right) \quad\text{with}\quad (x,y)\in[-\pi,\pi)^2 .
\end{equation}

We first prove that
\begin{equation}
\label{eq:log.lb}
\int_{\mathbb T_i^3}\frac{\rmd q}{\theta(p+q)}\le C\left(1+|\log b|\right).
\end{equation}
Using \eqref{eq:theta.lower} and the fact that $\sin^2\left(\frac{t}{2}\right)\ge\frac{t^2}{\pi^2}$ for every $t\in[-\pi,\pi]$, we have
$$\int_{\mathbb T_i^3}\frac{\rmd q}{\theta(p+q)}=\int_{[-\pi,\pi)^2}\frac{\rmd x\rmd y}{2b+2\sin^2(x/2)+2\sin^2(y/2)}\le C\int_{[-\pi,\pi)^2}\frac{\rmd x\rmd y}{b+x^2+y^2}.$$
The square $[-\pi,\pi)^2$ is contained in the disk of radius $2\pi$.
Using polar coordinates on this disk, we obtain
$$\int_{[-\pi,\pi)^2}\frac{\rmd x\rmd y}{b+x^2+y^2}\le C\int_0^{2\pi}\frac{r\rmd r}{b+r^2}\le C\left(1+|\log b|\right).$$
This proves \eqref{eq:log.lb}.
Moreover,
$$\int_{\mathbb T_i^3}\frac{\rmd q}{\lambda+\kappa\theta(p+q)}=\frac{1}{\kappa}\int_{\mathbb T_i^3}\frac{\rmd q}{\lambda/\kappa+\theta(p+q)}\le\frac{C}{\kappa}(1+|\log b|).$$
Finally, using the fact that $\sin^2(p^i)=4\sin^2\left(\frac{p^i}{2}\right)\cos^2\left(\frac{p^i}{2}\right)\le 4b$ and that $\sup_{0<b\le1}b\left(1+|\log b|\right)^2<\infty$, we obtain
$$\sin^2(p^i)\left(\int_{\mathbb T_i^3}\frac{\rmd q}{\lambda+\kappa\theta(p+q)}\right)\left(\int_{\mathbb T_i^3}\frac{\rmd q}{\theta(p+q)}\right)\le\frac{C}{\kappa}b(1+|\log b|)^2\le\frac{C}{\kappa}.$$
This proves the lemma.
\end{proof}

\begin{proposition}\label{prop:hc.form}
Assume that $R$ is a positive diagonal operator whose multipliers $(\mathfrak r^{(n)})_{n\ge 1}$ satisfy \eqref{eq:R.momentum} and \eqref{eq:R.res}. Then there exists $C<\infty$, depending only on $C_R$, such that, for every $n\ge 1$ and $f\in\ell^2_0(U,3,n)$,
\begin{equation}\label{eq:hc.form}
\left|\left\langle f, A_-R A_+f\right\rangle_{\ell^2(U,3,n)}-\left(V_{\rm Diag}(R,f)+V_{\rm Off}(R,f)\right)\right|\le\frac{Cn}{\kappa}\langle f,(-S)f\rangle_{\ell^2(U,3,n)} .
\end{equation}
\end{proposition}

\begin{proof}
Throughout the proof, $C$ denotes a finite constant depending only on $C_R$. Its value may change from line to line. For simplicity, throughout the proof, we write $\langle\cdot,\cdot\rangle$ for the scalar product on $\ell^2(U,3,m)$, where the degree $m$ is determined by the arguments and may vary from one occurrence to another.
We first prove that for every $i\in[3]$,
\begin{equation}
\label{eq:Bi.square}(n+1)\langle B_if,RB_if\rangle\le\frac{Cn}{\kappa}\langle f,(-S)f\rangle .
\end{equation}
Fix $i\in[3]$.  For $f\in\ell^2_0(U,3,n)$, we note that $\mathcal B_i^{(n)}f= \mathbf 1_{\left\{z_{1:n}:\#\{r\in[n]:z_r=(\mathbf0,i)\}=1\right\}}f$. Since $\tau_i$ and $\tau_{-i}$ preserve $E_n$, we thus have  $B_if=\mathbf 1_{\left\{z_{1:n+1}:\#\{r\in[n+1]:z_r=(\mathbf0,i)\}=2\right\}}B_if.$
Since $B_if$ is symmetric and $R$ preserves symmetry, we have
\begin{align*}
\langle B_if,RB_if\rangle = \binom{n+1}{2} \sum_{j_{1:n-1}\in[3]^{n-1}} \sum_{x_{1:n-1}\in U^{n-1}} (B_if)_{j_{1:n-1},i,i} (x_{1:n-1},\mathbf0,\mathbf0) (RB_if)_{j_{1:n-1},i,i} (x_{1:n-1},\mathbf0,\mathbf0).\end{align*}
Note that
\begin{align*}
(B_if)_{j_{1:n-1},i,i} (x_{1:n-1},\mathbf0,\mathbf0) &= \frac{1}{n+1}\Big[f_{j_{1:n-1},i}\big((x_\ell-e_i+e_{j_\ell})_{1\le \ell\le n-1},\mathbf0\big)\nonumber\\
&\quad - f_{j_{1:n-1},i} \big((x_\ell+e_i-e_{j_\ell})_{1\le \ell\le n-1},\mathbf0\big)\Big].
\end{align*}
Let $g_{j_{1:n-1}}(x_{1:n-1}):=f_{j_{1:n-1},i}(x_{1:n-1},\mathbf0)$ and $P:=\sum_{\ell=1}^{n-1}p_\ell.$ Taking the Fourier transform only in the variables $x_1,\ldots,x_{n-1}$, we get 
\begin{align} \label{eq:trace.fourier}
  \sum_{x_{1:n-1}\in U^{n-1}} (B_if)_{j_{1:n-1},i,i} (x_{1:n-1},\mathbf0,\mathbf0) \exp\left(\mathbf i\sum_{\ell=1}^{n-1}p_\ell\cdot x_\ell\right)= \frac{1}{n+1} \left( {\rm e}^{\mathbf iP^i} - {\rm e}^{-\mathbf iP^i} \right) \widehat g_{j_{1:n-1}}(p_{1:n-1}).
\end{align}
Using the Fourier representation of $R$,
\begin{align*}&(n+1)\langle B_if,RB_if\rangle
\\
&\quad\le C n \sum_{j_{1:n-1}\in [3]^{n-1}} \int_{\mathbb{T}^3_{j_{1:n-1}}} \sin^2(P^i) \left( \int_{\mathbb T_i^3} \int_{\mathbb T_i^3} r_{n+1}(P+q+w)\rmd q\rmd w\right) |\widehat g_{j_{1:n-1}}(p_{1:n-1})|^2\rmd p_{1:n-1}.
\end{align*}

By translation invariance of Haar measure on $\mathbb T_i^3$,
$$
\int_{\mathbb T_i^3}\int_{\mathbb T_i^3} r_{n+1}(P+q+w)\rmd q\rmd w=|\mathbb T_i^3|\int_{\mathbb T_i^3} r_{n+1}(P+w)\rmd w .
$$
Using \eqref{eq:R.res}, we obtain
\begin{equation}
\label{eq:Bi.form}
(n+1)\langle B_if,RB_if\rangle \le C n \sum_{j_{1:n-1}\in[3]^{n-1}} \int_{\mathbb T^3_{j_{1:n-1}}} \sin^2(P^i) \left( \int_{\mathbb T_i^3} \frac{\rmd w}{\lambda+\kappa\theta(P+w)} \right) |\widehat g_{j_{1:n-1}}(p_{1:n-1})|^2\rmd p_{1:n-1}.
\end{equation}
On the other hand, by Fourier inversion formula, we have
\begin{align*}
    \widehat g_{j_{1:n-1}}(p_{1:n-1}) =\frac{1}{(2\pi)^2} \int_{\mathbb T_i^3} \widehat f_{j_{1:n-1},i}(p_{1:n-1},q)\rmd q.
\end{align*}
Therefore, by the Cauchy–Schwarz inequality with weight $\theta(P+q)$, 
\begin{equation}
\label{eq:weighted.cauchy}
|\widehat g_{j_{1:n-1}}(p_{1:n-1})|^2 \le C\left(\int_{\mathbb T_i^3} |\widehat f_{j_{1:n-1},i}(p_{1:n-1},q)|^2 \theta(P+q)\rmd q\right)\left( \int_{\mathbb T_i^3}\frac{\rmd q}{\theta(P+q)}\right).
\end{equation}
If $\sin(P^i/2)=0$, then the integrand is zero. When $\sin(P^i/2)\neq0$,
by Lemma~\ref{lem:trace.estimate}, we have
$$
\sin^2(P^i)\left(\int_{\mathbb T_i^3}\frac{\rmd w}{\lambda+\kappa\theta(P+w)}\right)\left(\int_{\mathbb T_i^3}\frac{\rmd q}{\theta(P+q)}\right)\le\frac{C}{\kappa}.
$$
Thus,
\begin{align*}
    \sin^2(P^i)\left(\int_{\mathbb T_i^3}\frac{\rmd w}{\lambda+\kappa\theta(P+w)}\right)|\widehat g_{j_{1:n-1}}(p_{1:n-1})|^2 \le\frac{C}{\kappa}\int_{\mathbb T_i^3}|\widehat f_{j_{1:n-1},i}(p_{1:n-1},q)|^2\theta(P+q)\rmd q.
\end{align*}
Substituting this into
\eqref{eq:Bi.form}, we get
$$
(n+1)\langle B_if,RB_if\rangle\le\frac{C n}{\kappa}\sum_{j_{1:n-1}\in[3]^{n-1}}\int_{\mathbb T^3_{j_{1:n-1}}}\int_{\mathbb T_i^3}|\widehat f_{j_{1:n-1},i}(p_{1:n-1},q)|^2\theta(P+q)\rmd q\rmd p_{1:n-1}.
$$
By Lemma \ref{lem:Fourier} and the Parseval-Plancherel formula,
$$\sum_{j_{1:n-1}\in[3]^{n-1}}\int_{\mathbb T^3_{j_{1:n-1}}\times\mathbb T_i^3}|\widehat f_{j_{1:n-1},i}(p_{1:n-1},q)|^2\theta(P+q)\rmd p_{1:n-1}\rmd q\le C\langle f,(-S)f\rangle.$$
This implies \eqref{eq:Bi.square}.

Since $R$ is positive and $B=\sum_iB_i$, we have
$$
\left\langle f, B^*R Bf\right\rangle=(n+1)\left\langle\sum_{i=1}^3B_if, R\sum_{i=1}^3B_if\right\rangle\le 3(n+1)\sum_{i=1}^3\langle B_if,RB_if\rangle ,
$$
and therefore \eqref{eq:Bi.square} implies that
\begin{align}
    \label{blocked.part} \left\langle f, B^*R Bf\right\rangle \le \frac{Cn}{\kappa} \langle f,(-S)f\rangle.
\end{align}

We next show that, for every $f\in\ell^2_0(U,3,n) $,
\begin{align}\label{mixed.part}
\big|\langle f, (\widetilde A_+)^*RBf\rangle\big| \le\frac{C n}{\kappa}\langle f,(-S)f\rangle.
\end{align}
For $h\in[3]$ and $1\le a<b\le n+1$, define
$$ D_{a,b}^{(h)}=\left\{z_{1:n+1}:\, z_a=z_b,\ z_a\in U\times\{h\}\right\},\quad D^{(h)}=\bigcup_{1\le a<b\le n+1}D_{a,b}^{(h)}.$$
We first prove that
\begin{equation}
\label{eq:RBh.support}
RB_hf=\mathbf 1_{D^{(h)}}RB_hf .
\end{equation}
Note that
\begin{equation}
\label{eq:Bh.decomp}
B_hf=\sum_{1\le a<b\le n+1}\mathbf 1_{\{z_{1:n+1}:z_a=z_b=(\mathbf0,h)\}}B_hf .
\end{equation}
Let $F$ be a function such that, for all $j_{1:n+1}\in[3]^{n+1}$ and all $x_{1:n+1}$,
\begin{equation*}
F_{j_{1:n+1}}(x_{1:n+1}) =\mathbf 1_{\{j_a=j_b=h\}}\mathbf 1_{\{x_a=x_b\}} F_{j_{1:n+1}}(x_{1:n+1}).
\end{equation*}
We show that $RF$ satisfies the same identity. If $j_a\neq h$ or $j_b\neq h$, then $F_{j_{1:n+1}}=0$ and $(RF)_{j_{1:n+1}}=0$. If $j_a=j_b=h$, then  $F_{j_{1:n+1}}(x_{1:n+1})=\mathbf 1_{\{x_a=x_b\}} F_{j_{1:n+1}}(x_{1:n+1})$ and thus
\begin{align*}
\widehat F_{j_{1:n+1}}(p_{1:n+1})= &\sum_{x_{1:n+1\setminus\{a,b\}}\in U^{n-1}, y\in U} F_{j_{1:n+1}}(x_1,\ldots,x_{a-1},y,x_{a+1},\ldots,x_{b-1},y,x_{b+1},\ldots,x_{n+1})\\
&\times\exp\left(\mathbf i\sum_{\ell\neq a,b}p_\ell\cdot x_\ell+\mathbf i(p_a+p_b)\cdot y\right).
\end{align*}
Hence $\widehat F_{j_{1:n+1}}$ depends on $p_a$ and $p_b$ only through $p_a+p_b$ in this case. On the other hand, by \eqref{eq:R.momentum}, we have
$$\widehat{RF}_{j_{1:n+1}}(p_{1:n+1})=r_{n+1}\left(\sum_{\ell=1}^{n+1}p_\ell\right)\widehat F_{j_{1:n+1}}(p_{1:n+1}).$$
The right-hand side also depends on $p_a$ and $p_b$ only through $p_a+p_b$. Therefore, by Fourier inversion and orthogonality of characters on $\mathbb T_h^3$,
$$(RF)_{j_{1:n+1}}(x_{1:n+1})=\mathbf 1_{\{x_a=x_b\}}(RF)_{j_{1:n+1}}(x_{1:n+1}).$$
Combining this with the case $j_a\neq h$ or $j_b\neq h$, we deduce that
$$(RF)_{j_{1:n+1}}(x_{1:n+1})=\mathbf 1_{\{j_a=j_b=h\}}\mathbf 1_{\{x_a=x_b\}}(RF)_{j_{1:n+1}}(x_{1:n+1}).$$
Applying this to each summand in \eqref{eq:Bh.decomp}, we obtain \eqref{eq:RBh.support}.

Next we prove that, for every $i,h\in[3]$,
\begin{equation}
\label{eq:Atilde.B}
\mathbf 1_{D^{(h)}}\widetilde A_{+,i}f=\mathbf 1_{\{i=h\}}B_if .
\end{equation}
For every $z_{1:n+1}$, recall that
\begin{align}
\label{eq:Atilde.sum}
(\widetilde A_{+,i}f)(z_{1:n+1})&=\frac{1}{2(n+1)}\sum_{r=1}^{n+1}\mathbf 1_{\{z_r=(\mathbf0,i)\}}\left[(\tau_i-\tau_{-i})f\right](z_{1:n+1\setminus\{r\}}).\\
\label{eq:Bi.physical.sum} 
(B_if)(z_{1:n+1})&=\frac{1}{2(n+1)}\sum_{r=1}^{n+1}\mathbf 1_{\{z_r=(\mathbf0,i)\}}\mathbf 1_{\left\{\#\{s\in[n+1]\setminus\{r\}:z_s=(\mathbf0,i)\}\ge1\right\}}\left[(\tau_i-\tau_{-i})f\right](z_{1:n+1\setminus\{r\}}).
\end{align}
Let $z_{1:n+1}\in D^{(h)}$. Choose $a<b$ and $x\in U$ such that $z_a=z_b=(x,h)$. Suppose first that $i\neq h$. Then $z_{1:n+1\setminus\{r\}}$ contains the two entries equal to $(x,h)$. Since $f\in\ell^2_0(U,3,n)$, we have $(\tau_{\pm i}f)(z_{1:n+1\setminus\{r\}})=0$. Hence
$$\mathbf 1_{D^{(h)}}\widetilde A_{+,i}f=0\quad\text{for } i\neq h.$$
Now assume $i=h$. Let $z_{1:n+1}\in D^{(i)}$. In \eqref{eq:Atilde.sum}, the $r$-th summand is zero unless $z_r=(\mathbf0,i)$. Suppose $z_r=(\mathbf0,i)$. If $\#\{s\in[n+1]\setminus\{r\}:z_s=(\mathbf0,i)\}\ge1$, then the $r$-th summand in \eqref{eq:Atilde.sum} is exactly the $r$-th summand in \eqref{eq:Bi.physical.sum}. It remains to consider the case $\#\{s\in[n+1]\setminus\{r\}:z_s=(\mathbf0,i)\}=0.$
Since $z_{1:n+1}\in D^{(i)}$, there exist $a<b$ and $x\in U$ such that $z_a=z_b=(x,i)$ with $r\notin\{a,b\}$. Hence $z_{1:n+1\setminus\{r\}}$ contains the two entries equal to $(x,i)$.
Since $f\in\ell^2_0(U,3,n)$, we have $(\tau_{\pm i}f)(z_{1:n+1\setminus\{r\}})=0$. Therefore, on $D^{(i)}$, the nonzero summands in $\widetilde A_{+,i}f$ are exactly the summands in $B_if$. Hence
$$
\mathbf 1_{D^{(i)}}\widetilde A_{+,i}f=B_if.
$$
Combining this with the case $i\neq h$ proves
\eqref{eq:Atilde.B}.

Using \eqref{eq:RBh.support} and
\eqref{eq:Atilde.B}, we have
$$\left\langle\widetilde A_{+,i}f, RB_hf\right\rangle=\left\langle\mathbf 1_{D^{(h)}}\widetilde A_{+,i}f, RB_hf\right\rangle=\mathbf 1_{\{i=h\}}\langle B_if,RB_if\rangle.$$
Therefore
\begin{align}
\left\langle f, (\widetilde A_+)^*RBf\right\rangle&=(n+1)\sum_{i,h=1}^3\left\langle\widetilde A_{+,i}f, RB_hf\right\rangle=(n+1)\sum_{i=1}^3\langle B_if,RB_if\rangle.
\end{align}
Using \eqref{eq:Bi.square},
$$\left|\left\langle f, (\widetilde A_+)^*RBf\right\rangle\right|\le(n+1)\sum_{i=1}^3\langle B_if,RB_if\rangle\le\frac{C n}{\kappa}\langle f,(-S)f\rangle.$$

Recall that
$A_+=\widetilde A_+-B$ and $A_+^*=A_-$. Therefore
\begin{align}\label{A.diff}
&\left\langle f, A_-R A_+f\right\rangle-\left\langle f, (\widetilde A_+)^*R\widetilde A_+f\right\rangle =-\left\langle f, (\widetilde A_+)^*R Bf\right\rangle-\left\langle f, B^*R \widetilde A_+f\right\rangle +\left\langle f, B^*R Bf\right\rangle.
\end{align}
Since $R$ is self-adjoint, the second mixed term is the complex conjugate
of the first one. Hence, combining \eqref{A.diff} with  \eqref{blocked.part} and \eqref{mixed.part}, we obtain
$$\left|\left\langle f,A_-R A_+f\right\rangle-\left\langle f,(\widetilde A_+)^*R\widetilde A_+f\right\rangle\right|\le\frac{Cn}{\kappa}\langle f,(-S)f\rangle.$$
By the definition of $V_{\rm Diag}$ and $V_{\rm Off}$, we have
$\left\langle f,(\widetilde A_+)^*R\widetilde A_+f\right\rangle=V_{\rm Diag}(R,f)+V_{\rm Off}(R,f)$. This implies
\eqref{eq:hc.form}.
\end{proof}

\section{Bounds for the mean squared displacement}
\label{sec:bounds.msd}
The purpose of this section is to establish upper and lower bounds for the recursive operators $(\Lambda_m)_{m\ge1}$ defined in \eqref{eq:def.Lambda}, and then to use these bounds to prove Theorem~\ref{thm.main}. Motivated by the recursive comparison approach in \cite{Y2004} and \cite{CHT2022}, we first introduce diagonal comparison operators $(T_m)_{m\ge1}$ and the associated resolvents $(R_m)_{m\ge1}$ through suitable Fourier multipliers. We then prove the integral estimates required to control the diagonal and off-diagonal parts of the quadratic forms arising from $A_-R_m A_+$. Combining these estimates with Proposition~\ref{prop:hc.form}, we prove in Proposition~\ref{prop:recur.bounds} that, up to a multiplicative coefficient $c_m$, the operator $\Lambda_m$ is bounded from below by $T_m$ when $m$ is odd and from above by $T_m$ when $m$ is even.  Substituting these bounds into the truncated resolvent representation \eqref{eq:schur.rep}, we estimate the resulting first-chaos Fourier integrals and complete the proof of Theorem~\ref{thm.main}.

Throughout this section, we assume that $d=3$ and fix $\lambda\in(0,1)$. For simplicity of notation, we write $\langle\cdot,\cdot\rangle$ for the scalar product on $\ell^2(U,3,k)$, where the degree $k$ is determined by the arguments and may vary from one occurrence to another.

\subsection{Comparison multipliers}
\label{subsec:fourier.multipliers}

In this subsection, we define the diagonal Fourier multipliers used to compare $\Lambda_m$. For $n\ge1$, $j\in[3]^n$, $p=(p_1,\ldots,p_n)\in\mathbb T_j^3$, and $k\in[3]$, define
\begin{align}
\label{eq:theta.sum.Theta}
\theta_k(p)&:=2\sin^2\Big(\frac{1}2\sum_{\ell=1}^n p_\ell^k\Big)\quad\text{and} \quad \theta(p):=\sum_{k=1}^3\theta_k(p)=2\sum_{k=1}^3\sin^2\Big(\frac{1}2 \sum_{\ell=1}^n p_\ell^k\Big).
\end{align}
We use the following convention for degree $n=0$. The set $[3]^0$ consists of the empty index $\emptyset$, $\mathbb T_{\emptyset}^3$ consists of one point, and for the unique $p\in\mathbb T_{\emptyset}^3$, set $\theta(p)=0$ and $\theta_k(p)=0$ for each $k\in[3]$.

For $x>0$ and $z>1$, define
\begin{equation}
\label{eq:def.L}
L(x,z):=z+\log(1+x^{-1}),
\end{equation}
and, for $m\ge0$,
\begin{align}
\label{eq:def.LM}
{\rm LB}_m(x,z)&:=\sum_{r=0}^{m}\frac{1}{r!}\left(\frac12\log L(x,z)\right)^r, \quad {\rm UB}_m(x,z):=\frac{L(x,z)}{{\rm LB}_m(x,z)}.
\end{align}
The functions ${\rm LB}_m$ and ${\rm UB}_m$ are studied in Appendix A of \cite{CHT2022}.

We now define the lower and upper comparison multipliers:
\begin{align}\label{def:lower.multiplier}
\mathfrak l_m(\lambda,p,z)&:=\frac12\sum_{k=1}^3{\rm LB}_m(\lambda+\theta_k(p),z)\sin^2\Big(\sum_{\ell=1}^n p_\ell^k\Big),
\\ \label{def.up.multiplier}
\mathfrak u_m(\lambda,p,z) &:=\frac12\sum_{k=1}^3{\rm UB}_m(\lambda+\theta_k(p),z)\sin^2\Big(\sum_{\ell=1}^n p_\ell^k\Big).
\end{align}
By positivity of ${\rm LB}_m$, ${\rm UB}_m$, we note that $\mathfrak l_m(\lambda,p,z)\ge0$ and $\mathfrak u_m(\lambda,p,z)\ge0$.

Fix $\eps\in (0,1)$. Also, let
$K_1, K_2, K_{\rm err}\in (1,\infty)$ be sufficiently large constants such that $K_1\ge4K_{\rm err}$. Define
\begin{align}
\label{eq:def.zmn}
z_m(n)&:=K_2(n+m)^{2(2+\eps)}\quad\text{for $m\ge1$ and $n\ge0$},\\
\label{eq:def.eps.m} \varepsilon_m&:=(m+1)^{-1-\eps}\quad \text{for $m\ge1$}.
\end{align} 
Let $(c_m)_{m\ge 1}$ be a sequence of positive numbers such that
\begin{align}\label{eq:def.cm}
c_1:=\frac{1}{2K_{\rm err}}\quad \text{and for $m\ge1$,}\quad c_{2m}:=\frac{2 }{\pi c_{2m-1}}(1+\varepsilon_m),\quad c_{2m+1}:=\frac{2 }{\pi c_{2m}}(1-\varepsilon_m).
\end{align}

For each $m,n\ge1$, let $T_m: \ell^2_0(U,3,n) \to \ell^2_0(U,3,n)$ be the diagonal operator with Fourier multiplier $\mathfrak s_n^{(m)}(\lambda,p)$, where
\begin{equation}
\label{eq:def.s.one}
\mathfrak s_n^{(1)}(\lambda,p):=0,
\end{equation}
and for $m\ge1$,
\begin{align}
\label{eq:def.s.even}
\mathfrak s_n^{(2m)}(\lambda,p)&:= K_1\sqrt{z_{2m}(n)}\mathfrak u_{m-1}(\lambda,p,z_{2m}(n))+K_{\rm err}\sqrt{z_{2m}(n)}\,\theta(p),
\\ \label{eq:def.s.odd}
\mathfrak s_n^{(2m+1)}(\lambda,p)&:=\frac{1}{K_1\sqrt{z_{2m+1}(n)}}\mathfrak l_m(\lambda,p,z_{2m+1}(n))-K_{\rm err}\theta(p).
\end{align}
Thus, for every $f\in \ell_0^2(U,3,n)$, 
\begin{equation}
\label{eq:T.def}
\widehat{T_mf}_j(p) = \mathfrak s_n^{(m)}(\lambda,p)\widehat f_j(p)\quad\text{for } j\in [3]^n,\; p\in\mathbb T_j^3.
\end{equation}

The even comparison operators are non-negative since  
\begin{equation}
\label{eq:even.nonnegative}
\mathfrak s_n^{(2m)}(\lambda,p)\ge0\quad\text{for } m\ge1.
\end{equation}
The odd comparison multipliers are not necessarily non-negative.

By \eqref{eq:def.cm}, we notice that, for each $m\ge 0$,
\begin{equation}
\label{eq:odd.condition}
0<c_{2m+1}\le c_1=\frac{1}{2K_{\rm err}}.
\end{equation}
By \eqref{eq:def.s.odd} and using \eqref{eq:odd.condition} together with the fact that $\mathfrak l_r(\lambda,p,z)\ge 0$, we have
\begin{align}\nonumber
\lambda+\theta(p)+c_{2m+1}\mathfrak s_n^{(2m+1)}(\lambda,p) &= \lambda+\left(1-c_{2m+1}K_{\rm err}\right)\theta(p)+\frac{c_{2m+1}}{K_1\sqrt{z_{2m+1}(n)}}\mathfrak l_m(\lambda,p,z_{2m+1}(n))\\
\label{eq:odd.denominator.lower} &\ge \lambda+\frac12\theta(p).
\end{align}

Therefore, for every $m,n\ge1$, the operators $\lambda-S+c_{2m+1}T_{2m+1}$ are strictly positive on $\ell^2_0(U,3,n)$. Also, by \eqref{eq:even.nonnegative}, the operators $\lambda-S+c_{2m}T_{2m}$ are strictly positive on $\ell^2_0(U,3,n)$. Define the comparison resolvent
\begin{equation}
\label{eq:def.Rm}
R_m:= \left( \lambda-S+c_mT_m \right)^{-1},
\end{equation}
which is also a positive self-adjoint operator. Moreover, $R_m$ is diagonal with Fourier multiplier
\begin{equation}
\label{eq:Rm.multiplier}
r_m^{(n)}(\lambda,p) := \frac{1} {\lambda+\theta(p)+c_m\mathfrak s_{n}^{(m)}(\lambda,p)}.
\end{equation}

By \eqref{eq:even.nonnegative} and
\eqref{eq:Rm.multiplier}, we have
\begin{equation*}
0\le r_{2m}^{(n)}(\lambda,p)  \le \frac{1}{\lambda+\theta(p)} \quad \text{for } m\ge 1.
\end{equation*}
By \eqref{eq:odd.denominator.lower} and
\eqref{eq:Rm.multiplier}, we also have
\begin{equation*} 0\le r_{2m+1}^{(n)}(\lambda,p) \le \frac{1}{\lambda+\frac12\theta(p)} \le \frac{2}{\lambda+\theta(p)}\quad\text{for } m\ge 0.
\end{equation*}
Therefore, for every $m,n\ge1$, $j\in [3]^n$ and $p\in\mathbb T^3_j$,
\begin{equation}
\label{eq:R.bound}
0\le r_m^{(n)}(\lambda,p) \le \frac{2}{\lambda+\theta(p)}.
\end{equation}

\subsection{Integral estimates for the multipliers}
\label{subsec:integral.estimates}

In this subsection, we prove integral estimates for the comparison multipliers $\mathfrak l_m$ and $\mathfrak u_m$ defined by \eqref{def:lower.multiplier}-\eqref{def.up.multiplier}. These estimates will be used to establish bounds for the recursive comparison resolvents $(R_m)_{m\ge 1}$ in Section~\ref{sec:compr.res}. Throughout this subsection, we assume that $d=3$ and fix $\lambda\in(0,1)$.

The following lemma provides useful bounds for the functions ${\rm UB}_m$ and ${\rm LB}_m$ defined by \eqref{eq:def.LM}. 

\begin{lemma}
\label{lem:ULB} The following properties hold for the functions $L$, ${\rm LB}_m$ and ${\rm UB}_m$:
\begin{itemize}
    \item[i.] the functions $L$, ${\rm LB}_m$, and ${\rm UB}_m$ are non-increasing in the first variable, and
\begin{equation}
\label{eq:ULB.bounds}
1\le {\rm LB}_m(x,z)\le \sqrt{L(x,z)}\le {\rm UB}_m(x,z)\le L(x,z)\quad\text{for } x>0,\ z>1;
\end{equation}
\item[ii.] for every $0<\alpha<1$,
\begin{align}
\label{eq:LB.b}
\int_\alpha^1 \frac{\rmd r}{r{\rm LB}_m(r,z)} &\le 2{\rm UB}_m(\alpha,z) + \frac{{\rm UB}_m(\alpha,z)}{z}\quad\text{and}\\
\label{eq:UB.b}
\int_\alpha^\beta \frac{\rmd r}{r{\rm UB}_m(r,z)} &\ge 2\left( {\rm LB}_{m+1}(\alpha,z) - {\rm LB}_{m+1}(\beta,z) \right)\quad \text{for every $0<\alpha<\beta\le1$};
\end{align}
\item[iii.] for every fixed $c_0>0$, there exists $C<\infty$ such that, for all $\gamma\in[c_0z^{-1},1]$, $r\in(0,1)$, $m\ge0$, and $z>1$,
\begin{align} \label{eq:comp.LB}
{\rm LB}_m(\gamma r,z) &\le {\rm LB}_m(r,z) \left[ 1+ C\frac{(m+1)(1+\log z)}{z} \right],
\\
\label{eq:comp.UB}
{\rm UB}_m(\gamma r,z) &\le {\rm UB}_m(r,z) \left[ 1+ C\frac{(m+1)(1+\log z)}{z} \right];
\end{align}
\item[iv.] for each fixed $B>1$ there exists $C_{B}<\infty$ such that,
whenever $s,t\in(0,1)$, $m\ge0$, $z>1$, and $B^{-1}s\le t\le Bs$,
\begin{equation}
\label{eq:diag.local.LB}
C_{B}^{-1}{\rm LB}_m(s,z) \le {\rm LB}_m(t,z) \le C_{B}{\rm LB}_m(s,z);
\end{equation}
\item[v.] for every fixed $B\in(0,\infty)$, there exists $C_B<\infty$ such that, for all $r\in(0,1)$, $m\ge0$, and $z>1$,
\begin{equation}
\label{eq:diag.LB}
{\rm LB}_{m+1}(Br,z) \ge {\rm LB}_{m+1}(r,z) \left[ 1- C_B\frac{m+1}{z} \right] - C_B{\rm LB}_{m+1}(1,z).
\end{equation}
\end{itemize}
\end{lemma}

\begin{proof}
By \cite[Lemma A.1]{CHT2022}, we have \eqref{eq:ULB.bounds}, and that the functions $L$, ${\rm LB}_m$, and ${\rm UB}_m$ are non-increasing in the first variable, and that for every $0<\alpha<\beta$,
\begin{align}
\label{eq:ULB.primitive.UB}
\int_\alpha^\beta\frac{\rmd r} {(r^2+r){\rm UB}_m(r,z)}&=2\left({\rm LB}_{m+1}(\alpha,z)-{\rm LB}_{m+1}(\beta,z)\right),\\
\label{eq:ULB.primitive.LB}
\int_\alpha^\beta\frac{\rmd r} {(r^2+r){\rm LB}_m(r,z)}&\le 2\left( {\rm UB}_m(\alpha,z)- {\rm UB}_m(\beta,z)\right).
\end{align}
By \cite[Lemma A.5]{CHT2022}, we also have that, for every $0<\alpha<1$,
\begin{equation}
\label{eq:ULB.rho.replacement}
0\le\int_\alpha^1\frac{\rmd r}{r{\rm LB}_m(r,z)}-\int_\alpha^1\frac{\rmd r}{(r^2+r){\rm LB}_m(r,z)}\le\frac{{\rm UB}_m(\alpha,z)}{z}.
\end{equation}
Combining
\eqref{eq:ULB.primitive.LB} with
\eqref{eq:ULB.rho.replacement}, we have
$$
\int_\alpha^1 \frac{\rmd r}{r{\rm LB}_m(r,z)}
\le 2{\rm UB}_m(\alpha,z)-2{\rm UB}_m(1,z) +\frac{{\rm UB}_m(\alpha,z)}{z}.
$$
Using the fact that ${\rm UB}_m(1,z)\ge0$, we obtain \eqref{eq:LB.b}.

Using \eqref{eq:ULB.primitive.UB} and the fact that $r^2+r\ge r$ for $r>0$, we have
$$
\int_\alpha^\beta \frac{\rmd r}{r{\rm UB}_m(r,z)}
\ge \int_\alpha^\beta \frac{\rmd r}{(r^2+r){\rm UB}_m(r,z)}
= 2\left( {\rm LB}_{m+1}(\alpha,z) - {\rm LB}_{m+1}(\beta,z) \right).
$$
This proves \eqref{eq:UB.b}.

We next prove \eqref{eq:comp.LB}-\eqref{eq:comp.UB}. For $\gamma\in[c_0z^{-1},1]$, we note that
$$
0\le L(\gamma r,z)-L(r,z)
= \log\left( \frac{1+\gamma^{-1}r^{-1}}{1+r^{-1}} \right)
\le \log(\gamma^{-1}) \le C(1+\log z).
$$
Since $L(r,z)\ge z$, it follows that
$$
0 \le \frac12\log L(\gamma r,z) - \frac12\log L(r,z)
= \frac12 \log\left( \frac{L(\gamma r,z)}{L(r,z)} \right)
\le C\frac{1+\log z}{z}.
$$
For $u\ge0$,
$\frac{d}{du}\sum_{\ell=0}^{m}\frac{u^\ell}{\ell!} = \sum_{\ell=0}^{m-1}\frac{u^\ell}{\ell!} \le \sum_{\ell=0}^{m}\frac{u^\ell}{\ell!}.
$
Therefore
$$ {\rm LB}_m(\gamma r,z) \le {\rm LB}_m(r,z) \exp\left( C\frac{1+\log z}{z}\right).
$$
This implies \eqref{eq:comp.LB}. Moreover,
$$
\frac{L(\gamma r,z)}{L(r,z)} \le 1+ C\frac{1+\log z}{z}, \quad {\rm LB}_m(\gamma r,z)\ge {\rm LB}_m(r,z),
$$
since $\gamma r\le r$ and $x\mapsto{\rm LB}_m(x,z)$ is non-increasing. Hence
$$
{\rm UB}_m(\gamma r,z) = \frac{L(\gamma r,z)}{{\rm LB}_m(\gamma r,z)} \le {\rm UB}_m(r,z) \left[ 1+ C\frac{1+\log z}{z} \right].
$$
This proves \eqref{eq:comp.UB}.

We next prove \eqref{eq:diag.local.LB}.
For $s,t\in(0,1)$ and $B^{-1}s\le t\le Bs$, we have
$$
\left| L(t,z)-L(s,z) \right| = \left| \log(1+t^{-1})-\log(1+s^{-1}) \right| \le C_{B}.
$$
Since $L(s,z),L(t,z)\ge z$, we obtain
$$
\left| \frac12\log L(t,z) - \frac12\log L(s,z) \right|
\le \frac{|L(t,z)-L(s,z)|}{2\min\{L(t,z),L(s,z)\}} \le C_{B}z^{-1}.$$
Using again
$
\frac{d}{du} \sum_{\ell=0}^{m}\frac{u^\ell}{\ell!} \le \sum_{\ell=0}^{m}\frac{u^\ell}{\ell!},
$
we get
$$
{\rm LB}_m(t,z)\le \exp(C_Bz^{-1}){\rm LB}_m(s,z)\le {\rm e}^{C_B}{\rm LB}_m(s,z).
$$
The reverse inequality follows by interchanging $s$ and $t$. This proves
\eqref{eq:diag.local.LB}.

We now prove \eqref{eq:diag.LB}. If $B\le1$, then $
{\rm LB}_{m+1}(Br,z)\ge{\rm LB}_{m+1}(r,z).
$
Thus \eqref{eq:diag.LB} holds in this case. Assume $B>1$. If $Br\le1$, then
$
0\le L(r,z)-L(Br,z) = \log\left( \frac{1+r^{-1}}{1+B^{-1}r^{-1}} \right) \le \log B.
$
Hence
$$
0 \le \frac12\log L(r,z) - \frac12\log L(Br,z) \le C_Bz^{-1}.
$$
Using $\frac{\rmd}{\rmd u}\sum_{\ell=0}^{m+1}\frac{u^\ell}{\ell!} \le \sum_{\ell=0}^{m+1}\frac{u^\ell}{\ell!}$ for $u\ge 0$, we get
$$
{\rm LB}_{m+1}(Br,z) \ge {\rm LB}_{m+1}(r,z) \exp(-C_Bz^{-1}) \ge {\rm LB}_{m+1}(r,z) \left[ 1- C_B\frac{m+1}{z} \right].
$$
If $Br>1$, then $r>B^{-1}$, and hence
$
L(r,z) \le z+\log(1+B) \le C_B L(1,z).
$
Therefore
$
{\rm LB}_{m+1}(r,z) \le C_B{\rm LB}_{m+1}(1,z).
$
For sufficiently large $C_B$, the right-hand side of \eqref{eq:diag.LB} is non-positive. Hence \eqref{eq:diag.LB} also holds when $Br>1$.

\end{proof}

The next lemma controls the integral appearing in $V_{\mathrm{Diag}}(R_m,f)$.
\begin{lemma}
\label{lem:diagonal.integrals}
There exists a constant $C<\infty$ such that, for every $m\ge0$, $z>1$, $\lambda\in(0,1)$, $n\ge0$, $j\in[3]^n$, $p\in\mathbb T_j^3$, and $k\in[3]$,
\begin{align}
\label{eq:diagonal.upper.integral}
\int_{\mathbb T_k^3} \frac{\rmd q} {\lambda+\theta(p,q)+\mathfrak l_m(\lambda,(p,q),z)} &\le 4\pi \left[ 1+ C\left( z^{-1/2} + \frac{(m+1)(1+\log z)}{z} \right) \right] \nonumber\\ 
&\quad\times {\rm UB}_m(\lambda+\theta_k(p),z)+C, \\
\label{eq:diagonal.lower.integral} \int_{\mathbb T_k^3} \frac{\rmd q} {\lambda+\theta(p,q)+\mathfrak u_m(\lambda,(p,q),z)} &\ge 4\pi \left[
1- C\left( z^{-1/2} + \frac{(m+1)(1+\log z)}{z} \right) \right] \nonumber\\
&\quad\times{\rm LB}_{m+1}(\lambda+\theta_k(p),z)- C{\rm LB}_{m+1}(1,z)-C.
\end{align}
Here, when $n=0$, we use the degree-zero convention from
Section~\ref{subsec:fourier.multipliers}.
\end{lemma}

\begin{proof}
It is enough to prove the estimates
for $k=3$. Let $q\in\mathbb T_3^3$, and set
$$
x:=\sum_{\ell=1}^n p_\ell^1+q^1,
\quad
y:=\sum_{\ell=1}^n p_\ell^2+q^2,
\quad
a:=\lambda+\theta_3(p), \quad \rho
:=
2\sin^2\left(\frac{x}{2}\right)
+
2\sin^2\left(\frac{y}{2}\right).
$$
Since $q^3=0$, we have $\lambda+\theta(p,q)=a+\rho$. Fix
$\rho_0\in(0,1/2)$. Since
$
\sin^2x+\sin^2y
=
2\rho
-
4\sin^4\left(\frac{x}{2}\right)
-
4\sin^4\left(\frac{y}{2}\right),
$
whenever $\rho<\rho_0$, we have
\begin{equation}
\label{eq:diag.sin.rho}
\rho \le \sin^2x+\sin^2y\le 2\rho
\quad\text{and}\quad
0\le
\rho-\frac12(\sin^2x+\sin^2y)
\le
\frac12\rho^2.
\end{equation}

We first prove the upper bound. Since $\mathfrak l_m\ge0$, on $\{\rho\ge\rho_0\}$, we have
$
{\lambda+\theta(p,q)+\mathfrak l_m(\lambda,(p,q),z)}
\ge {a+\rho} \ge {\rho_0}.
$
Therefore
\begin{equation}
\label{eq:diag.upper.rho.large}
\int_{\{\rho\ge\rho_0\}}
\frac{\rmd q}
{\lambda+\theta(p,q)+\mathfrak l_m(\lambda,(p,q),z)}
\le C.
\end{equation}

We now consider the integral on $\{\rho<\rho_0\}$. For every non-negative continuous function $F$ on $[0,\rho_0]$, by the changes of variables
$
u=\sqrt2\sin(x/2),\ v=\sqrt2\sin(y/2)
$
and then
$
u=\sqrt r\cos\varphi,\ v=\sqrt r\sin\varphi
$, we have
\begin{equation}
\label{eq:diag.coarea}
\int_{\{\rho<\rho_0\}}F(\rho)\rmd q
=
\int_0^{\rho_0}J(r)F(r)\rmd r,
\end{equation}
where
\begin{equation}
\label{eq:diag.coarea.density}
J(r)
:=
\int_0^{2\pi}
\frac{\rmd\varphi}
{
\sqrt{1-\frac r2\cos^2\varphi}
\sqrt{1-\frac r2\sin^2\varphi}
}
=
2\pi+O(r)
\quad\text{for } 0<r<\rho_0.
\end{equation}

For $h=1,2$, we note that
$\lambda+\theta_h(p,q)\le a+\rho$.
Since $x\mapsto{\rm LB}_m(x,z)$ is non-increasing, we have
$
{\rm LB}_m(\lambda+\theta_h(p,q),z)
\ge
{\rm LB}_m(a+\rho,z)$.
Hence
\begin{align*}
\mathfrak l_m(\lambda,(p,q),z)
&=
\frac12
\sum_{h=1}^3
{\rm LB}_m(\lambda+\theta_h(p,q),z)\sin^2\Big(\sum_{\ell=1}^n p_\ell^h+q^h\Big)
\ge
\frac12(\sin^2x+\sin^2y)
{\rm LB}_m(a+\rho,z).
\end{align*}
Combining this with \eqref{eq:diag.sin.rho}, we thus have
$
\mathfrak l_m(\lambda,(p,q),z)
\ge
\left(\rho-\frac12\rho^2\right)
{\rm LB}_m(a+\rho,z)
$
when $\rho<\rho_0$. Therefore,
\begin{align*}
\lambda+\theta(p,q)+\mathfrak l_m(\lambda,(p,q),z)
&\ge
a+\rho+
\left(\rho-\frac12\rho^2\right)
{\rm LB}_m(a+\rho,z)
\ge
\left(1-\frac{\rho}{2}\right)
\left[
a+\rho+\rho{\rm LB}_m(a+\rho,z)
\right].
\end{align*}
Since $\rho<\rho_0<1/2$, we have
$
(1-\rho/2)^{-1}\le1+C\rho.
$
Using this inequality together with
\eqref{eq:diag.coarea}-\eqref{eq:diag.coarea.density}, we obtain
\begin{align*}
\int_{\{\rho<\rho_0\}}
\frac{\rmd q}
{\lambda+\theta(p,q)+\mathfrak l_m(\lambda,(p,q),z)}
&\le
\int_0^{\rho_0}
\frac{2\pi+Cr}
{a+r+r{\rm LB}_m(a+r,z)}
\rmd r\\\le
&2\pi
\int_0^{\rho_0}
\frac{\rmd r}
{a+r+r{\rm LB}_m(a+r,z)}
+
C,
\end{align*}
where in the last inequality we used
$
a+r+r{\rm LB}_m(a+r,z)\ge r.
$
Combining this with \eqref{eq:diag.upper.rho.large}, we get
\begin{equation}
\label{eq:diag.upper.radial}
\int_{\mathbb T_3^3}
\frac{\rmd q}
{\lambda+\theta(p,q)+\mathfrak l_m(\lambda,(p,q),z)}
\le
2\pi
\int_0^{\rho_0}
\frac{\rmd r}
{a+r+r{\rm LB}_m(a+r,z)}
+
C.
\end{equation}

If $a\ge\rho_0$, then
$$
\int_0^{\rho_0}
\frac{\rmd r}
{a+r+r{\rm LB}_m(a+r,z)}
\le
\int_0^{\rho_0}\frac{\rmd r}{a}
\le1.
$$
Since ${\rm UB}_m(a,z)\ge1$, the upper bound follows. Assume now that
$0<a<\rho_0$. Changing variables $s=a+r$, we obtain
$$
\int_0^{\rho_0}
\frac{\rmd r}
{a+r+r{\rm LB}_m(a+r,z)}
=
\int_a^{a+\rho_0}
\frac{\rmd s}
{s+(s-a){\rm LB}_m(s,z)}.
$$
For $s\in[a,2a]$,
$$
\int_a^{2a}
\frac{\rmd s}
{s+(s-a){\rm LB}_m(s,z)}
\le
\int_a^{2a}\frac{\rmd s}{s}
=
\log2.
$$
For $s\ge2a$, we have $s-a\ge s/2$, and hence
\begin{align*}
\frac{1}
{s+(s-a){\rm LB}_m(s,z)}
&\le
\frac{1}{(s-a){\rm LB}_m(s,z)}
=
\frac{1}{s{\rm LB}_m(s,z)}
+
\frac{a}{s(s-a){\rm LB}_m(s,z)}
\\
&\le
\frac{1}{s{\rm LB}_m(s,z)}
+
\frac{2a}{s^2{\rm LB}_m(s,z)}.
\end{align*}
Since ${\rm LB}_m(s,z)\ge1$,
$$
\int_{2a}^{a+\rho_0}
\frac{a\rmd s}{s^2{\rm LB}_m(s,z)}
\le C.
$$
Consequently,
$$
\int_0^{\rho_0}
\frac{\rmd r}
{a+r+r{\rm LB}_m(a+r,z)}
\le
\int_a^1
\frac{\rmd s}{s{\rm LB}_m(s,z)}
+
C.
$$
By Lemma~\ref{lem:ULB}.ii,
$$
\int_a^1
\frac{\rmd s}{s{\rm LB}_m(s,z)}
\le
2{\rm UB}_m(a,z)
+
\frac{{\rm UB}_m(a,z)}{z}.
$$
Therefore
\begin{align*}
\int_{\mathbb T_3^3}
\frac{\rmd q}
{\lambda+\theta(p,q)+\mathfrak l_m(\lambda,(p,q),z)}
&\le
4\pi
\left(1+\frac{C}{z}\right)
{\rm UB}_m(a,z)
+
C.
\end{align*}
This implies \eqref{eq:diagonal.upper.integral}.

We now prove the lower bound. If $a\ge\rho_0/4$, then by the monotonicity
of ${\rm LB}_{m+1}$ when $a\ge1$, and
Lemma~\ref{lem:ULB}.iv, applied with
$B=4/\rho_0$ and by continuity, when $\rho_0/4\le a<1$, we have
$
{\rm LB}_{m+1}(a,z)
\le
C{\rm LB}_{m+1}(1,z).
$
Hence, for sufficiently large $C$, the right-hand side of
\eqref{eq:diagonal.lower.integral} is non-positive, and the lower bound
follows. Assume from now on that $a<\rho_0/4$.
We restrict the integral to the set
$$
D:=\Big\{q\in\mathbb T_3^3 :a\le \rho<\rho_0,\;  
2\sin^2\left(\frac{x}{2}\right)\ge (4z)^{-1}\rho,\;  2\sin^2\left(\frac{y}{2}\right)\ge (4z)^{-1}\rho\Big\}.
$$
On $\{\rho<\rho_0\}$, put
$
u=\sqrt2\sin\left(\frac{x}{2}\right),
v=\sqrt2\sin\left(\frac{y}{2}\right).
$
Then $\rho=u^2+v^2$. Writing
$
u=r^{1/2}\cos\varphi
$
and
$
v=r^{1/2}\sin\varphi,
$
the two angular conditions in the definition of $D$ become
$
\cos^2\varphi\ge(4z)^{-1}
$
and
$
\sin^2\varphi\ge(4z)^{-1}.
$
The complement of this set in $[0,2\pi)$ has Lebesgue measure
$
8\arcsin\left((4z)^{-1/2}\right)\le \frac{4\pi}{3}z^{-1/2}.
$
Since 
$\left( 1-\frac r2\cos^2\varphi\right)^{-1/2}
\left(1-\frac r2\sin^2\varphi\right)^{-1/2}\ge 1$, for every non-negative continuous function $F$,
\begin{align}
\int_DF(\rho)\rmd q
&\ge
2\pi
\left(
1-\frac{2}{3}z^{-1/2}
\right)
\int_a^{\rho_0}F(r)\rmd r.
\label{eq:diag.restricted.coarea}
\end{align}

Since $a\le\rho$ on $D$, for $h=1,2$ we have
$
\frac{1}{8z}
\le
\frac{\lambda+\theta_h(p,q)}{a+\rho}
\le1.
$
Note that $0<a+\rho<5\rho_0/4<1$. Therefore, applying Lemma~\ref{lem:ULB}.iii with
$c_0=1/8$,
$\gamma=\frac{\lambda+\theta_h(p,q)}{a+\rho}$ and $r=a+\rho$,
we obtain
$$
{\rm UB}_m(\lambda+\theta_h(p,q),z)
\le
{\rm UB}_m(a+\rho,z)
\left[
1+
C\frac{(m+1)(1+\log z)}{z}
\right]\quad\text{for } h=1,2.
$$
Furthermore,
\begin{align*}
    \frac12\sum_{h=1}^2\sin^2\Big(\sum_{\ell=1}^n p_\ell^h+q^h\Big)
&=
\frac12(\sin^2x+\sin^2y)
\le
\rho,\quad
\frac12\sin^2\Big(\sum_{\ell=1}^n p_\ell^3\Big)
\le
\theta_3(p)
\le a.
\end{align*}
Since $\lambda+\theta_3(p)=a$, we obtain
\begin{align}
\mathfrak u_m(\lambda,(p,q),z)
&=
\frac12
\sum_{h=1}^3
{\rm UB}_m(\lambda+\theta_h(p,q),z)\sin^2\Big(\sum_{\ell=1}^n p_\ell^h+q^h\Big)
\nonumber\\
&\le
\rho{\rm UB}_m(a+\rho,z)
\left[
1+
C\frac{(m+1)(1+\log z)}{z}
\right]
+
a{\rm UB}_m(a,z).
\label{eq:diag.u.upper.angular}
\end{align}
Therefore
\begin{align*}
\lambda+\theta(p,q)+\mathfrak u_m(\lambda,(p,q),z)
&\le
a+\rho
+
\rho{\rm UB}_m(a+\rho,z)
+
a{\rm UB}_m(a,z)
\\
&\quad+
C\frac{(m+1)(1+\log z)}{z}
\rho{\rm UB}_m(a+\rho,z)
\\
&\le
\left[
1+
C\frac{(m+1)(1+\log z)}{z}
\right]
\left[
a+\rho
+
\rho{\rm UB}_m(a+\rho,z)
+
a{\rm UB}_m(a,z)
\right].
\end{align*}
Applying \eqref{eq:diag.restricted.coarea} with
$
F(r)=
\left(
a+r+r{\rm UB}_m(a+r,z)+a{\rm UB}_m(a,z)
\right)^{-1},
$
we obtain
\begin{align}
\label{eq:diag.lower.radial}
\int_{\mathbb T_3^3}
\frac{\rmd q}
{\lambda+\theta(p,q)+\mathfrak u_m(\lambda,(p,q),z)}
&\ge
\frac{
2\pi\left(1-\frac{2}{3}z^{-1/2}\right)
}
{1+C\frac{(m+1)(1+\log z)}{z}}
\int_a^{\rho_0}
\frac{\rmd r}
{a+r+r{\rm UB}_m(a+r,z)+a{\rm UB}_m(a,z)}.
\end{align}
Recall that $a<\rho_0/4<\rho_0/2$. Changing variables
$s=a+r$, we obtain
\begin{align}
\int_a^{\rho_0}
\frac{\rmd r}
{a+r+r{\rm UB}_m(a+r,z)+a{\rm UB}_m(a,z)}
&=
\int_{2a}^{a+\rho_0}
\frac{\rmd s}
{s+(s-a){\rm UB}_m(s,z)+a{\rm UB}_m(a,z)}
\nonumber\\
&\ge
\int_{2a}^{\rho_0}
\frac{\rmd s}
{s+(s-a){\rm UB}_m(s,z)+a{\rm UB}_m(a,z)}.
\label{eq:diag.radial.change.s}
\end{align}
For $s\ge2a$, since $x\mapsto{\rm UB}_m(x,z)$ is non-increasing,
we have ${\rm UB}_m(a,z)\ge{\rm UB}_m(s,z)$ and
$s+(s-a){\rm UB}_m(s,z)+a{\rm UB}_m(a,z)\ge s{\rm UB}_m(s,z)$.
Therefore,
\begin{align*}
&\frac{1}{s{\rm UB}_m(s,z)}
-
\frac{1}
{s+(s-a){\rm UB}_m(s,z)+a{\rm UB}_m(a,z)}
\\
&=
\frac{
s+(s-a){\rm UB}_m(s,z)+a{\rm UB}_m(a,z)
-
s{\rm UB}_m(s,z)
}
{
s{\rm UB}_m(s,z)
\left[
s+(s-a){\rm UB}_m(s,z)+a{\rm UB}_m(a,z)
\right]
}
\le
\frac{s+a{\rm UB}_m(a,z)}
{s^2{\rm UB}_m(s,z)^2}.
\end{align*}
Therefore
\begin{align}\nonumber
\int_{2a}^{\rho_0}
\frac{\rmd s}
{s+(s-a){\rm UB}_m(s,z)+a{\rm UB}_m(a,z)}
&\ge
\int_{2a}^{\rho_0}\frac{\rmd s}{s{\rm UB}_m(s,z)}
- \int_{2a}^{\rho_0}\frac{\rmd s}{s{\rm UB}_m(s,z)^2}\\
&\quad -
a{\rm UB}_m(a,z)
\int_{2a}^{\rho_0}
\frac{\rmd s}
{s^2{\rm UB}_m(s,z)^2}.
\label{eq:diag.radial.error.decomposition}
\end{align}
By Lemma~\ref{lem:ULB}.i, we have
$
{\rm UB}_m(s,z)\ge\sqrt{L(s,z)}\ge\sqrt z.
$
Hence
\begin{equation}
\label{eq:diag.radial.first.error}
\int_{2a}^{\rho_0}
\frac{\rmd s}
{s{\rm UB}_m(s,z)^2}
\le
z^{-1/2}
\int_{2a}^{\rho_0}
\frac{\rmd s}
{s{\rm UB}_m(s,z)}.
\end{equation}
We next prove
\begin{equation}
\label{eq:diag.radial.second.error}
a{\rm UB}_m(a,z)
\int_{2a}^{\rho_0}
\frac{\rmd s}
{s^2{\rm UB}_m(s,z)^2}
\le C.
\end{equation}
Decompose $[2a,\rho_0]$ into dyadic intervals
$[2^\ell a,2^{\ell+1}a]$, with the last interval truncated at $\rho_0$. Since $x\mapsto{\rm UB}_m(x,z)$ is non-increasing,
for $s\in[2^\ell a,2^{\ell+1}a]$, we have
$
{\rm UB}_m(s,z)\ge{\rm UB}_m(2^{\ell+1}a,z).
$
Thus
\begin{align*}
&a{\rm UB}_m(a,z)
\int_{2^\ell a}^{2^{\ell+1}a}
\frac{\rmd s}{s^2{\rm UB}_m(s,z)^2}
\le
C2^{-\ell}
\frac{{\rm UB}_m(a,z)}
{{\rm UB}_m(2^{\ell+1}a,z)^2}
\le
C2^{-\ell}
\frac{{\rm UB}_m(a,z)}
{{\rm UB}_m(2^{\ell+1}a,z)}.
\end{align*}
Since ${\rm UB}_m=L/{\rm LB}_m$ and $x\mapsto{\rm LB}_m(x,z)$ is
non-increasing,
$$
\frac{{\rm UB}_m(a,z)}
{{\rm UB}_m(2^{\ell+1}a,z)}
=
\frac{L(a,z)}{L(2^{\ell+1}a,z)}
\frac{{\rm LB}_m(2^{\ell+1}a,z)}
{{\rm LB}_m(a,z)}
\le
\frac{L(a,z)}{L(2^{\ell+1}a,z)}.
$$
Moreover,
$
0\le
L(a,z)-L(2^{\ell+1}a,z)
\le
(\ell+1)\log2.
$
Consequently,
\begin{align*}
\frac{L(a,z)}{L(2^{\ell+1}a,z)}
&=
1+
\frac{L(a,z)-L(2^{\ell+1}a,z)}
{L(2^{\ell+1}a,z)}\le
1+
\frac{(\ell+1)\log2}{z}.
\end{align*}
Therefore
$$
\sum_{\ell\ge1}
2^{-\ell}
\frac{L(a,z)}{L(2^{\ell+1}a,z)}
\le
\sum_{\ell\ge1}
2^{-\ell}
\left(1+\frac{(\ell+1)\log2}{z} \right)
\le C.
$$
This proves \eqref{eq:diag.radial.second.error}. 

Combining \eqref{eq:diag.radial.change.s}, \eqref{eq:diag.radial.error.decomposition}, \eqref{eq:diag.radial.first.error}, and \eqref{eq:diag.radial.second.error}, we get
\begin{align}
\label{eq:diag.am.error}
\int_a^{\rho_0} \frac{\rmd r}{a+r+r{\rm UB}_m(a+r,z)+a{\rm UB}_m(a,z)} &\ge \left(1-z^{-1/2}\right) \int_{2a}^{\rho_0}\frac{\rmd s}{s{\rm UB}_m(s,z)}-C. \end{align}
By Lemma~\ref{lem:ULB}.ii,
$$
\int_{2a}^{\rho_0} \frac{\rmd s}{s{\rm UB}_m(s,z)} \ge 2\left( {\rm LB}_{m+1}(2a,z)- {\rm LB}_{m+1}(\rho_0,z) \right).
$$
By Lemma~\ref{lem:ULB}.v, applied with $B=2$,
$$ {\rm LB}_{m+1}(2a,z) \ge {\rm LB}_{m+1}(a,z)\left[1-C\frac{m+1}{z}\right]-C{\rm LB}_{m+1}(1,z).
$$
Moreover, by Lemma~\ref{lem:ULB}.iv, applied with $B=1/\rho_0$ and by continuity, we have
${\rm LB}_{m+1}(\rho_0,z)\le C{\rm LB}_{m+1}(1,z).$
Combining these estimates with \eqref{eq:diag.am.error}, we obtain
\begin{align}
&\int_a^{\rho_0}
\frac{\rmd r} {a+r+r{\rm UB}_m(a+r,z)+a{\rm UB}_m(a,z)} \nonumber\\ &\quad\ge 2\left[ 1- C\left( z^{-1/2}+\frac{m+1}{z} \right) \right] {\rm LB}_{m+1}(a,z) - C{\rm LB}_{m+1}(1,z)-C.
\label{eq:diag.radial.lower.primitive}
\end{align}
Combining \eqref{eq:diag.lower.radial} and \eqref{eq:diag.radial.lower.primitive}, using $(1+x)^{-1}\ge1-x$ for $x\ge0$, we get
\begin{align*}
\int_{\mathbb T_3^3} \frac{\rmd q} {\lambda+\theta(p,q)+\mathfrak u_m(\lambda,(p,q),z)} &\ge 4\pi\left[1-C\left(z^{-1/2}+\frac{(m+1)(1+\log z)}{z}\right)\right]{\rm LB}_{m+1}(a,z)\nonumber\\
&\quad -C{\rm LB}_{m+1}(1,z) - C,
\end{align*}
for sufficiently large $C$. Recalling that $a=\lambda+\theta_3(p)$, the lower bound follows for $k=3$. The cases $k=1,2$ follow by symmetry. The proof is complete.
\end{proof}

The next lemma controls the integral appearing in $V_{\mathrm{Off}}(R_m,f)$.

\begin{lemma}
\label{lem:schur.offdiag}
There exists a constant $C<\infty$ such that, for every $m\ge1$, $n\ge1$, $z>1$, $\lambda\in(0,1)$ and $f\in \ell_0^2(U,3,n)$, we have
\begin{align}
\label{eq:two.schur.lower}
&\sum_{j\in[3]^n}\sum_{k=1}^3 \int_{\mathbb T_j^3}\int_{\mathbb T_k^3}\frac{ \Big|\sin\Big(\sum_{\ell=1}^n p_\ell^k\Big)\Big|\cdot\Big|\sin\Big(\sum_{\ell=1}^{n-1}p_\ell^{j_n}+q^{j_n}\Big)\Big|} {\lambda+\theta(p,q)+K_1^{-1}z^{-1/2}\mathfrak l_{m-1}(\lambda,(p,q),z)}|\widehat f_j(p)||\widehat f_{j_{1:n-1},k}(p_{1:n-1},q)| \rmd q\rmd p\nonumber\\
&\quad\le C\sum_{j\in[3]^n}\int_{\mathbb T_j^3}\mathfrak u_{m-1}(\lambda,p,z) |\widehat f_j(p)|^2\rmd p,
\\ \label{eq:two.schur.upper}
&\sum_{j\in[3]^n}\sum_{k=1}^3\int_{\mathbb T_j^3}\int_{\mathbb T_k^3}\frac{\Big|\sin\Big(\sum_{\ell=1}^n p_\ell^k\Big)\Big|\cdot\Big|\sin\Big(\sum_{\ell=1}^{n-1}p_\ell^{j_n}+q^{j_n}\Big)\Big|}{\lambda+\theta(p,q)+K_1c_{2m}\sqrt z\,\mathfrak u_{m-1}(\lambda,(p,q),z)}|\widehat f_j(p)||\widehat f_{j_{1:n-1},k}(p_{1:n-1},q)| \rmd q\rmd p
\nonumber\\ &\quad\le \frac{C}{K_1c_{2m}z}\sum_{j\in[3]^n}\int_{\mathbb T_j^3}\theta(p)|\widehat f_j(p)|^2\rmd p.
\end{align}
\end{lemma}

\begin{proof}
Set
$$
W(s):=\sum_{h=1}^3\sin^2(s^h)
\quad\text{for } s=(s^1, s^2, s^3)\in\mathbb T^3.
$$
We first prove that there exists $C<\infty$ such that, for every
$P\in\mathbb T^3$, $i,k\in[3]$, and $r\in\mathbb T_i^3$,
\begin{equation}
\label{eq.weighted.fiber}
\left|\sin(P^k+r^k)\right|
\int_{\mathbb T_k^3}
\frac{
\left|\sin(P^i+q^i)\right|
}
{
W(P+q+r)W(P+q)
}
\rmd q
\le C.
\end{equation}
For fixed $P\in\mathbb T^3$ and $r\in\mathbb T_i^3$, the set
$\{q\in\mathbb T_k^3 : W(P+q) W(P+q+r)=0\}$ is finite. Removing this
finite set does not change the integral in
\eqref{eq.weighted.fiber}; all fractions below are therefore
considered outside their sets of zero denominators.

Let $\Omega:=\{0,\pi\}^3$. For $x,y\in\mathbb T$, set
$
d_{\mathbb T}(x,y)
:=
\min\{|t|:t\in\mathbb R,\ t+2\pi\mathbb Z=x-y\}.
$
Fix $\eta\in(0,\pi/8)$ and, for every $\omega\in\Omega$, set
$$
U_\omega
:=
\left\{
s\in\mathbb T^3:
d_{\mathbb T}(s^h,\omega^h)<\eta
\text{ for every }h\in[3]
\right\}.
$$
The sets $(U_\omega)_{\omega\in\Omega}$ are pairwise disjoint. Indeed,
if $s\in U_\omega\cap U_{\omega'}$ for distinct
$\omega,\omega'\in\Omega$, then, for some $h\in[3]$,
\begin{align*}
\pi
&=
d_{\mathbb T}(\omega^h,(\omega')^h)
\le
d_{\mathbb T}(\omega^h,s^h)
+
d_{\mathbb T}(s^h,(\omega')^h)
<2\eta<\pi,
\end{align*}
which is impossible.

Since $\omega^h\in\{0,\pi\}$, we note that
$|\sin(s^h)|=\sin(d_{\mathbb T}(s^h,\omega^h))$ for every
$s\in U_\omega$ and $h\in[3]$. Since $d_{\mathbb T}(s^h,\omega^h)<\eta<\pi/2$, we have
$
\frac{2}{\pi}d_{\mathbb T}(s^h,\omega^h)
\le
|\sin(s^h)|
\le
d_{\mathbb T}(s^h,\omega^h).
$
Consequently, there exist $c>0$ and $C<\infty$ such that
\begin{equation}
\label{eq:W.local}
c\sum_{h=1}^3d_{\mathbb T}(s^h,\omega^h)^2
\le
W(s)
\le
C\sum_{h=1}^3d_{\mathbb T}(s^h,\omega^h)^2,
\quad
s\in U_\omega.
\end{equation}
Moreover, $W$ is bounded from below by a positive constant on
$
\mathbb T^3
\setminus
\bigcup_{\omega\in\Omega}U_\omega,
$
since this set is compact and contains no zero of $W$.

Since $q^k=0$ and $r^i=0$, we have
$\sin(P^k+r^k)=\sin(P^k+q^k+r^k)$ and
$\sin(P^i+q^i)=\sin(P^i+q^i+r^i)$.
Consequently, $\left|
\sin(P^k+r^k)
\sin(P^i+q^i)
\right|
\le
W(P+q+r)$.
Since $W(P+q)\ge c$ whenever
$
P+q\notin\bigcup_{\omega\in\Omega}U_\omega,
$
we obtain
\begin{align}\label{eq:dom1}
&\left|\sin(P^k+r^k)\right|
\int_{\{q\in\mathbb T_k^3:\,P+q\notin
\bigcup_{\omega\in\Omega}U_\omega\}}
\frac{
\left|\sin(P^i+q^i)\right|
}
{
W(P+q+r)W(P+q)
}
\rmd q
\le C.
\end{align}

Fix $\omega\in\Omega$. We next integrate over the set
$
\{q\in\mathbb T_k^3:P+q\in U_\omega,\,
P+q+r\notin\bigcup_{\omega'\in\Omega}U_{\omega'}\}.
$
On this set, $W(P+q+r)$ is bounded from below by a positive constant. Suppose first that $i\ne k$, and let $\ell$ be the unique element of
$[3]\setminus\{i,k\}$. Let $\alpha,u,v\in(-\eta,\eta)$ be determined
by
$$
P^k\equiv\omega^k+\alpha,\quad
P^i+q^i\equiv\omega^i+u,\quad
P^\ell+q^\ell\equiv\omega^\ell+v
\pmod{2\pi}.
$$
By
\eqref{eq:W.local},
$$
\frac{|\sin(P^i+q^i)|}{W(P+q)}
\le
\frac{C|u|}{\alpha^2+u^2+v^2}.
$$
The change from $q$ to $(u,v)$ preserves Lebesgue measure. Therefore, the integral in
\eqref{eq.weighted.fiber} over this set is bounded by
\begin{align*}
C\int_{[-\eta,\eta]^2}
\frac{|u|}{\alpha^2+u^2+v^2}
\rmd u\rmd v
&\le
C\int_{[-\eta,\eta]^2}
\frac{\rmd u\rmd v}{(u^2+v^2)^{1/2}}
\le 4\pi C\eta.
\end{align*}
Suppose now that $i=k$, and let $\{h,\ell\}=[3]\setminus\{i\}$. Let $\alpha\in(-\eta,\eta)$ and $u,v\in(-\eta,\eta)$ be determined by
$$
P^i\equiv\omega^i+\alpha, \quad P^h+q^h\equiv\omega^h+u,
\quad
P^\ell+q^\ell\equiv\omega^\ell+v\pmod{2\pi}.$$
Since $q^i=r^i=0$, we have
$
\left|
\sin(P^k+r^k)
\sin(P^i+q^i)
\right|
=
\sin^2(P^i)
\le
\alpha^2.
$
Therefore, for $\alpha\ne0$, the integral in
\eqref{eq.weighted.fiber} over this set is bounded by
$$
C\alpha^2
\int_{[-\eta,\eta]^2}
\frac{\rmd u\rmd v}{\alpha^2+u^2+v^2}
\le
C\alpha^2
\int_0^{2\eta}
\frac{t\rmd t}{\alpha^2+t^2}
=
\frac{C\alpha^2}{2}
\log\left(1+\frac{4\eta^2}{\alpha^2}\right)
\le 2C\eta^2.
$$
When $\alpha=0$, the original integral is zero.

It remains to integrate over the sets
$
\{q\in\mathbb T_k^3:P+q\in U_\omega,\,
P+q+r\in U_{\omega'}\}$ for some $\omega,\omega'\in\Omega$.
Suppose first that $i\ne k$, and let $\ell$ be the unique element of
$[3]\setminus\{i,k\}$. Since $r^i=0$ and $\eta<\pi/8$, this set is
empty unless $\omega^i=(\omega')^i$. Assume that
$\omega^i=(\omega')^i$. Let
$\alpha,\beta,u,v\in(-\eta,\eta)$ be determined by
$$
P^k\equiv\omega^k+\alpha,\quad
P^k+r^k\equiv(\omega')^k+\beta,\quad
P^i+q^i\equiv\omega^i+u,\quad
P^\ell+q^\ell\equiv\omega^\ell+v
\pmod{2\pi}.
$$
Since the set under consideration is nonempty, the congruence
$
b\equiv(\omega')^\ell-\omega^\ell-r^\ell
\pmod{2\pi}
$
has a solution $b\in(-2\eta,2\eta)$. This solution is unique as
$4\eta<2\pi$. For every $q$ in this set, we have
$
P^\ell+q^\ell+r^\ell
\equiv
(\omega')^\ell+v-b
\pmod{2\pi},
$
and the condition $P+q+r\in U_{\omega'}$ implies
$v-b\in(-\eta,\eta)$. By
\eqref{eq:W.local},
\begin{align*}
&W(P+q)
\ge
c\left(\alpha^2+u^2+v^2\right),
\quad
W(P+q+r)
\ge
c\left(\beta^2+u^2+(v-b)^2\right),
\\
&|\sin(P^k+r^k)|
\le
C|\beta|,
\quad
|\sin(P^i+q^i)|
\le
C|u|.
\end{align*}
If $\beta=0$, then $\sin(P^k+r^k)=0$. If $\beta\ne0$, the integral
over this set is bounded by
\begin{align*}
&C|\beta|
\int_{\mathbb R^2}
\frac{|u|\rmd u\rmd v}
{
(\alpha^2+u^2+v^2)
(\beta^2+u^2+(v-b)^2)
}
\\
&=
C|\beta|
\int_0^1
\int_{\mathbb R^2}
\frac{|u|\rmd u\rmd v\rmd s}
{
\left[
u^2+(v-(1-s)b)^2
+s\alpha^2+(1-s)\beta^2+s(1-s)b^2
\right]^2
}
\\
&=
C|\beta|
\int_0^1\frac{\rmd s}{\left[s\alpha^2+(1-s)\beta^2+s(1-s)b^2\right]^{1/2}} \le C|\beta|\int_0^1\frac{\rmd s}{|\beta|(1-s)^{1/2}}=2C,
\end{align*}
where first equality follows from the fact that
$$
\frac1{AB}
=
\int_0^1
\frac{\rmd s}{[sA+(1-s)B]^2}
\quad\text{for } A,B>0,
$$ and the second equality
follows from
$$
\int_{\mathbb R^2}
\frac{|u|\rmd u\rmd v}{(u^2+v^2+a^2)^2}
=
\frac{\pi}{a}
\quad\text{for } a>0.
$$
Suppose now that $i=k$. Since $q^i=r^i=0$, the set under
consideration is empty unless $\omega^i=(\omega')^i$. Assume that
$\omega^i=(\omega')^i$. Let $\alpha\in(-\eta,\eta)$ and $y\in(-\eta,\eta)^2$ be determined by
$$
P^i\equiv\omega^i+\alpha\pmod{2\pi},\quad 
P^h+q^h\equiv\omega^h+y^h\pmod{2\pi}
\quad\text{for } h\in[3]\setminus\{i\}.
$$
Since the set under consideration is nonempty, for each
$h\in[3]\setminus\{i\}$, the congruence
$
b^h\equiv(\omega')^h-\omega^h-r^h\pmod{2\pi}
$
has a solution $b^h\in(-2\eta,2\eta)$. This solution is unique as
$4\eta<2\pi$. Thus $b\in(-2\eta,2\eta)^2$ is independent of $q$, and
$$
P^h+q^h+r^h
\equiv
(\omega')^h+y^h-b^h
\pmod{2\pi}\quad \text{for } h\in[3]\setminus\{i\}.
$$
Since $P+q+r\in U_{\omega'}$, we also have
$y-b\in(-\eta,\eta)^2$. If
$\alpha=0$, then
$\sin(P^k+r^k)\sin(P^i+q^i)=\sin^2(P^i)=0$.
If $\alpha\ne0$, the
integral over this set is bounded by
\begin{align*}
C\alpha^2
\int_{\mathbb R^2}
\frac{\rmd y}
{
(\alpha^2+|y|^2)
(\alpha^2+|y-b|^2)
}&=
C\alpha^2
\int_0^1
\int_{\mathbb R^2}
\frac{\rmd y\rmd s}
{
\left[
|y-(1-s)b|^2
+\alpha^2+s(1-s)|b|^2
\right]^2
}
\\
&=
C\pi\alpha^2
\int_0^1
\frac{\rmd s}{\alpha^2+s(1-s)|b|^2}
\le C\pi.
\end{align*}
Since $|\Omega|=8$, this proves
\eqref{eq.weighted.fiber}.

We now prove \eqref{eq:two.schur.lower}. Fix
$j=(j_1,\ldots,j_n)\in[3]^n$ and $k\in[3]$, and set
$P:=\sum_{\ell=1}^{n-1}p_\ell$, $i:=j_n$. For fixed
$(p_1,\ldots,p_{n-1})$, denote by $K(p_n,q)$ the fraction in the
integrand on the left-hand side of
\eqref{eq:two.schur.lower}. Since
$\mathfrak l_{m-1}\ge0$,
$\theta(p,q)=\theta(P+p_n+q)$, and
$W(P+p_n+q)\le2\theta(p,q)$, we have
\begin{align*}
\frac{K(p_n,q)}{W(P+q)}
&\le
2
\frac{
|\sin(P^k+p_n^k)|
|\sin(P^i+q^i)|
}
{W(P+p_n+q)W(P+q)},
\quad
\frac{K(p_n,q)}{W(P+p_n)}
\le
2
\frac{
|\sin(P^k+p_n^k)|
|\sin(P^i+q^i)|
}
{W(P+p_n+q)W(P+p_n)}.
\end{align*}
Therefore, applying \eqref{eq.weighted.fiber}, first with
$(i,k,r)=(j_n,k,p_n)$ and then with
$(i,k,r)=(k,j_n,q)$, we obtain
\begin{align}
\label{eq:schur.lower.row}
\sup_{p_n\in\mathbb T_i^3}
\int_{\mathbb T_k^3}
\frac{K(p_n,q)}{W(P+q)}
\rmd q
&\le C,
\quad
\sup_{q\in\mathbb T_k^3}
\int_{\mathbb T_i^3}
\frac{K(p_n,q)}{W(P+p_n)}
\rmd p_n
\le C.
\end{align}
The quotients in \eqref{eq:schur.lower.row} are defined to be zero where
their additional denominators vanish. Indeed, if $W(P+p_n)=0$, then
$\sin(P^k+p_n^k)=0$, and if $W(P+q)=0$, then
$\sin(P^i+q^i)=0$.

For $W(P+p_n)W(P+q)>0$, we have
\begin{align*}
&K(p_n,q)
|\widehat f_j(p)|
|\widehat f_{j_{1:n-1},k}(p_{1:n-1},q)|
\\
&=
\left[
K(p_n,q)
\frac{W(P+p_n)}{W(P+q)}
|\widehat f_j(p)|^2
\right]^{1/2}
\left[
K(p_n,q)
\frac{W(P+q)}{W(P+p_n)}
|\widehat f_{j_{1:n-1},k}(p_{1:n-1},q)|^2
\right]^{1/2}.
\end{align*}
For $W(P+p_n)W(P+q)=0$, the left-hand side is zero. Applying the
Cauchy--Schwarz inequality and
\eqref{eq:schur.lower.row}, we obtain
\begin{align*}
&\int_{\mathbb T_j^3}
\int_{\mathbb T_k^3}
K(p_n,q)
|\widehat f_j(p)|
|\widehat f_{j_{1:n-1},k}(p_{1:n-1},q)|
\rmd q\rmd p
\\
&\le
C
\left(
\int_{\mathbb T_j^3}
W\Big(\sum_{\ell=1}^n p_\ell\Big)
|\widehat f_j(p)|^2
\rmd p
\right)^{1/2}
\left(
\int_{\mathbb T_{j_{1:n-1},k}^3}
W\Big(\sum_{\ell=1}^n p_\ell\Big)
|\widehat f_{j_{1:n-1},k}(p)|^2
\rmd p
\right)^{1/2}.
\end{align*}
Using $2ab\le a^2+b^2$, summing over $j\in[3]^n$ and $k\in[3]$, and
using that each sequence in $[3]^n$ has exactly three preimages under
$(j_{1:n},k)\mapsto(j_{1:n-1},k)$, we obtain
\begin{align}
\label{eq:schur1}
&\sum_{j\in[3]^n}\sum_{k=1}^3
\int_{\mathbb T_j^3}
\int_{\mathbb T_k^3}
K(p_n,q)
|\widehat f_j(p)|
|\widehat f_{j_{1:n-1},k}(p_{1:n-1},q)|
\rmd q\rmd p
\le
3C
\sum_{j\in[3]^n}
\int_{\mathbb T_j^3}
W\Big(\sum_{\ell=1}^n p_\ell\Big)
|\widehat f_j(p)|^2
\rmd p.
\end{align}
By Lemma~\ref{lem:ULB}.i, ${\rm UB}_{m-1}(x,z)\ge1$, and hence
$
\mathfrak u_{m-1}(\lambda,p,z)
\ge
\frac12
W\left(\sum_{\ell=1}^n p_\ell\right).
$
Combining this with \eqref{eq:schur1} proves
\eqref{eq:two.schur.lower}.

We next prove \eqref{eq:two.schur.upper}. By
Lemma~\ref{lem:ULB}.i, since $L(x,z)\ge z$, we have
$
{\rm UB}_{m-1}(x,z)\ge\sqrt z.
$
Therefore,
\begin{align*}
&\lambda+\theta(p,q)
+
K_1c_{2m}\sqrt z\,
\mathfrak u_{m-1}(\lambda,(p,q),z)
\ge
\lambda+\theta(p,q)
+
\frac{K_1c_{2m}z}{2}
W\Big(\sum_{\ell=1}^n p_\ell+q\Big).
\end{align*}
Fix $Z>0$ and set 
$$K_Z(p_n,q):=\frac{|\sin(P^k+p_n^k)|
|\sin(P^i+q^i)|}{\lambda+\theta(p,q)
+
Z
W\left(\sum_{\ell=1}^n p_\ell+q\right)}.$$
Since $W(s)\le2\theta(s)$ for $s\in\mathbb T^3$, we have
\begin{align*}
\frac{K_Z(p_n,q)}{\theta(P+q)}
&\le \frac{2}{Z}\frac{|\sin(P^k+p_n^k)||\sin(P^i+q^i)|}{W(P+p_n+q)W(P+q)}, \quad \frac{K_Z(p_n,q)}{\theta(P+p_n)}\le \frac{2}{Z} \frac{|\sin(P^k+p_n^k)||\sin(P^i+q^i)| }{W(P+p_n+q)W(P+p_n)}.
\end{align*}
Therefore, applying \eqref{eq.weighted.fiber}, first with $(i,k,r)=(j_n,k,p_n)$ and then with $(i,k,r)=(k,j_n,q)$, we obtain
\begin{align}
\label{eq:schur.upper.row} \sup_{p_n\in\mathbb T_i^3} \int_{\mathbb T_k^3} \frac{K_Z(p_n,q)}{\theta(P+q)} \rmd q &\le \frac{C}{Z}, \quad \sup_{q\in\mathbb T_k^3} \int_{\mathbb T_i^3} \frac{K_Z(p_n,q)}{\theta(P+p_n)} \rmd p_n \le \frac{C}{Z}.
\end{align}
In the first quotient in \eqref{eq:schur.upper.row}, the value is set equal to zero when $\theta(P+q)=0$, and in the second quotient it is set equal to zero when $\theta(P+p_n)=0$. Indeed, if $\theta(P+p_n)=0$, then $\sin(P^k+p_n^k)=0$, and if $\theta(P+q)=0$, then $\sin(P^i+q^i)=0$.

Repeating the preceding Cauchy--Schwarz argument with
$W(P+p_n)$ and $W(P+q)$ replaced by
$\theta(P+p_n)$ and $\theta(P+q)$, respectively, we obtain
\begin{align*}
&\sum_{j\in[3]^n}\sum_{k=1}^3
\int_{\mathbb T_j^3}
\int_{\mathbb T_k^3}
K_Z(p_n,q)
|\widehat f_j(p)|
|\widehat f_{j_{1:n-1},k}(p_{1:n-1},q)|
\rmd q\rmd p\le
\frac{C}{Z}
\sum_{j\in[3]^n}
\int_{\mathbb T_j^3}
\theta(p)|\widehat f_j(p)|^2
\rmd p.
\end{align*}
The fraction appearing in the integrand on the left-hand side of
\eqref{eq:two.schur.upper} is bounded above by $K_Z(p_n,q)$. Setting
$
Z:=\frac{K_1c_{2m}z}{2},
$
the last estimate proves \eqref{eq:two.schur.upper}.
\end{proof}

\subsection{Comparison resolvents}\label{sec:compr.res}
In this subsection, we first use the integral estimates from Section \ref{subsec:integral.estimates} to bound the diagonal parts and the off-diagonal parts of the quadratic forms corresponding to $A_-R_m A_+$ for $m\ge 1$. In Proposition \ref{prop:res.quad.form}, we combine these results with Proposition~\ref{prop:hc.form} to compare $A_-R_m A_+$ with the diagonal operators $T_m$ for $m\ge 1$. This will allow us to establish the recursive bounds for $\Lambda_m$ in Section \ref{subsec:recur.operator.bounds}.

The next lemma gives useful properties of the sequences $(z_{m}(n))_{m\ge 1}$, $(\eps_m)_{m\ge 1}$ and $(c_m)_{m\ge 1}$ defined by \eqref{eq:def.zmn}, \eqref{eq:def.eps.m}
and \eqref{eq:def.cm} respectively. 

\begin{lemma}
\label{lem:scale.identities.losses}
There exists $C_\eps<\infty$ such that, for every $m,n\ge1$,
\begin{align}
\label{eq:scale.losses}
&\frac{n}{\sqrt{z_{2m}(n)}} + \frac{n}{\sqrt{z_{2m+1}(n)}} +\frac{m(1+\log z_{2m+1}(n))}{z_{2m+1}(n)}+\frac{m(1+\log z_{2m}(n))}{z_{2m}(n)}
\le C_\eps K_2^{-1/2}\varepsilon_m.
\end{align}
Let $(c_m)_{m\ge1}$ be the sequence defined by
\eqref{eq:def.cm}. Then there
exist constants $0<c_-<c_+<\infty$ such that for each $m\ge 1$, we have
\begin{align}
\label{eq:coeff.uniform.bounds}
c_-\le c_m\le c_+,\\
\label{eq:odd.coeff.err.condition}
c_{2m-1}K_{\rm err}\le\frac12.
\end{align}
For sufficiently large $K_2$, we have
\begin{align}
\label{eq:scale.condition}
\frac{c_{2m-1}}{K_1\sqrt{z_{2m}(n)}}
\le\frac12, \quad
 c_{2m}K_1\sqrt{z_{2m+1}(n)}
\ge
2\left(
1+c_{2m}K_{\rm err}\sqrt{z_{2m+1}(n)}
\right).
\end{align}
\end{lemma}

\begin{proof}
We now prove \eqref{eq:scale.losses}. First,
$$
\frac{n}{\sqrt{z_{2m}(n)}}
=
K_2^{-1/2}
\frac{n}{(n+2m)^{2+\eps}}.
$$
If $n\le2m$, then
$$
\frac{n}{(n+2m)^{2+\eps}}
\le
\frac{2m}{(2m)^{2+\eps}}
\le
C_\eps(m+1)^{-1-\eps}.
$$
If $n>2m$, then
$$
\frac{n}{(n+2m)^{2+\eps}}
\le
n^{-1-\eps}
\le
C_\eps(m+1)^{-1-\eps}.
$$
Thus
$$
\frac{n}{\sqrt{z_{2m}(n)}}
\le
C_\eps K_2^{-1/2}\varepsilon_m.
$$
The same argument gives
$$
\frac{n}{\sqrt{z_{2m+1}(n)}}
\le
C_\eps K_2^{-1/2}\varepsilon_m.
$$
It remains to bound the logarithmic terms. Since $K_2>1$ and
$\log x\le C_\eps x^{\frac{1+\eps}{4+2\eps}}$ for $x\ge1$, we have
$$
1+\log z_{2m}(n)\le
C_\eps K_2^{1/2}(n+2m)^{1+\eps}.
$$
Therefore
$$
\frac{m(1+\log z_{2m}(n))}{z_{2m}(n)}
\le
C_\eps K_2^{-1/2}
\frac{m}{(n+2m)^{3+\eps}}
\le
C_\eps K_2^{-1/2}(m+1)^{-1-\eps}.
$$
Similarly,
$$
\frac{m(1+\log z_{2m+1}(n))}{z_{2m+1}(n)}
\le
C_\eps K_2^{-1/2}
\frac{m}{(n+2m+1)^{3+\eps}}
\le
C_\eps K_2^{-1/2}\varepsilon_m.
$$
Combining the preceding estimates proves \eqref{eq:scale.losses}.

We next prove \eqref{eq:coeff.uniform.bounds}-\eqref{eq:odd.coeff.err.condition}. Combining the two identities in \eqref{eq:def.cm}, for $m\ge1$,
we obtain
\begin{align}
\label{eq:odd.coeff.recursion}
c_{2m+1}
&=
\frac{2 }{\pi c_{2m}}(1-\varepsilon_m)
=
\frac{2}
{\pi
\frac{2}{\pi c_{2m-1}}(1+\varepsilon_m)
}
(1-\varepsilon_m)
=
c_{2m-1}
\frac{1-\varepsilon_m}{1+\varepsilon_m}.
\end{align}
Thus
\begin{equation}
\label{eq:odd.coeff.product}
c_{2m+1}
=
c_1
\prod_{r=1}^{m}
\frac{1-\varepsilon_r}{1+\varepsilon_r}\quad\text{for all } m\ge1.
\end{equation}
Since $0<\varepsilon_r<1$, each factor in
\eqref{eq:odd.coeff.product} belongs to $(0,1)$. Hence for each $m\ge 0$, we have
$
0<c_{2m+1}\le c_1=\frac{1}{2K_{\rm err}},
$
which proves \eqref{eq:odd.coeff.err.condition}.

Since $\sum_{r=1}^{\infty}\varepsilon_r<\infty$, we have
$\Pi_\varepsilon
:=
\prod_{r=1}^{\infty}
\frac{1-\varepsilon_r}{1+\varepsilon_r}>0$. Therefore
\begin{equation}
\label{eq:odd.coeff.lower}
c_1\Pi_\varepsilon
\le
c_{2m+1}
\le
c_1\quad\text{for all } m\ge0.
\end{equation}
For the even subsequence, using \eqref{eq:odd.coeff.lower} and
$0<\varepsilon_m<1$, we get
\begin{equation}
\label{eq:even.coeff.bounds}
\frac{2 }{\pi c_1}
\le
c_{2m}
=
\frac{2 }{\pi c_{2m-1}}(1+\varepsilon_m)
\le
\frac{4}{\pi c_1\Pi_\varepsilon}
\quad\text{for each }  m\ge1.
\end{equation}
Combining \eqref{eq:odd.coeff.lower} and
\eqref{eq:even.coeff.bounds}, we obtain \eqref{eq:coeff.uniform.bounds}. For
example, one may take
$$
c_-:=
\min\left\{
c_1\Pi_\varepsilon,
\frac{2}{\pi c_1}
\right\},
\quad
c_+:=
\max\left\{
c_1,
\frac{4}{\pi c_1\Pi_\varepsilon}
\right\}.
$$

It remains to verify that $K_2$ can be chosen so that \eqref{eq:scale.condition} holds. For the first condition in \eqref{eq:scale.condition}, using that $z_{2m}(n)\ge K_2 3^{2(2+\eps)}$ for $m,n\ge1$ and the fact from \eqref{eq:odd.coeff.err.condition} that  $c_{2m-1}\le c_1$ for each $m\ge1$, it is enough to choose  $K_2$ sufficiently large such that $\frac{c_1}{K_1\sqrt{K_2}\,3^{2+\eps}} \le \frac12$.  For the second condition in \eqref{eq:scale.condition}, we need $c_{2m}(K_1-2K_{\rm err})\sqrt{z_{2m+1}(n)}\ge2$. Using that $K_1\ge4K_{\rm err}$, $c_{2m}\ge c_-$ and that $z_{2m+1}(n)\ge K_2 4^{2(2+\eps)}$ for $m,n\ge1$, it is enough to choose $K_2$ sufficiently large such that $\frac12 c_-K_1\sqrt{K_2}\,4^{2+\eps} \ge 2$. 
\end{proof}

\begin{proposition}
\label{prop:diag.estimates}
 There exists a constant $C<\infty$ such that, for every $m,n\ge1$ and every $f\in\ell^2_0(U,3,n)$,
\begin{align}
\label{eq:Vdiag.even}
V_{\rm Diag}(R_{2m-1},f) &\le \frac{2}{\pi c_{2m-1}} \left[ 1+ C\left( z_{2m}(n)^{-1/2} + \frac{m(1+\log z_{2m}(n))} {z_{2m}(n)} \right) \right] \nonumber\\
&\quad\times \left\langle f, K_1\sqrt{z_{2m}(n)} \mathfrak u_{m-1}(\lambda,\cdot,z_{2m}(n))f \right\rangle + \frac{C}{c_{2m-1}} \langle f,(-S)f\rangle,
\\
\label{eq:Vdiag.odd}
V_{\rm Diag}(R_{2m},f) &\ge \frac{2}{\pi c_{2m}} \left[ 1- C\left( z_{2m+1}(n)^{-1/2} + \frac{m(1+\log z_{2m+1}(n))} {z_{2m+1}(n)} \right) \right]
\nonumber\\
&\quad\times \Big\langle f, \frac{1}{K_1\sqrt{z_{2m+1}(n)}} \mathfrak l_m(\lambda,\cdot,z_{2m+1}(n))f \Big\rangle - \frac{C}{c_{2m}} \langle f,(-S)f\rangle.
\end{align}
\end{proposition}

\begin{proof}
For $k\in[3]$, $j\in[3]^n$, and
$p=(p_1,\cdots,p_n)\in\mathbb T_j^3$, let
$\Sigma_k(p):=\sin^2(\sum_{\ell=1}^n p_\ell^k)$. We first prove \eqref{eq:Vdiag.even}. Note that $z_{2m-1}(n+1)=z_{2m}(n)$.
Set
$
z:=z_{2m}(n).
$
We claim that, for every $m,n\ge1$,
\begin{equation}
\label{eq:diag.form.odd.res.upper}
r_{2m-1}^{(n+1)}(\lambda,(p,q))
\le
\frac{K_1\sqrt z}{c_{2m-1}}
\frac{1}
{
\lambda+\theta(p,q)
+
\mathfrak l_{m-1}(\lambda,(p,q),z)
}.
\end{equation}
For $m=1$, we have
$T_1=0$ and thus
$$
r_1^{(n+1)}(\lambda,(p,q))
=
\frac{1}{\lambda+\theta(p,q)}.
$$
By \eqref{eq:theta.sum.Theta}, we note that
$\mathfrak l_0(\lambda,(p,q),z)
=
\frac12\sum_{h=1}^3\sin^2(\sum_{\ell=1}^n p_{\ell}^h+q^h)
\le
\theta(p,q)$. By \eqref{eq:scale.condition}, we have
$c_1/(K_1\sqrt z)\le1/2$. Hence
$$
\frac{c_1}{K_1\sqrt z}
\left[
\lambda+\theta(p,q)+\mathfrak l_0(\lambda,(p,q),z)
\right]
\le
\lambda+\theta(p,q).
$$
This proves \eqref{eq:diag.form.odd.res.upper} for $m=1$. Assume now that $m\ge2$. By the definition of $R_{2m-1}$ and
\eqref{eq:def.s.odd}, its Fourier multiplier is
$$
r_{2m-1}^{(n+1)}(\lambda,(p,q))
=
\frac{1}
{
\lambda+\theta(p,q)
+
c_{2m-1}
\left[
\frac{1}{K_1\sqrt z}
\mathfrak l_{m-1}(\lambda,(p,q),z)
-
K_{\rm err}\theta(p,q)
\right]
}.
$$
Since $c_{2m-1}K_{\rm err}\le\frac12$, the denominator is bounded
below by
$$
\lambda+\frac12\theta(p,q)
+
\frac{c_{2m-1}}{K_1\sqrt z}
\mathfrak l_{m-1}(\lambda,(p,q),z).
$$
By \eqref{eq:scale.condition},
$c_{2m-1}/(K_1\sqrt z)\le1/2$. Hence
$$
\lambda+\frac12\theta(p,q)
+
\frac{c_{2m-1}}{K_1\sqrt z}
\mathfrak l_{m-1}(\lambda,(p,q),z)
\ge
\frac{c_{2m-1}}{K_1\sqrt z}
\left[
\lambda+\theta(p,q)
+
\mathfrak l_{m-1}(\lambda,(p,q),z)
\right].
$$
This proves \eqref{eq:diag.form.odd.res.upper} for all
$m\ge1$.

By Lemma~\ref{lem:V.diag},
\eqref{eq:diag.form.odd.res.upper}, and
the upper bound in Lemma~\ref{lem:diagonal.integrals}, applied with $m$ replaced by
$m-1$, we obtain
\begin{align*}
V_{\rm Diag}(R_{2m-1},f)
&\le
\frac{2}{\pi c_{2m-1}}
\left[
1+
C\left(
z^{-1/2}
+
\frac{m(1+\log z)}{z}
\right)
\right]
\left\langle
f,
K_1\sqrt z
\mathfrak u_{m-1}(\lambda,\cdot,z)f
\right\rangle
\nonumber\\
&\quad+
\frac{CK_1\sqrt z}{c_{2m-1}}
\sum_{k=1}^3
\langle f,\Sigma_k f\rangle.
\end{align*}
By Lemma~\ref{lem:ULB}.i, since $L(x,z)\ge z$, we have
${\rm UB}_{m-1}(x,z)\ge\sqrt z$, and hence
$
\mathfrak u_{m-1}(\lambda,p,z)
\ge
\frac{\sqrt z}{2}\sum_{k=1}^3\Sigma_k(p).
$
Consequently,
$$
\frac{CK_1\sqrt z}{c_{2m-1}}
\sum_{k=1}^3
\langle f,\Sigma_k f\rangle
\le
\frac{C}{c_{2m-1}\sqrt z}
\left\langle
f,
K_1\sqrt z
\mathfrak u_{m-1}(\lambda,\cdot,z)f
\right\rangle.
$$
After increasing $C$, this term is absorbed into the $z^{-1/2}$ term.
Since $z=z_{2m}(n)$, adding the non-negative term
$
C c_{2m-1}^{-1}\langle f,(-S)f\rangle
$
proves \eqref{eq:Vdiag.even}.

We now prove \eqref{eq:Vdiag.odd}. Note that $z_{2m}(n+1)=z_{2m+1}(n)$. Set
$
z:=z_{2m+1}(n).
$
By \eqref{eq:def.s.even}, the multiplier of $R_{2m}$ is
$$
r_{2m}^{(n+1)}(\lambda,(p,q))
=
\frac{1}
{
\lambda+\theta(p,q)
+
c_{2m}K_1\sqrt z\,
\mathfrak u_{m-1}(\lambda,(p,q),z)
+
c_{2m}K_{\rm err}\sqrt z\,\theta(p,q)
}.
$$
By \eqref{eq:scale.condition},
$
c_{2m}K_1\sqrt z
\ge
2(1+c_{2m}K_{\rm err}\sqrt z),
$
and hence
$
c_{2m}K_1\sqrt z\ge1+c_{2m}K_{\rm err}\sqrt z
$
and
$
c_{2m}K_1\sqrt z\ge1.
$
Thus
$$
\lambda+\theta(p,q)
+
c_{2m}K_1\sqrt z\,
\mathfrak u_{m-1}(\lambda,(p,q),z)
+
c_{2m}K_{\rm err}\sqrt z\,\theta(p,q)
\le
c_{2m}K_1\sqrt z
\left[
\lambda+\theta(p,q)
+
\mathfrak u_{m-1}(\lambda,(p,q),z)
\right].
$$
Consequently,
\begin{equation}
\label{eq:diag.form.even.res.lower}
r_{2m}^{(n+1)}(\lambda,(p,q))
\ge
\frac{1}{c_{2m}K_1\sqrt z}
\frac{1}
{
\lambda+\theta(p,q)
+
\mathfrak u_{m-1}(\lambda,(p,q),z)
}.
\end{equation}
By the lower-bound part of Lemma~\ref{lem:V.diag}, \eqref{eq:diag.form.even.res.lower}, and
the lower bound in Lemma~\ref{lem:diagonal.integrals}, applied with $m$ replaced by
$m-1$, we obtain
\begin{align*}
V_{\rm Diag}(R_{2m},f)
&\ge
\frac{2}{\pi c_{2m}}
\left[
1-
C\left(
z^{-1/2}
+
\frac{m(1+\log z)}{z}
\right)
\right]
\left\langle
f,
\frac{1}{K_1\sqrt z}
\mathfrak l_m(\lambda,\cdot,z)f
\right\rangle
\nonumber\\
&\quad-
\frac{C[{\rm LB}_m(1,z)+1]}
{c_{2m}K_1\sqrt z}
\sum_{k=1}^3
\langle f,\Sigma_k f\rangle.
\end{align*}
By Lemma~\ref{lem:ULB}.i,
$
{\rm LB}_m(1,z)
\le
\sqrt{L(1,z)}
\le
C\sqrt z.
$
Since $z>1$ and
$
\frac12\sum_{k=1}^3\Sigma_k\le\theta,
$
we have
$$
\frac{C[{\rm LB}_m(1,z)+1]}
{c_{2m}K_1\sqrt z}
\sum_{k=1}^3
\langle f,\Sigma_k f\rangle
\le
\frac{C}{c_{2m}}
\langle f,(-S)f\rangle.
$$
Since $z=z_{2m+1}(n)$, this proves
\eqref{eq:Vdiag.odd}. The proof is complete.
\end{proof}

\begin{proposition}
\label{lem:offdiag.estimates}
There exists a constant $C<\infty$, such that, for every $m,n\ge1$ and every
$f\in\ell^2_0(U,3,n)$,
\begin{align}
\label{eq:offdiag.form.odd.res}
\left|V_{\rm Off}(R_{2m-1},f)\right|
&\le
\frac{C n}{c_{2m-1}}
\left\langle
f,\mathfrak u_{m-1}(\lambda,\cdot,z_{2m}(n))f
\right\rangle,
\\
\label{eq:offdiag.form.even.res}
\left|V_{\rm Off}(R_{2m},f)\right|
&\le
\frac{C n}
{K_1 c_{2m}z_{2m+1}(n)}
\left[
\left\langle
f,\mathfrak l_m(\lambda,\cdot,z_{2m+1}(n))f
\right\rangle
+
\left\langle f,(-S)f\right\rangle
\right].
\end{align}
\end{proposition}

\begin{proof}
All constants below are independent of $m,n,\lambda$, $f$. We first prove \eqref{eq:offdiag.form.odd.res}. By the symmetric
representation of the off-diagonal form,
\begin{align}
\label{eq:VOff.odd.start}
V_{\rm Off}(R_{2m-1},f)
&=
\frac{n}{(2\pi)^{2(n+1)}}
\operatorname{Re}
\sum_{j\in[3]^{n+1}}
\int_{\mathbb T_j^3}
r_{2m-1}^{(n+1)}(\lambda,p_{1:n+1})
\overline{\widehat f_{j_{1:n}}(p_{1:n})}
\widehat f_{j_{1:n-1},j_{n+1}}
(p_{1:n-1},p_{n+1})
\nonumber\\
&\quad\times
\sin\left(
\sum_{\ell=1}^{n-1}p_\ell^{j_n}+p_{n+1}^{j_n}
\right)
\sin\left(
\sum_{\ell=1}^{n}p_\ell^{j_{n+1}}
\right)
\rmd p_{1:n+1}.
\end{align}
For fixed $(j_1,\ldots,j_n)\in[3]^n$, put
$k:=j_{n+1}$, $q:=p_{n+1}$, and 
$p:=(p_1,\ldots,p_n)$.
Set $z:=z_{2m}(n)$. 

We claim that
\begin{equation}
\label{eq:R.odd.multiplier.bound.offdiag}
r_{2m-1}^{(n+1)}(\lambda,(p,q))
\le
\frac{C}{c_{2m-1}}
\frac{1}
{
\lambda+\theta(p,q)
+
K_1^{-1}z^{-1/2}
\mathfrak l_{m-1}(\lambda,(p,q),z)
}.
\end{equation}
If $m=1$, then $R_{2m-1}=R_1^{(n+1)}$ and
$T_1^{(n+1)}=0$, so
$$
r_1^{(n+1)}(\lambda,(p,q))
=
\frac{1}{\lambda+\theta(p,q)}.
$$
By \eqref{eq:theta.sum.Theta},
$\mathfrak l_0(\lambda,(p,q),z)
=
\frac12\sum_{h=1}^3\sin^2(\sum_{\ell=1}^n p_{\ell}^h+q^h)
\le
\theta(p,q)$. Since $K_1^{-1}z^{-1/2}\le1$, we have
$$
\lambda+\theta(p,q)
+
K_1^{-1}z^{-1/2}\mathfrak l_0(\lambda,(p,q),z)
\le
2(\lambda+\theta(p,q)).
$$
Also $c_1K_{\rm err}=1/2$ and $K_{\rm err}>1$ imply $c_1\le1/2$.
Hence \eqref{eq:R.odd.multiplier.bound.offdiag} follows for
$m=1$.

Assume now that $m\ge2$. By the definition of $R_{2m-1}$,
\eqref{eq:def.s.odd} and $z_{2m-1}(n+1)=z_{2m}(n)$,
$$
r_{2m-1}^{(n+1)}(\lambda,(p,q))
=
\frac{1}
{
\lambda+\theta(p,q)
+
c_{2m-1}
\left[
\frac{1}{K_1\sqrt z}
\mathfrak l_{m-1}(\lambda,(p,q),z)
-
K_{\rm err}\theta(p,q)
\right]
}.
$$
Using the fact that $c_{2m-1}K_{\rm err}\le1/2$, the denominator is bounded below by
$$
\lambda+\frac12\theta(p,q)
+
\frac{c_{2m-1}}{K_1\sqrt z}
\mathfrak l_{m-1}(\lambda,(p,q),z).
$$
Since $K_{\rm err}>1$ and $c_{2m-1}K_{\rm err}\le1/2$, we have
$c_{2m-1}\le1/2$. Therefore
$$
\lambda+\frac12\theta(p,q)
+
\frac{c_{2m-1}}{K_1\sqrt z}
\mathfrak l_{m-1}(\lambda,(p,q),z)
\ge
c_{2m-1}
\left[
\lambda+\theta(p,q)
+
K_1^{-1}z^{-1/2}
\mathfrak l_{m-1}(\lambda,(p,q),z)
\right].
$$
This proves \eqref{eq:R.odd.multiplier.bound.offdiag} for $m\ge2$, and hence for every $m\ge1$.

Substituting \eqref{eq:R.odd.multiplier.bound.offdiag} into \eqref{eq:VOff.odd.start}, taking absolute values, and applying \eqref{eq:two.schur.lower} with $z=z_{2m}(n)$, we obtain
$$
\left|V_{\rm Off}(R_{2m-1},f)\right| \le \frac{Cn}{c_{2m-1}} \frac{1}{(2\pi)^{2n}} \sum_{j\in[3]^n} \int_{\mathbb T_j^3} \mathfrak u_{m-1}(\lambda,p,z_{2m}(n)) |\widehat f_j(p)|^2\rmd p.
$$
By the Parseval–Plancherel formula \eqref{eq:parseval.parameterized}, this implies \eqref{eq:offdiag.form.odd.res}.

We now prove \eqref{eq:offdiag.form.even.res}. By the definition of $V_{\rm Off}$, we have
\begin{align}
\label{eq:VOff.even.start}
V_{\rm Off}(R_{2m},f)
&=
\frac{n}{(2\pi)^{2(n+1)}}
\operatorname{Re}
\sum_{j\in[3]^{n+1}}
\int_{\mathbb T_j^3}
r_{2m}^{(n+1)}(\lambda,p_{1:n+1})
\overline{\widehat f_{j_{1:n}}(p_{1:n})}
\widehat f_{j_{1:n-1},j_{n+1}}
(p_{1:n-1},p_{n+1})
\nonumber\\
&\quad\times
\sin\left(
\sum_{\ell=1}^{n-1}p_\ell^{j_n}+p_{n+1}^{j_n}
\right)
\sin\left(
\sum_{\ell=1}^{n}p_\ell^{j_{n+1}}
\right)
\rmd p_{1:n+1}.
\end{align}
Again put $k:=j_{n+1}$, $q:=p_{n+1}$. Set $z:=z_{2m+1}(n)$. 

By the definition of $R_{2m}$, \eqref{eq:def.s.even} and $z_{2m}(n+1)=z_{2m+1}(n)$, we have
\begin{align*}
r_{2m}^{(n+1)}(\lambda,(p,q)) &= \frac{1} { \lambda+\theta(p,q) + c_{2m}K_1\sqrt z\, \mathfrak u_{m-1}(\lambda,(p,q),z) + c_{2m}K_{\rm err}\sqrt z\,\theta(p,q)}\\
 & \le \frac{1}{\lambda+\theta(p,q)+c_{2m}K_1\sqrt z\,\mathfrak u_{m-1}(\lambda,(p,q),z)}.
\end{align*}
Applying \eqref{eq:two.schur.upper} to \eqref{eq:VOff.even.start}, we obtain
$$
\left|V_{\rm Off}(R_{2m},f)\right|
\le
\frac{Cn}{K_1c_{2m}z}
\langle f,(-S)f\rangle.
$$
Using the fact that $\mathfrak l_m(\lambda,\cdot,z)\ge0$ and substituting $z=z_{2m+1}(n)$, this implies
$$
\left|V_{\rm Off}(R_{2m},f)\right|
\le
\frac{Cn}{K_1 c_{2m}z_{2m+1}(n)}
\left[
\left\langle
f,\mathfrak l_m(\lambda,\cdot,z_{2m+1}(n))f
\right\rangle
+
\left\langle f,(-S)f\right\rangle
\right].
$$
Hence 
\eqref{eq:offdiag.form.even.res} is verified.
\end{proof}

\begin{proposition}
\label{prop:res.quad.form}
There exists a sufficiently large constant $K_{\rm err}^0<\infty$ such that the following holds. For every
$K_{\rm err}\ge K_{\rm err}^0$ and every $K_1\ge4K_{\rm err}$, there
exists a sufficiently large constant $K_2^0<\infty$, depending on $K_1$ and $K_{\rm err}$ such that, for every $m,n\ge1$ and
every $f\in\ell^2_0(U,3,n)$,
\begin{align}
\label{eq:one.step.even}
\left\langle
f,
A_-R_{2m-1}A_+f
\right\rangle
&\le
\frac{2}{\pi c_{2m-1}}
(1+\varepsilon_m)
\langle f,T_{2m}f\rangle,
\\
\label{eq:one.step.odd}
\left\langle
f,
A_-R_{2m}A_+f
\right\rangle
&\ge
\frac{2}{\pi c_{2m}}
(1-\varepsilon_m)
\langle f,T_{2m+1}f\rangle.
\end{align}
\end{proposition}

\begin{proof}
All constants $C$ below are independent of $m$, $n$, $\lambda$, $K_{\rm err}$, $K_1$ and $K_2$. We first prove \eqref{eq:one.step.even}.
By \eqref{eq:def.s.odd} and \eqref{eq:Rm.multiplier}, we have
$$r_{2m+1}^{(n+1)}(\lambda, p)\le \frac{1}{\lambda+(1-c_{2m+1}K_{\rm err})\theta(p)}.$$
Since $c_{2m-1}K_{\rm err}\le\frac12$, the resolvent
$R_{2m-1}$ satisfies
\eqref{eq:R.res} with $\kappa=1$ and $C_R=2$. Therefore, applying Proposition~\ref{prop:hc.form}, we get
\begin{align}
\label{eq:one.step.even.after}
\left\langle f,A_-R_{2m-1}A_+f \right\rangle &\le V_{\rm Diag}(R_{2m-1},f) + |V_{\rm Off}(R_{2m-1},f)|+Cn\langle f,(-S)f\rangle.
\end{align}
By Proposition~\ref{prop:diag.estimates} and
\eqref{eq:scale.losses}, for $K_2\ge K_2^0$ with sufficiently large $K_2^0$, we have 
\begin{align}
\label{eq:one.step.even.diag}
V_{\rm Diag}(R_{2m-1},f)
&\le
\frac{2\left(1+\frac{\varepsilon_m}{8}\right)}{\pi c_{2m-1}}
\left\langle
f,
K_1\sqrt{z_{2m}(n)}
\mathfrak u_{m-1}(\lambda,\cdot,z_{2m}(n))f
\right\rangle
+
\frac{C}{c_{2m-1}}
\langle f,(-S)f\rangle.
\end{align}
By Proposition~\ref{lem:offdiag.estimates},
\begin{align*}
|V_{\rm Off}(R_{2m-1},f)|
&\le
\frac{C n}{c_{2m-1}}
\left\langle
f,
\mathfrak u_{m-1}(\lambda,\cdot,z_{2m}(n))f
\right\rangle.
\end{align*}
By \eqref{eq:scale.losses}, $\frac{n}{\sqrt{z_{2m}(n)}}
\le C_\eps K_2^{-1/2}\varepsilon_m$. For $K_2\ge K_2^0$ with sufficiently large $K_2^0$, we thus have
\begin{align}
\label{eq:one.step.even.offdiag}
|V_{\rm Off}(R_{2m-1},f)|&\le
\frac{\varepsilon_m}{4\pi c_{2m-1}}
\left\langle
f,
K_1\sqrt{z_{2m}(n)}
\mathfrak u_{m-1}(\lambda,\cdot,z_{2m}(n))f
\right\rangle.
\end{align}
Combining \eqref{eq:one.step.even.after} with \eqref{eq:one.step.even.diag} and \eqref{eq:one.step.even.offdiag}, we obtain
\begin{align*}
\left\langle
f,
A_-R_{2m-1}A_+f
\right\rangle
&\le
\frac{2\left(1+\frac{\varepsilon_m}{4}\right)}{\pi c_{2m-1}}
\left\langle
f,
K_1\sqrt{z_{2m}(n)}
\mathfrak u_{m-1}(\lambda,\cdot,z_{2m}(n))f
\right\rangle
+
\Big(Cn+\frac{C}{c_{2m-1}}\Big) \langle f,(-S)f\rangle.
\end{align*}
Since $c_{2m-1}K_{\rm err}\le1/2$, we have $c_{2m-1}\le1/2$ and thus
$Cn+\frac{C}{c_{2m-1}}
\le
\frac{C}{c_{2m-1}}(1+n)$. By \eqref{eq:scale.losses}, for $K_2\ge K_2^0$ with sufficiently large $K_2^0$, we have
$\frac{n+1}{\sqrt{z_{2m}(n)}}
\le
C_\eps K_2^{-1/2}\varepsilon_m
\le
\frac{K_{\rm err}}{4\pi C}$.
Therefore
$$
C\Big(
n+\frac{1}{c_{2m-1}}
\Big)
\langle f,(-S)f\rangle
\le
\frac{K_{\rm err}}{4\pi c_{2m-1}}
\sqrt{z_{2m}(n)}
\langle f,(-S)f\rangle.
$$
Thus
\begin{align*}
\left\langle
f,
A_-R_{2m-1}A_+f
\right\rangle
&\le
\frac{2 }{\pi c_{2m-1}}
\left[
\left(1+\frac{\varepsilon_m}{4}\right)
\left\langle
f,
K_1\sqrt{z_{2m}(n)}
\mathfrak u_{m-1}(\lambda,\cdot,z_{2m}(n))f
\right\rangle
\right.
\nonumber\\
&\quad\left.
+
\frac18K_{\rm err}\sqrt{z_{2m}(n)}
\langle f,(-S)f\rangle
\right].
\end{align*}
Since
$T_{2m}
=
K_1\sqrt{z_{2m}(n)}
\mathfrak u_{m-1}(\lambda,\cdot,z_{2m}(n))
+
K_{\rm err}\sqrt{z_{2m}(n)}(-S)$,
and since $\frac{1}{8}\le 1+\frac{\varepsilon_m}{4}\le1+\varepsilon_m$, this implies \eqref{eq:one.step.even}.

We now prove \eqref{eq:one.step.odd}. Using \eqref{eq:def.s.even}, \eqref{eq:Rm.multiplier} and the fact that $z_{2m}(n+1)=z_{2m+1}(n)$, we have
$$
r_{2m}^{(n+1)}(\lambda,p)
\le
\frac{1}
{\lambda+
\left(
1+c_{2m}K_{\rm err}\sqrt{z_{2m+1}(n)}
\right)
\theta(p)}.
$$
Applying Proposition~\ref{prop:hc.form} with
$\kappa=
1+c_{2m}K_{\rm err}\sqrt{z_{2m+1}(n)}$, we have
\begin{align}
\label{eq:one.step.odd.after}
\left\langle
f,
A_-R_{2m}A_+f
\right\rangle
&\ge
V_{\rm Diag}(R_{2m},f)
-
|V_{\rm Off}(R_{2m},f)|
-
\frac{C n \langle f,(-S)f\rangle}
{1+c_{2m}K_{\rm err}\sqrt{z_{2m+1}(n)}}.
\end{align}
By Proposition~\ref{prop:diag.estimates} and
\eqref{eq:scale.losses}, for $K_2\ge K_2^0$ with sufficiently large $K_2^0$, we have
\begin{align}
\label{eq:one.step.odd.diag}
V_{\rm Diag}(R_{2m},f)
&\ge
\frac{2\left(1-\frac{\varepsilon_m}{8}\right)}{\pi c_{2m}} \Big\langle f, \frac{1}{K_1\sqrt{z_{2m+1}(n)}} \mathfrak l_m(\lambda,\cdot,z_{2m+1}(n))f \Big\rangle - \frac{C}{c_{2m}} \langle f,(-S)f\rangle.
\end{align}
By Proposition~\ref{lem:offdiag.estimates}, with
$z=z_{2m+1}(n)$,
\begin{align*}
|V_{\rm Off}(R_{2m},f)|
\le
\frac{Cn}{K_1c_{2m}z_{2m+1}(n)}
\left[
\left\langle
f,\mathfrak l_m(\lambda,\cdot,z_{2m+1}(n))f
\right\rangle
+
\left\langle f,(-S)f\right\rangle
\right].
\end{align*}
By \eqref{eq:scale.losses}, we have
$\frac{n}{z_{2m+1}(n)}\le \frac{n}{\sqrt {z_{2m+1}(n)}}
\le
C_\eps K_2^{-1/2}\varepsilon_m$.
For $K_2\ge K_2^0$ with sufficiently large $K_2^0$, we thus obtain
\begin{align}
\label{eq:one.step.odd.offdiag}
|V_{\rm Off}(R_{2m},f)|
&\le \frac{\varepsilon_m}{4\pi c_{2m}} \Big\langle f,\frac{1}{K_1\sqrt{z_{2m+1}(n)}} \mathfrak l_m(\lambda,\cdot,z_{2m+1}(n))f \Big\rangle + \frac{C\varepsilon_m}{c_{2m}} \langle f,(-S)f\rangle.
\end{align}
The last term is bounded by
$C c_{2m}^{-1}\langle f,(-S)f\rangle$. Combining \eqref{eq:one.step.odd.after},
\eqref{eq:one.step.odd.diag}, and
\eqref{eq:one.step.odd.offdiag}, we obtain
\begin{align*}
\left\langle
f,
A_-R_{2m}A_+f
\right\rangle
&\ge
\frac{2 \left(1-\frac{\varepsilon_m}{4}\right)}{\pi c_{2m}}
\Big\langle f,\frac{1}{K_1\sqrt{z_{2m+1}(n)}}
\mathfrak l_m(\lambda,\cdot,z_{2m+1}(n))f \Big\rangle \\
&\quad
-
\left[
\frac{C}{c_{2m}}
+
\frac{C n}
{1+c_{2m}K_{\rm err}\sqrt{z_{2m+1}(n)}}
\right]
\langle f,(-S)f\rangle.
\end{align*}
By \eqref{eq:scale.losses},
$\frac{n}{\sqrt{z_{2m+1}(n)}}
\le
C_\eps K_2^{-1/2}\varepsilon_m$. For $K_2\ge K_2^0$ with sufficiently large $K_2^0$, we thus have
$$
\frac{n}
{1+c_{2m}K_{\rm err}\sqrt{z_{2m+1}(n)}}
\le
\frac{1}{c_{2m}K_{\rm err}}
\frac{n}{\sqrt{z_{2m+1}(n)}}\le \frac{C_\eps K_2^{-1/2}\varepsilon_m}{c_{2m}K_{\rm err}}
\le
\frac{C}{c_{2m}}.
$$
Therefore, for $K_{\rm err}\ge K_{\rm err}^0$ with sufficiently large $K_{\rm err}^0$, we have
$$
\left[
\frac{C}{c_{2m}}
+
\frac{C n}
{1+c_{2m}K_{\rm err}\sqrt{z_{2m+1}(n)}}
\right]
\le
\frac{K_{\rm err}}{2\pi c_{2m}}.
$$
Hence
\begin{align*}
\left\langle
f,
A_-R_{2m}A_+f
\right\rangle
&\ge
\frac{2\left(1-\frac{\varepsilon_m}{4}\right)}{\pi c_{2m}}
\Big\langle
f,
\frac{1}{K_1\sqrt{z_{2m+1}(n)}}
\mathfrak l_m(\lambda,\cdot,z_{2m+1}(n))f
\Big\rangle -
\frac{K_{\rm err}}{2\pi c_{2m}}
\langle f,(-S)f\rangle.
\end{align*}
Since $1-\frac{\varepsilon_m}{4}\ge1-\varepsilon_m$ and
$$
T_{2m+1}
=
\frac{1}{K_1\sqrt{z_{2m+1}(n)}}
\mathfrak l_m(\lambda,\cdot,z_{2m+1}(n))
-
K_{\rm err}(-S),
$$
this implies
\eqref{eq:one.step.odd}. The proof is complete.
\end{proof}

\subsection{Recursive operator bounds}
\label{subsec:recur.operator.bounds}
Throughout this subsection we assume $d=3$ and fix $\lambda\in (0,1)$. We prove the recursive comparison between the operator $\Lambda_m$ defined in \eqref{eq:def.Lambda} and the diagonal operators $T_m$ defined by the Fourier multipliers
\eqref{eq:T.def}.

\begin{proposition}
\label{prop:recur.bounds}
We have $\Lambda_1=c_1T_1=0$. For every $m\ge1$,
\begin{align}
\label{eq:recur.even.bound}
\Lambda_{2m} &\le c_{2m}T_{2m} \quad\text{and}\quad \Lambda_{2m+1}\ge c_{2m+1}T_{2m+1}.
\end{align}
\end{proposition}

\begin{proof}
By \eqref{eq:def.Lambda}, we have
$\Lambda_1=0$. Since $\mathfrak s_n^{(1)}(\lambda,p)=0$, we also
have $T_1=0$. Therefore
$\Lambda_1=c_1T_1=0$ on $\ell_0^2(U,3,n)$ for all $n\ge1$.

We prove \eqref{eq:recur.even.bound} by induction. First, we prove
that, for each $m\ge1$,
\begin{equation}
\label{eq:recur.odd.to.even}
\text{if}\quad
\Lambda_{2m-1}
\ge
c_{2m-1}T_{2m-1}
\quad\text{then} \quad
\Lambda_{2m}
\le
c_{2m}T_{2m}.
\end{equation}
Assume that, for some fixed $m\ge1$ and every $n\ge1$,
\begin{equation}
\label{eq:recur.induction.odd.ass}
\Lambda_{2m-1}
\ge
c_{2m-1}T_{2m-1}\quad\text{on $\ell_0^2(U,3,n)$}.
\end{equation}
Let $n\ge1$ and let $f\in\ell^2_0(U,3,n)$. By the definition \eqref{eq:def.Lambda},
\begin{align}
\label{eq:Lambda.even.start}
\langle f,\Lambda_{2m}f\rangle
&=
\left\langle
f,A_-
\left(
\lambda-S+\Lambda_{2m-1}
\right)^{-1}
A_+f
\right\rangle.
\end{align}
Applying \eqref{eq:recur.induction.odd.ass}, we have
\begin{equation*}
\lambda-S+\Lambda_{2m-1}
\ge
\lambda-S+c_{2m-1}T_{2m-1}.
\end{equation*}
Both sides of
the above inequality are strictly positive. Hence 
\begin{align*}
\left(
\lambda-S+\Lambda_{2m-1}
\right)^{-1}
&\le
\left(
\lambda-S
+
c_{2m-1}T_{2m-1}
\right)^{-1}
=
R_{2m-1}.
\end{align*}
Combining this with \eqref{eq:Lambda.even.start}  and using \eqref{eq:one.step.even}, we obtain
\begin{align*}
\langle f,\Lambda_{2m}f\rangle
&\le
\left\langle
f,A_-
R_{2m-1}
A_+f
\right\rangle\le
\frac{2 }{\pi c_{2m-1}}
(1+\varepsilon_m)
\langle f,T_{2m}f\rangle.
\end{align*}
By the definition \eqref{eq:def.cm},
$c_{2m}
=
\frac{2 }{\pi c_{2m-1}}(1+\varepsilon_m)$.
Hence, we obtain
\begin{equation*}
\langle f,\Lambda_{2m}f\rangle
\le
c_{2m}\langle f,T_{2m}f\rangle.
\end{equation*}
This proves \eqref{eq:recur.odd.to.even}.

Second, we prove that, for each $m\ge1$,
\begin{equation}
\label{eq:recur.even.to.odd}
\text{if}\quad \Lambda_{2m}
\le
c_{2m}T_{2m}\quad\text{then}\quad
\Lambda_{2m+1}
\ge
c_{2m+1}T_{2m+1}.
\end{equation}
Assume that, for some fixed $m\ge1$ and every $n\ge1$,
\begin{equation}
\label{eq:recur.induction.even.ass}
\Lambda_{2m}
\le
c_{2m}T_{2m} \quad\text{on $\ell_0^2(U,3,n)$}.
\end{equation}
Let $n\ge1$ and let $f\in\ell^2_0(U,3,n)$. By
\eqref{eq:def.Lambda},
\begin{align}
\label{eq:Lambda.odd.start}
\langle f,\Lambda_{2m+1}f\rangle
&=
\left\langle
f,A_-
\left(
\lambda-S+\Lambda_{2m}
\right)^{-1}
A_+f
\right\rangle.
\end{align}
Applying \eqref{eq:recur.induction.even.ass}, we have
\begin{equation*}
\lambda-S+\Lambda_{2m}
\le
\lambda-S+c_{2m}T_{2m}.
\end{equation*}
Both sides of the above inequality are strictly positive. Hence
\begin{align*}
\left(
\lambda-S+\Lambda_{2m}
\right)^{-1}
&\ge
\left(
\lambda-S
+
c_{2m}T_{2m}
\right)^{-1}
=
R_{2m}.
\end{align*}
Combining this with
\eqref{eq:Lambda.odd.start} and using \eqref{eq:one.step.odd}, we obtain
\begin{align*}
\langle f,\Lambda_{2m+1}f\rangle &\ge \left\langle f, A_- R_{2m} A_+f \right\rangle \ge \frac{2 }{\pi c_{2m}} (1-\varepsilon_m) \langle f,T_{2m+1}f\rangle.
\end{align*}
Using the fact that $c_{2m+1} = \frac{2}{\pi c_{2m}}(1-\varepsilon_m)$, we thus get
\begin{equation*}
\langle f,\Lambda_{2m+1}f\rangle \ge c_{2m+1}\langle f,T_{2m+1}f\rangle.
\end{equation*}
This proves \eqref{eq:recur.even.to.odd}.

It remains to prove that the two implications imply the proposition for
all $m\ge1$. Since $\Lambda_1=c_1T_1$, we
have $\Lambda_1\ge c_1T_1$. Applying
\eqref{eq:recur.odd.to.even} with $m=1$, we have
$\Lambda_2\le c_2T_2$. Applying
\eqref{eq:recur.even.to.odd} with $m=1$, we have
$\Lambda_3\ge c_3T_3$. Iterating the two
implications, we obtain \eqref{eq:recur.even.bound} for every $m\ge1$. The proof is complete.
\end{proof}

\subsection{Proof of the main theorem}
\label{subsec:proof.main.theorem}

We now prove Theorem \ref{thm.main}.

\begin{proof}[Proof of Theorem \ref{thm.main}]
Fix $\eps>0$. Recall from Section 2 that
\begin{equation*}
D_G(\lambda)=2\lambda^{-2}\langle \phi,(\lambda-G)^{-1}\phi\rangle_{L^2(\mathcal{E},\pi)}= 2\lambda^{-2} \llangle u,(\lambda-G)^{-1}u\rrangle.
\end{equation*}
Moreover, $\E[|X_t|^2]=d\left(t+E_G(t)\right)$. Therefore $D(\lambda)=d\lambda^{-2}+dD_G(\lambda)$, and hence
\begin{equation}
\label{eq:lambda2D.res}
\lambda^2D(\lambda) = d+ 2d\llangle u,(\lambda-G)^{-1}u\rrangle.
\end{equation}
Thus it is enough to estimate $\llangle u,(\lambda-G)^{-1}u\rrangle$.

Let $\lambda_0\in(0,1)$ be sufficiently small such that $\left\lfloor\frac12\log\log(1+\lambda^{-1})\right\rfloor\ge1$ for every $0<\lambda<\lambda_0$. For such $\lambda$, set
\begin{equation}
\label{eq:def.Mlambda}
M_\lambda := \left\lfloor \frac12\log \log(1+\lambda^{-1}) \right\rfloor.
\end{equation}
In the rest of the proof we write $M$ for $M_\lambda$. By Lemma
\ref{lemma.monotonicity}, we have
\begin{equation*}
\llangle u,(\lambda-G_{2M})^{-1}u\rrangle \le \llangle u,(\lambda-G)^{-1}u\rrangle \le \llangle u,(\lambda-G_{2M+1})^{-1}u\rrangle.
\end{equation*}
Using \eqref{eq:schur.rep}, we obtain
\begin{align}
\label{eq:schur.trunc}
\left\langle u^{(1)}, (\lambda-S+\Lambda_{2M})^{-1} u^{(1)} \right\rangle &\le \llangle u,(\lambda-G)^{-1}u\rrangle \le \left\langle u^{(1)}, (\lambda-S+\Lambda_{2M+1})^{-1} u^{(1)} \right\rangle.
\end{align}

We first prove the upper bound. By Proposition \ref{prop:recur.bounds}, we have $\Lambda_{2M+1}\ge c_{2M+1}T_{2M+1}$. Since both $\lambda-S+\Lambda_{2M+1}$ and $\lambda-S+c_{2M+1}T_{2M+1}$ are strictly positive, we get
\begin{align}
\label{eq:upper.inv}
&\left\langle u^{(1)}, (\lambda-S+\Lambda_{2M+1})^{-1} u^{(1)} \right\rangle \le \left\langle u^{(1)}, (\lambda-S+c_{2M+1}T_{2M+1})^{-1} u^{(1)}
\right\rangle.
\end{align}
By \eqref{eq:def.s.odd},
\begin{align*}
&\lambda+\theta(p)+c_{2M+1}\mathfrak s_1^{(2M+1)}(\lambda,p) = \lambda+(1-c_{2M+1}K_{\rm err})\theta(p)+ \frac{c_{2M+1}}{K_1\sqrt{z_{2M+1}(1)}} \mathfrak l_M(\lambda,p,z_{2M+1}(1)).
\end{align*}
By Lemma \ref{lem:scale.identities.losses}, we have $c_{2M+1}\ge c_{-}$ and $c_{2M+1}K_{\rm err}\le \frac{1}{2}$. Therefore, the last denominator is bounded
from below by a positive constant multiple of
$$
K_1^{-1}z_{2M+1}(1)^{-1/2} \left[ \lambda+\theta(p)+ \mathfrak l_M(\lambda,p,z_{2M+1}(1))\right].
$$
Since $\widehat u^{(1)}_j(p)=\mathbf 1_{\{j=1\}}$, by the Parseval–Plancherel formula \eqref{eq:parseval.parameterized} and the upper bound in Lemma~\ref{lem:diagonal.integrals}, applied with the degree-zero convention, with $n=0$, $k=1$, $p=\emptyset$, $z=z_{2M+1}(1)$ and $m=M$, we get
\begin{align}
\label{eq:degree.one.upper.raw}
&\left\langle u^{(1)}, (\lambda-S+c_{2M+1}T_{2M+1})^{-1} u^{(1)} \right\rangle \le C\sqrt{z_{2M+1}(1)}\,{\rm UB}_{M}(\lambda,z_{2M+1}(1)).
\end{align}
Combining \eqref{eq:schur.trunc},
\eqref{eq:upper.inv}, and \eqref{eq:degree.one.upper.raw}, we obtain
\begin{equation}
\label{eq:res.upper2}
\llangle u,(\lambda-G)^{-1}u\rrangle \le C\sqrt{z_{2M+1}(1)} \,{\rm UB}_{M}(\lambda,z_{2M+1}(1)).
\end{equation}

We now prove the lower bound. By Proposition
\ref{prop:recur.bounds}, we have
$\Lambda_{2M}\le c_{2M}T_{2M}$. Thus
\begin{align*}
&\left\langle u^{(1)}, (\lambda-S+\Lambda_{2M})^{-1} u^{(1)} \right\rangle \ge \left\langle u^{(1)}, (\lambda-S+c_{2M}T_{2M})^{-1} u^{(1)} \right\rangle.
\end{align*}
Using \eqref{eq:def.s.even} and the fact $c_{2M}\le c_+$, we notice that there exists $C<\infty$ such that
\begin{align*}
&\lambda+\theta(p)+c_{2M}\mathfrak s_1^{(2M)}(\lambda,p) \le C\sqrt{z_{2M}(1)} \left[ \lambda+\theta(p) + \mathfrak u_{M-1}(\lambda,p,z_{2M}(1)) \right].
\end{align*}
Applying the Parseval–Plancherel formula \eqref{eq:parseval.parameterized} and using the lower bound in Lemma~\ref{lem:diagonal.integrals}, applied with the degree-zero convention, with $n=0$, $k=1$, $p=\emptyset$, $z=z_{2M}(1)$ and $m=M-1$, we obtain, for all sufficiently small $\lambda$,
\begin{align*}
&\left\langle
u^{(1)}, (\lambda-S+c_{2M}T_{2M})^{-1} u^{(1)} \right\rangle \ge \frac{C^{-1}}{\sqrt{z_{2M}(1)}} \left[ {\rm LB}_{M}(\lambda,z_{2M}(1)) - C{\rm LB}_{M}(1,z_{2M}(1)) - C
\right].
\end{align*}
Since ${\rm LB}_{M}(1,z_{2M}(1))\le\sqrt{L(1,z_{2M}(1))} \le C\sqrt{z_{2M}(1)}$, the second and third terms on the right-hand side of the above inequality are bounded below by a negative constant. Consequently,
\begin{align}
\label{eq:res.lower}
\llangle u,(\lambda-G)^{-1}u\rrangle &\ge \frac{C^{-1}}{\sqrt{z_{2M}(1)}} {\rm LB}_{M}(\lambda,z_{2M}(1)) -C.
\end{align}

Set $z_\lambda^+:=z_{2M+1}(1)$ and $z_\lambda^-:=z_{2M}(1)$.
By \eqref{eq:def.zmn} and \eqref{eq:def.Mlambda}, there exists $C<\infty$ such that
\begin{equation}
\label{eq:zlambda.sqrt.bound}
\sqrt{z_\lambda^\pm} \le C(\log \log(1+\lambda^{-1}))^{2+\eps}.
\end{equation}
Hence there exists sufficiently small $\lambda_0$ such that
$z_\lambda^\pm\le (\log(1+\lambda^{-1}))^{1/4}$ for all $0<\lambda<\lambda_0$.
Therefore
\begin{equation}
\label{eq:Llambda.z}
\log(1+\lambda^{-1}) \le L(\lambda,z_\lambda^\pm)\le 2\log(1+\lambda^{-1}) \quad\text{for all $0<\lambda<\lambda_0$.}
\end{equation}
Using \eqref{eq:Llambda.z} and the definition of $M$, we have
\begin{equation}
\label{eq:a.lambda}
\Big|M-\frac{1}{2}\log L(\lambda,z_\lambda^\pm)\Big|\le C.
\end{equation}
If $a=\frac12\log L(\lambda,z)$, then ${\rm LB}_M(\lambda,z)={\rm e}^a\,\mathbb P(N_a\le M)$, where $N_a$ is a Poisson random variable with parameter $a$. Since $|M-a|\le C$ by \eqref{eq:a.lambda}, the central limit theorem for Poisson random variables implies that there exists $C_0\in(1,\infty)$ such that, for all sufficiently small $\lambda$,
$$
C_0^{-1}\le\mathbb P(N_a\le M)\le C_0.
$$
Hence, for all sufficiently small $\lambda$,
\begin{equation}
\label{eq:LB.midpoint}
C^{-1}\sqrt{\log(1+\lambda^{-1})}\le {\rm LB}_{M}(\lambda,z_\lambda^\pm)\le C\sqrt{\log(1+\lambda^{-1})}.
\end{equation}
Using ${\rm UB}_{M}(\lambda,z)=L(\lambda,z)/{\rm LB}_{M}(\lambda,z)$ and
\eqref{eq:Llambda.z}, we also get
\begin{equation}
\label{eq:UB.midpoint}
{\rm UB}_{M}(\lambda,z_\lambda^\pm)\le C\sqrt{\log(1+\lambda^{-1})}.
\end{equation}

Combining \eqref{eq:res.upper2},
\eqref{eq:zlambda.sqrt.bound}, and \eqref{eq:UB.midpoint}, we obtain
\begin{equation}
\label{res.upper}
\llangle u,(\lambda-G)^{-1}u\rrangle \le C(\log \log(1+\lambda^{-1}))^{2+\eps}\sqrt{\log(1+\lambda^{-1})}.
\end{equation}
Combining \eqref{eq:res.lower},
\eqref{eq:zlambda.sqrt.bound}, and \eqref{eq:LB.midpoint}, we obtain
\begin{equation}
\label{eq:res.lower2}
\llangle u,(\lambda-G)^{-1}u\rrangle\ge C^{-1}(\log \log(1+\lambda^{-1}))^{-2-\eps} \sqrt{\log(1+\lambda^{-1})}.
\end{equation}

By \eqref{res.upper} and
\eqref{eq:res.lower2}, there exists a constant $C_{\eps}\in (1,\infty)$ such that for every $\lambda\in (0,e^{-2})$, 
\begin{align*}
C_\eps^{-1}(\log|\log\lambda|)^{-2-\eps}\sqrt{|\log\lambda|}&\le\llangle u,(\lambda-G)^{-1}u\rrangle\le C_\eps(\log|\log\lambda|)^{2+\eps}\sqrt{|\log\lambda|}.
\end{align*}
Recall from \eqref{eq:lambda2D.res} that $\lambda^2D(\lambda)=d+2d\llangle u,(\lambda-G)^{-1}u\rrangle$. Hence, for sufficiently large $C_{\eps}$,
$$
C_\eps^{-1}(\log|\log\lambda|)^{-2-\eps} \le \frac{\lambda^2D(\lambda)}{\sqrt{|\log\lambda|}}\le C_\eps(\log|\log\lambda|)^{2+\eps}.
$$
This completes the proof of Theorem \ref{thm.main}.
\end{proof}

\section*{Acknowledgment} This work was partially supported by Australian Research Council grant ARC DP230102209. Part of the work was conducted during the program “Stochastic Systems for Anomalous Diffusion” (July–December 2024), hosted by the Isaac Newton Institute for Mathematical Sciences (Cambridge, UK) funded by EPSRC grant EP/Z000580/1.
The author would like to thank Andrea Collevecchio and B\'alint T\'oth for their valuable feedback and literature suggestions, which helped improve the manuscript.
\bibliographystyle{amsplain}
\bibliography{refs}

\end{document}